\documentclass[12pt]{article}

\def\roff  {\mbox{\boldmath$\varepsilon$}}

\newcommand{\0}{\mathbf{0}}
\newcommand{\Id}{\mathbb I}
\newcommand{\Mbb}{\mathbb{M}}
\newcommand{\Cbb}{\mathbb{C}}
\newcommand{\Kbb}{\mathbb{K}}
\def\roff  {\mbox{\boldmath$\varepsilon$}}

\def\rank {\mathop{\rm rank}\nolimits}

\newcommand{\iu}{\mathbf{i}}
\usepackage{amsfonts,amssymb,amsmath}
\usepackage{algorithmic}
\usepackage{bbm}
\usepackage{bm}
\usepackage{caption}
\usepackage{verbatim}
\usepackage{colortbl}
\usepackage{tabulary}
\usepackage{etoolbox}
\usepackage[dvipsnames]{xcolor}
\usepackage{graphicx}
\usepackage{listings}
\usepackage{float}
\usepackage[ruled]{algorithm2e}

\usepackage{array}
\newcolumntype{?}{!{\vrule width 2pt}}
\makeatletter
\newcommand{\thickhline}{
	\noalign {\ifnum 0=`}\fi \hrule height 2pt
	\futurelet \reserved@a \@xhline}
\newlength\savedwidth

\usepackage{geometry}
\usepackage{amsthm}
\newtheorem{remark}{Remark}
\newtheorem{proposition}{Proposition}
\newtheorem{theorem}{Theorem}
\newtheorem{example}{Example}

\usepackage{subcaption}

\SetAlFnt{\small}
\SetAlCapFnt{\small}
\SetAlCapNameFnt{\small}
\SetAlCapHSkip{0pt}
\IncMargin{-\parindent}

\begin{document}

\title{
	An algorithm  for the complete solution of the quartic eigenvalue problem
	}
\author{ZLATKO DRMA\v{C},
IVANA \v{S}AIN GLIBI\'{C}}
\date{}
\maketitle

\begin{abstract}
Quartic eigenvalue problem $(\lambda^4 A + \lambda^3 B + \lambda^2C + \lambda D + E)x = \0$  naturally arises in a plethora of applications, e.g. when solving the Orr--\-Sommerfeld equation  in the stability analysis of  the {Poiseuille} flow, in  theoretical analysis and experimental design of  locally resonant phononic plates, modeling a robot  with electric motors in the joints, calibration of catadioptric vision system, or e.g. computation of the guided and leaky modes of a planar waveguide. 
{This paper proposes a new numerical method for the full solution (all eigenvalues and all left and right eigenvectors) that, starting with a suitable linearization, 
uses an initial, structure preserving, reduction designed to reveal and deflate certain number of zero and infinite eigenvalues before the final linearization is forwarded to the QZ algorithm. The backward error in the reduction phase is bounded column wise in each coefficient matrix, which is advantageous if the coefficient matrices are graded.}
Numerical examples 
show that the proposed algorithm is capable of computing the eigenpairs with small residuals, and that it is competitive with the available state of the art  methods.
\end{abstract}
%

%
%


\section{Introduction and preliminaries}
\noindent We propose a new method for  numerical solution of the \emph{quartic eigenvalue problem}
\begin{equation}\label{eq:QuarticEP}
(\lambda^4 A + \lambda^3 B + \lambda^2C + \lambda D + E)x = \0,
\end{equation}
where the coefficient matrices  $A,  B,  C, D, E\in \mathbb{C}^{n\times n}$ are assumed general, with no particular structure (such as symmetry, sparsity). We are interested in the full solution, i.e. computation of all eigenvalues with the corresponding (left and/or right) eigenvectors, and our ultimate goal is to provide a robust mathematical software that can be used in ever increasing number of applications in applied sciences and engineering.

The quartic eigenvalue problem naturally arises in solving the Orr -- Sommerfeld equation which appears in the hydrodynamic analysis of the stability of the {Poiseuille} flow by eliminating the pressure from the linearized Navier-Stokes equation.  
Other applications include e.g. theoretical analysis and experimental design of  locally resonant phononic plates \cite{yong-phononic-plate-2014}, finite element analysis of two dimensional phononic crystals \cite{doi:10.1063/1.4819209}, modeling a robot  with electric motors in the joints \cite{mehrmann-watkins2002}, computing deformation modes of thin-walled structures \cite{VIEIRA2014575},  or e.g. computation of the guided and leaky modes of a planar waveguide \cite{Stowell-waveguide-2009}, or solving  an optical waveguiding problem involving atomically thick 2D materials \cite{SONG2020109871}.
In these examples the matrix eigenvalue problem is the result of discretization of differential operators and thus (depending on the discretization method) the coefficient matrices are sparse and usually only some eigenvalues are needed -- those may be prescribed by specifying e.g.  a region of interest in the complex plane.  In such cases, methods for large sparse problems such as e.g. NLFEAST \cite{Gavin_2018}, \cite{SAKURAI2003119}, \cite{Shinnosuke-Yokota2013}, \cite{Chen2017} will find a subspace that contains eigenvectors of interest, and then Rayleigh--Ritz extraction uses the projected problem in which the Rayleigh quotients are medium size dense matrices.
Reliable solution of the projected problem is important both for the convergence of the iterations towards the wanted part of the spectrum (e.g. for robust implementation of locking and purging)  and for the accuracy of the computed solution.

{These examples illustrate the wide spectrum of important applications of the quartic eigenvalue problem, and justify, even demand, development of methods specialized for (\ref{eq:QuarticEP})}. {Yet, to the best of our knowledge, there is no published  custom-built solver with a supporting analysis that would provide certain level of confidence/guarantee that  is comparable e.g. to the currently available solvers of the quadratic eigenvalue problem such as \cite{Hammarling:QUADEIG}, \cite{KVADeig-arxiv}.} Instead,  (\ref{eq:QuarticEP})  is usually numerically solved by a standard linearization and deployment of the solvers such as \texttt{polyeig} in Matlab.  
On the other hand, numerical difficulties in solving nonlinear eigenvalue problems become nontrivial even in the simplest case of the polynomial quadratic problem, which is at the core of the theory and applications of mechanical systems. 
More carefully designed custom-made algorithm often proves much better than a generic solver -- an excellent example is the quadratic eigenvalue problem, where  the algorithms \texttt{quadeig} \cite{Hammarling:QUADEIG} and \texttt{kvadeig} \cite{KVADeig-arxiv} outperform  \texttt{polyeig}, in particular when the spectrum contains multiple  infinite eigenvalues. 


\subsection{{Backward stability and scaling}}\label{SS=Scaling}

Numerical algorithms for computing the eigenvalues and the corresponding right and left eigenvectors of a general regular matrix polynomial $P_k(\lambda)=\sum_{i=0}^k \lambda^i A_i$ usually consist of three main steps: \emph{(i)} linearization, i.e. definition of an equivalent linear generalized eigenvalue problem for a suitably constructed $k n\times kn$ pencil $\mathcal{A}-\lambda\mathcal{B}$;
\emph{(ii)}  computation of the eigenpairs of the linearization $\mathcal{A}-\lambda\mathcal{B}$,  using e.g. the QZ algorithm; \emph{(iii)} reconstruction of the eigenpairs of the original problem. Some algorithms include a preprocessor that transforms the linearization \emph{(i)} into a form that reveals some canonical (e.g. spectral) structure and that is in some numerical sense better input to a particular software implementation of the QZ algorithm in \emph{(ii)} (e.g. scaling). For a general framework of this scheme we refer the reader to \cite{van1979computation},  \cite{van1983eigenstructure}.
\vspace{-1mm}
\subsubsection{Weak and strong norm-wise backward stability}\label{SSS=weak-strong-bs}
In order to be used with confidence in  applications, the eigenvalues and eigenvectors computed in finite precision arithmetic are justified by proving that they are exact spectral elements of a nearby polynomial 
\vspace{-1mm}
$$ \widetilde{P}_k(\lambda)=\sum_{i=0}^k \lambda^i (A_i + \Delta A_i).
\vspace{-1mm}
$$  
If the sizes of the  perturbations (backward errors) $\Delta A_i$ are appropriately small (e.g. of the same order of magnitude as initial uncertainties in the coefficients $A_i$, as measured in a matrix norm) then the computation is usually deemed backward stable. 

In a numerical algorithm, the canonical structure revealing steps 
\cite{van1979computation},  \cite{van1983eigenstructure} and the QZ algorithm are based on unitary transformations, so that the entire process is backward stable -- the computed result corresponds exactly to a linear pencil
$$
\mathcal{A}+\Delta\mathcal{A}-\lambda(\mathcal{B}+\Delta\mathcal{B}),\;\; \|\Delta\mathcal{A}\|_F\leq \xi \|\mathcal{A}\|_F,\; \|\Delta\mathcal{B}\|_F\leq \xi \|\mathcal{B}\|_F,\;\;0\leq \xi \ll 1.
$$
We refer to this as \emph{strong norm-wise backward stability} of the solution of the linearized problem.
However, this statement must be carefully interpreted. The QZ algorithm is oblivious to the underlying structure of the linear pencil, and $\mathcal{A}+\Delta\mathcal{A}-\lambda(\mathcal{B}+\Delta\mathcal{B})$ most likely will not have the structure of the linearization of a matrix polynomial, and the backward stability cannot be stated in terms of the original polynomial eigenproblem. 

Relating the pencil $\mathcal{A}+\Delta\mathcal{A}-\lambda(\mathcal{B}+\Delta\mathcal{B})$ with a matrix polynomial close to $P_k(\lambda)$ requires an additional theoretical construction in the error analysis.  For large classes of linearizations, there is an equivalence transformation 
$$
\mathcal{A}+\Delta\mathcal{A}-\lambda(\mathcal{B}+\Delta\mathcal{B}) \longrightarrow 
(I+E) [\mathcal{A}+\Delta\mathcal{A}-\lambda(\mathcal{B}+\Delta\mathcal{B})](I+F),\;\;\|E\|_F\leq\epsilon_1,\;\;\|F\|_F\leq \epsilon_2,
$$
such that the new pencil is the linearization of $\widetilde{P}_k(\lambda)=\sum_{i=0}^k \lambda^i (A_i + \Delta A_i)$, where, under certain assumptions, 
\begin{equation}\label{eq:backw-array-wise}
\|\begin{pmatrix} \Delta A_0 & \Delta A_1 & \ldots & \Delta A_k \end{pmatrix}\|_F \leq \epsilon_3 \|\begin{pmatrix} A_0 & A_1 & \ldots & A_k\end{pmatrix}\|_F, 
\end{equation}
with some $0\leq \epsilon_1, \epsilon_2, \epsilon_3 \ll 1$ that depend on the roundoff unit, dimensions of the problem and algorithmic details. This form of \emph{weak norm-wise backward stability}  bounds each $\|\Delta A_i\|_F$ relative to the norm of the coefficients array $\begin{pmatrix} A_0 & A_1 & \ldots & A_k\end{pmatrix}$.
For detailed  in depth discussion see \cite[\S 4]{van1983eigenstructure}, \cite{JOHANSSON20131062},  \cite{DMYTRYSHYN2017213}, \cite{geom-matrix-pol-2020}.

Note that (\ref{eq:backw-array-wise}) may be difficult to interpret in an application where the coefficient matrices carry information of different physical nature (e.g. mass, damping and stiffness) expressed in appropriate physical units.  Unfortunately, except in some particular cases, (\ref{eq:backw-array-wise}) cannot be  strengthened into \emph{strong norm-wise backward stability}\footnote{An anonymous referee suggested the term \emph{coefficient-wise backward stability}.} estimate
\begin{equation}\label{eq:backerr-strong}
 \| \Delta A_i\|_F \leq \epsilon \|A_i\|_F,\;\;i=0,\ldots, k.   
\end{equation}
Construction of an algorithm with the backward error (\ref{eq:backerr-strong}) may not be feasible without the framework of mixed error analysis, see \cite{Mastronardi-VanDooren-quad-back-err-ETNA} where this is shown for the case $n=1$, $k=2$. Constructing an algorithm that has guaranteed (provable)  strong norm-wise backward/mixed stability is a challenging open problem.

\subsubsection{Backward stability of individual eigenpairs. Residual}\label{SSS=residuals}
Another way to measure the quality of an approximate eigenpair (an eigenvalue with a corresponding right or left eigenvector) \emph{a posteriori} is through the residual, which we discuss next, in terms of our original problem (\ref{eq:QuarticEP}). If $(\lambda, x)$ is a computed eigenpair with the right eigenvector\footnote{For the sake of brevity, here we omit an analogous discussion with the left eigenvector.} $x$, then the 
minimal size of backward error of the type (\ref{eq:backerr-strong}) that makes $(\lambda, x)$ an exact eigenpair {of a backward perturbed quartic eigenvalue problem}, i.e. 
\begin{multline*}
\min\{ \epsilon  : 
(\lambda^4 (A+\Delta A)+\lambda^3 (B+\Delta B)+\lambda^2(C+\Delta C)+\lambda(D+\Delta D)+(E+\Delta E))x=\0, \\
\|\Delta A\|_F\! \leq \epsilon\|A\|_F,  \|\Delta B\|_F\!\leq\epsilon \|B\|_F,
\|\Delta C\|_F\!\leq\epsilon\|C\|_F,
\|\Delta D\|_F\!\leq\epsilon\|D\|_F, \|\Delta E\|_F\!\leq\epsilon\|E\|_F
\}\!,
\end{multline*}
can be explicitly computed as the normalized residual \cite[\S 2.2]{tisseur2000backward}
\begin{align}\label{eq:back-err-0}
\begin{split}
\eta({\lambda}, x)&=\frac{\| (\lambda^4 A + \lambda^3 B + \lambda^2 C + \lambda D + E )x\|_2}{(|\lambda|^4 \|A\|_2+|\lambda|^3\|B\|_2 + |\lambda|^2 \|C\|_2 + |\lambda| \|D\|_2 + \|E\|_2)\|x\|_2}, \; \lambda\neq \infty;\\
\eta(\infty,x)& = \frac{\|Ax\|_2}{\|A\|_2\|x\|_2}.
\end{split}
\end{align}
The key difficulty is that an eigenpair obtained by this procedure may have large residual (norm-wise backward error (\ref{eq:back-err-0})), although the norm-wise backward error for the eigenpair (with the same $\lambda$) of the corresponding linearization\footnote{The norm-wise backward error of $\lambda$ as an approximate eigenvalue of the linearization is computed analogously to (\ref{eq:back-err-0}).} is acceptably small.
This phenomenon is further analyzed in \cite{higham2007backward}, and it is proven that this kind of variation in the backward errors is due to the fact that the norms of the coefficient matrices of the original problem are not equilibrated. As a result, the backward stability of the linearization is not inherited by the transformation to the original polynomial. The problem can be alleviated by parameter scaling, which we review next.
%
%
\subsubsection{Parameter scaling}\label{SSS=param-scaling}
Parameter scaling is a useful, powerful, albeit not omnipotent,  tool for stabilizing polynomial eigensolvers; the techniques vary from simple heuristics to sophisticated concepts from tropical polynomial algebra. Here we briefly review the scaling used in this paper; other scalings can be easily incorporated.
  
 Write $\lambda = \gamma \nu$, where $\gamma>0$ is parameter to be determined. Using an additional free parameter $\theta >0$, define the scaled quartic polynomial as
 \vspace{-1mm}
$$
\theta E + \nu (\gamma \theta D) + \nu^2 (\gamma^2\theta C) + \nu^3 (\gamma^3\theta B) + \nu^4 (\gamma^4\theta A) \equiv \widehat{E} + \nu \widehat{D} + \nu^2\widehat{C} + \nu^3\widehat{B} + \nu^4\widehat{A}.
$$
%
%
%
The parameters $\gamma$ and $\theta$ are defined so that the norms of the new coefficient matrices do not vary much, and are close to a traget value. This can be done by adapting the Fan, Lin and Van Dooren's scaling \cite{fan2004normwise}.
 For {$\gamma$}, we choose
${\gamma} = \sqrt[4]{\frac{\|E\|_F}{\|A\|_F}}$,
which is the optimal $\gamma$ for minimizing the factor
${\max (1,\|\widehat{A}\|_F,\|\widehat{B}\|_F,\|\widehat{C}\|_F, \|\widehat{D}\|_F,\|\widehat{E}\|_F)^2}/{\min (\|\widehat{E}\|_F,\|\widehat{A}\|_F)}$
in the backward error ratio bounds  \cite{betcke2008optimal}, and for $\theta$, we choose, following \cite{campos2016parallel}, 
${\theta} = {4}/{(\|E\|_F + \gamma\| D\|_F + \gamma^2\| C\|_F +\gamma^{3}\|B\|_F)}.$

Although simple and easy to implement in any polynomial eigensolver, scaling can achieve good results in controlling the growth factor of the backward error. For instance, \cite{zeng-su-quadeig-bs} showed that the quadratic eigenvalue problem can be solved with small backward errors (\ref{eq:back-err-0}) in all eigenpairs by carefully examining  six linearizations. Detailed analysis showed that backward stability in all eigenpairs was  achieved using two linearizations.  

Note that in this interpretation of backward stability, the optimal backward errors (\ref{eq:back-err-0}) are constructed separately for each eigenpair, which is different from the backward stability discussed in \S \ref{SSS=weak-strong-bs}. 

\subsubsection{Graded matrices}\label{SSS=graded-m}
{We should keep in mind that, in addition to different magnitudes of the norms $\|A\|_F, \|B\|_F,\ldots$, $\|E\|_F$, the entries inside each coefficient matrix may be on different scales of magnitude, where some small entries may be important parameters. If, for instance, $A$ has  graded columns whose norms vary over several orders of magnitude, and $\Delta A$ is small perturbation that satisfies $\|\Delta A\|_F\leq \epsilon \|A\|_F$ with a small $\epsilon >0$, then some small columns of $A$ may be completely wiped out by $\Delta A$. (In an engineering application the columns of the coefficient matrices may be scaled to properly interpret the norm of the eigenvector $x$ whose components are of different physical nature.) Parameter scaling cannot counteract differently scaled columns in a particular coefficient matrix.
}

\vspace{-1mm}
\begin{remark}\label{REM=Balancing}
In addition to parameter scaling, we can use diagonal scaling matrices $\varDelta_{\ell}$ and $\varDelta_{r}$, for scaling all coefficients from the left with $\varDelta_{\ell}$ and from the right with $\varDelta_r$. {The goal is to equilibrate the absolute values of all matrix entries.} These scaling matrices can be computed by an extension of the scheme described in \cite[\S 4.2]{KVADeig-arxiv}.
\end{remark}

\vspace{-4mm}
\subsection{A quandary about the infinite eigenvalues}\label{SS=Quandary-infinite-eigs}

The presence of infinite eigenvalues, indicated by the rank deficiency of $A$,  may cause difficulties in the QZ algorithm, which is usually deployed for solving the linearized problem; infinite eigenvalues may not be identified correctly, they may have negative impact on the accuracy of the computed finite eigenvalues, see e.g. \cite[Example 2]{van1979computation}. It is then advantageous to remove infinite eigenvalues by a deflation and proceed with a problem of smaller dimension, with only finite eigenvalues. This framework, introduced in \cite{Hammarling:QUADEIG}, proved much better than direct solution of the linearized problem.

In some cases, certain number of infinite eigenvalues of (\ref{eq:QuarticEP}) can be identified and removed already during the problem formulation. An illustrative example
is given in the eigenvalue problem for the channel and Blasius boundary layer in semi-infinite domain \cite{cheb-orr-somm-1990}. The Orr--Sommerfel differential equation is  discretized using the Chebyshev collocation matrix method, and  the boundary conditions are imposed in $E$;  the remaining matrix coefficients $A$, $B$, $C$, $D$ have the corresponding last four rows equal to zero. In the case of linearly independent boundary conditions, by a clever column permutation, four infinite eigenvalues can be separated and deflated, see \cite{cheb-orr-somm-1990} for technical details. 

The structure of the infinite eigenvalue (the number and the dimensions of blocks in the Kronecker Canonical Form (KCF)) cannot be inferred by only inspecting the rank of $A$. Rank deficiency in $A$ reveals only certain number of infinite eigenvalues and further steps are necessary to either confirm that there are no more infinite eigenvalues or to reveal more  blocks   carrying $\lambda=\infty$ in the KCF. {For more details see \cite{van1979computation} and \cite{van1983eigenstructure}. }

Removals of infinite and zero eigenvalues  involve decisions on the numerical ranks of some intermediate matrices that have been contaminated by the roundoff noise from the previous steps. If the data is not well scaled, and if the computation cannot be interpreted as backward stable in terms of the original coefficients {that may have an initial uncertainty from the very problem formulation}, then there may be quite a few spurious eigenvalues with large absolute values.   The backward stability of the beginning steps that carry the critical responsibility of removing {as many as possible} infinite eigenvalues must be as much as possible in terms of the initial coefficient matrices, and it has to be as much as possible structured, e.g. column-wise small (backward error in each column small relative to that column's norm) instead of only small in matrix norm. 

{
As we pointed out in \cite{KVADeig-arxiv}, the goal of the pre-processing is to remove many zero and infinite eigenvalues before calling the QZ algorithm, and to use QZ software optimized for a given computing machinery. The reason is that handling infinities numerically in QZ is a delicate issue with many fine details \cite{watkins2000performance}, and the development of optimized software  often uses techniques, such as e.g. block-oriented formulations and parallelization, that accept speed-accuracy trade-offs.   \\
\indent Another possibility to deal with infinite eigenvalues of matrix polynomials, pursued  in \cite{VANBAREL2018186},  \cite{tisseur2020min}, is, after a suitable linearization and scaling, to modify software implementation of the QZ algorithm by lowering the original threshold  for setting small numbers to zero in the part of the algorithm that can create infinite eigenvalues. This change of a critical parameter was compensated by increasing the maximal allowed number of iterations.}  {This method performed well as measured by the residual of refined eigenvectors $\eta_P(\lambda)=\min_{x\neq \0}\eta(\lambda,x)$.\\ \indent Although the goal to safely deflate eigenvalues that may be difficult for QZ iterations is the same,  our approach of deflation, based as much as possible on initial data, is conceptually different, as we describe next.
}

{
\section{A new approach to the quartic eigenvalue problem}}
In our recent paper \cite{KVADeig-arxiv}, we built upon the \texttt{quadeig} algorithm of \cite{Hammarling:QUADEIG} and \cite{van1979computation},  \cite{van1983eigenstructure} and constructed an algorithm (designated as \texttt{kvadeig}) for the quadratic eigenvalue problem that makes several reduction steps toward the KCF. One of distinctive features of that reduction is that the backward error in the coefficient matrices is bounded on a finer-scale, e.g. column-wise. Although such a column-wise error bound does not extend to the entire algorithm (the QZ algorithm enjoys only the norm-wise bound), it may be of critical importance in the beginning steps when  decisions abut the zero and infinite eigenvalues have to be made, {in particular if the matrix coefficients are graded, as discussed in \S \ref{SSS=graded-m}.}  {Since this issue, tackled in  \texttt{kvadeig}, is separate from parameter scaling, the approach introduced in \texttt{kvadeig} can benefit from any good scaling, so that it can be combined e.g. with the strategy introduced in \cite{zeng-su-quadeig-bs}. }\\
\indent In this paper we extend the techniques of \texttt{quadeig} and \texttt{kvadeig} to the quartic eigenvalue problem (\ref{eq:QuarticEP}). {A direct connection of (\ref{eq:QuarticEP}) with the quadratic problem is quadratification. In \S \ref{SS=quadratification}, we briefly review quadratification by companion forms of grade 2, and then we discuss practical advantages and shortcomings of this approach to the quartic eigenvalue problem. Then, in \S \ref{SS=contributions} we present the main idea of the paper -- a linearization based on the quadratification provides a two-level structure that allows for a generalization of the scheme used in \texttt{kvadeig}.}\\

\subsection{Quadratification}\label{SS=quadratification}
 Both \texttt{quadeig} and \texttt{kvadeig} outperform the general solvers such as e.g. \texttt{polyeig} from Matlab. In addition, \cite{zeng-su-quadeig-bs}  provides a backward stable algorithm (in the sense of residuals reviewed in \S \ref{SSS=residuals}, albeit at double cost) whose semi-tropical scaling can be used in \texttt{quadeig/kvadeig} as well. {Hence, good quadratic solvers supported by numerical analysis are available.}
If one has these quadratic solvers implemented in a reliable software, then it makes sense to use \textit{quadratification} \cite{de2014spectral} to  reduce the quartic problem to a quadratic one and use the off the shelf quadratic code. 

To that end, define matrix polynomials
$B_1(\lambda) = \lambda^2C + \lambda D + E,\;\;
B_2(\lambda) = \lambda^2 A + \lambda B$.
The first  and the second companion form of grade $2$ are then defined, respectively,  as
\begin{align*}
C^2_1(\lambda) \!=\!\! \begin{pmatrix}
B_2(\lambda) & B_1(\lambda) \\
-\mathbb{I}_n & \lambda^2 \mathbb{I}_n
\end{pmatrix} \!\!=\!\! \begin{pmatrix}
\lambda^2 A + \lambda B & \lambda^2C + \lambda D + E\\
-\mathbb{I}_n & \lambda^2 \mathbb{I}_n
\end{pmatrix}\! 
\!= \!\lambda^2 \!\begin{pmatrix}
A & C\\
\mathbf{0} & \mathbb{I}_n
\end{pmatrix}\!\! +\! \lambda \!\begin{pmatrix}
B & D\\
\mathbf{0} & \mathbf{0}
\end{pmatrix} \!\!+\! \begin{pmatrix}
\mathbf{0} & E\\
-\mathbb{I}_n & \mathbf{0}
\end{pmatrix}\!.
\end{align*}
\begin{equation}\label{eq:SecondComapnionFormGrade2}
C^2_2(\lambda) = \begin{pmatrix}
B_2(\lambda) & -\mathbb{I}_n \\
B_1(\lambda) & \lambda^2 \mathbb{I}_n
\end{pmatrix} 
= \lambda^2\begin{pmatrix}
A & \mathbf{0}\\
C & \mathbb{I}_n
\end{pmatrix} + \lambda \begin{pmatrix}
B & \mathbf{0}\\
D & \mathbf{0}
\end{pmatrix} + \begin{pmatrix}
\mathbf{0} & -\mathbb{I}_n\\
E & \mathbf{0}
\end{pmatrix} 
= \lambda^2 \mathbb{M} + \lambda \mathbb{C} + \mathbb{K}.
\end{equation}
Both $C^2_1(\lambda)$ and $C^2_2(\lambda)$ are strong quadratifications,  
see  [Theorem 5.3, Theorem 5.4]\cite{de2014spectral},  {\cite{DETERAN2016344}.}

\subsubsection{Why quadratification alone is not enough?}\label{SSS=quadr-not-enough}
While solving the eigenvalue problem for (\ref{eq:SecondComapnionFormGrade2}) by a reliable quadratic eigensolver is a viable approach, its straightforward implementation has a significant limitation in light of the discussions in \S \ref{SS=Scaling} and \S \ref  {SS=Quandary-infinite-eigs}. Namely, just as the QZ algorithm is oblivious to the structure of the linearization, the quadratic eigensolver will be oblivious to the fact that the quadratic pencil represents a quadratification of the quartic problem and that the matrices $\mathbb{M}$, $\mathbb{C}$, $\mathbb{K}$ in (\ref{eq:SecondComapnionFormGrade2}) have block structure defined by the matrices of the original problem. As a result, in the preprocessing phase the  algorithm will not make the critical decisions (such as numerical rank revealing) based on the original coefficients of the quartic problem (\ref{eq:QuarticEP}). 

On the other hand, with a suitable choice of the quadratification and its linearization, we can generalize the framework of \texttt{kvadeig} by zooming into the block structure of the matrices constructed in the quadratification, and thus work with the original coefficients. This is the key idea in this work, and in the rest of this section we set the scene and present the structure of the paper. 
We will use the second companion form of grade $2$  because its structure  is compatible with the deflation scheme of \texttt{kvadeig}.

\subsection{A generalization of the  deflation scheme from \texttt{KVADEIG}}\label{SS=contributions}
The starting point of the development of the proposed algorithm is the quadratic polynomial (\ref{eq:SecondComapnionFormGrade2}). 
It 
can be further linearized using e.g. the second companion form. In that case, the final matrix pencil of size $4n\times 4n$, that represents a linearization of the quartic problem \ref{eq:QuarticEP}, reads
\vspace{-1mm}
\begin{equation}\label{eq:FinalLinearization}
\mathbb{A} - \lambda \mathbb{B} \!=\!\begin{pmatrix}
\Cbb & -\Id_{2n}\\
\Kbb & \0_{2n}
\end{pmatrix} - \lambda \!\begin{pmatrix}
-\Mbb & \0_{2n} \\
\0_{2n} & -\Id_{2n}
\end{pmatrix} \!=\!\left(\!\begin{array}{c | c ? c | c}
B & \0_n & -\Id_n & \0_n \\\hline
D & \0_n & \0_n & -\Id_n \\ \thickhline
\0_n & -\Id_n & \0_n & \0_n \\ \hline
E & \0_n & \0_n & \0_n 
\end{array}\!\right)- \lambda \!\left(\!\begin{array}{c | c ? c | c}
-A & \0_n & \0_n & \0_n\\ \hline
-C & -\Id_n & \0_n & \0_n\\ \thickhline
\0_n & \0_n & -\Id_n & \0_n \\ \hline
\0_n & \0_n & \0_n & -\Id_n
\end{array}\!\right)\!\!.
\end{equation}
Notice that \eqref{eq:FinalLinearization} is actually a block Kronecker linearization up to simultaneous interchanges of block rows 2 and 3, and block columns 2 and 3;  hence backward error analysis of type (\ref{eq:backw-array-wise}) applies, see \cite[Theorem 5.22, Corollary 5.24]{dopico2018block}. 
However, we find the interpretation via quadratification \eqref{eq:SecondComapnionFormGrade2} more intuitive and more natural for generalization of the scheme from \cite{KVADeig-arxiv}.

Now, we can follow the structure of \texttt{quadeig}/\texttt{kvadeig}, attempting to deflate the infinite eigenvalues of $\lambda^2 \Mbb + \lambda \Cbb + \Kbb$. Even if that is expected to perform better than a straightforward companion type linearization followed by \texttt{polyeig}, it is not the best one can do, see \S \ref{SSS=quadr-not-enough}. Instead, the goal is to implement the beginning critical steps, including deflations,  with small (hopefully to some extent structured) backward error in the original coefficient matrices $A$, $B$, $C$, $D$, $E$.  Therefore, on the global level, we follow the strategy of \texttt{kvadeig} \cite[Algorithm 3.1]{KVADeig-arxiv} to bring (\ref{eq:FinalLinearization}) to an upper triangular Kronecker Canonical Form (KCF), but the elementary steps are rewritten in terms of the original matrices whenever feasible. We will keep the two-level partitioning throughout the paper, and try to use transformations that respect the block structure as much a possible and,  if deflation is needed, the structure of the linearization should be considered when defining the transformation matrices.

\subsubsection{Outline of the paper} The new algorithm, designated as \texttt{kvarteig}, is described in detail in \S \ref{S=ALG}. 
{Recovering the eigenvectors of (\ref{eq:QuarticEP}) from those of (\ref{eq:FinalLinearization}) is discussed in detail in \S \ref{S=Vectors}, where we propose using least squares regularization, and suggest algorithmic details for an efficient software implementation.}
In \S \ref{S=BackError} we provide detailed backward error analysis of the first two steps that are critical for removing zero and infinite eigenvalues. We clearly identify moments in the algorithm where scaling of the data plays the key role in keeping the backward error in the initial data small. 
The numerical experiments, presented in \S \ref{S=KVATREIG-NUMEX}, show the advantage of the new framework, both of \texttt{kvarteig} and \texttt{kvadeig} (applied to the quadratification).
{The strong point of the proposed algorithm is the deflation process. The biggest differences in the results, as compared to other algorithms, can be seen in the element-wise backward errors and, in particular, in the examples with zero and/or infinite eigenvalues. 
Altogether, the numerical examples clearly demonstrate the importance of both  scaling (including balancing) and deflation in the pre-processing phase.}

The material of this paper should be considered as the second part of \cite{KVADeig-arxiv}, and numerical results in \S \ref{S=KVATREIG-NUMEX} once more illustrate the power of the approach introduced in \texttt{quadeig} and \texttt{kvadeig}.

\section{The \texttt{kvarteig} algorithm}\label{S=ALG}
{
We now describe the main ideas of the proposed procedure for deflating  infinite and/or zero eigenvalues.
As outlined in \S \ref{SS=contributions}, our plan is to adapt the deflation scheme from \texttt{quadeig}/\texttt{kvadeig} to the 
linearization (\ref{eq:FinalLinearization}).  Although the structure of the proposed algorithm is inspired by our recent quadratic eigensolver and its connection to the problem (\ref{eq:QuarticEP}) via quadratification, we stress again that the new algorithm is not a composition of quadratification and quadratic eigensolver.  We recall our discussion in \S \ref{SSS=quadr-not-enough} that applying a robust quadratic solver to the quadratification blindly (i.e. by ignoring the origin of the quadratic problem) is not satisfactory.  

The first immediate problem is revealing the numerical rank of the coefficient matrices. In \S \ref{SS=num-rank-block}, we briefly review this issue and use it to illustrate the two-level approach to the linearization (\ref{eq:FinalLinearization}) -- at the level of the $2\times 2$ partition, the algorithm mimics the structure of \texttt{kvadeig}, but all operations are adapted  to the $4\times 4$ block structure of the linearization and then, in \S \ref{SS=decision-tree}, further tailored for the quartic problem.}  

\subsection{Numerical rank and block-structure}\label{SS=num-rank-block}

{Revealing infinite and zero eigenvalues in presence of perturbations is a delicate task because it depends on the numerical ranks \cite{Golub-Klema-Stewart-Numerical_rank} of 
matrices that are either initial coefficients (possibly polluted by noise) or intermediate results in finite precision computation. An additional difficulty is the underlying structure of the involved matrices, that should preferably be preserved in a backward stability interpretation of the computed results. This is one of the reasons why applying \texttt{quadeig/kvadeig} directly to a quadratification is numerically not optimal.}

Namely, applying a rank revealing decomposition (such as the SVD or the pivoted QR factorization) to $\Mbb$ would mean looking for a small perturbation $\delta\Mbb$ such that $\Mbb+\delta\Mbb$ has lower rank that cannot be further reduced by a small perturbation.  Such a construction does not respect the block structure of $\Mbb$, and better way is to think at this step in terms of the numerical rank with constrained perturbation. If $J$ denotes the first $n$ columns of $\Id_{2n}$ then the allowed perturbation might be $\delta\Mbb = J\delta A J^T$ with an $n\times n$ $\delta A$. Similarly, the numerical rank of $\Kbb$ will be determined under the constraint that only $\Kbb(n+1:2n,1:n)=E$ is allowed to change. For a systematic treatment of the general case using the generalized SVD, see \cite{Zha-triplets-RSVD}. 

{Since we have the natural block structure, we can formulate the rank revealing steps directly, in terms of the original coefficients. This defines the first step of the  procedure whose details are explained in \S \ref{SS=decision-tree}. } 

Let $r_A = \rank(A)$, $r_E = \rank(E)$ and let 
\begin{equation}\label{eq:RankRevealingQRAE}
A\Pi_A = Q_AR_A, \;\; R_A = \begin{pmatrix}
\widehat{R}_A\\
\mathbf{0}_{n-r_A,n}
\end{pmatrix},  \;\;
E\Pi_E = Q_ER_E, \;\; R_E = \begin{pmatrix}
\widehat{R}_E\\
\mathbf{0}_{n-r_E,n}
\end{pmatrix},
\end{equation}
be the rank revealing QR factorizations for $A$ and $E$, computed as in \cite{bus-gol-65}, \cite{drmac-bujanovic-2008}. 
%
Note that (\ref{eq:RankRevealingQRAE}) yields a structure preserving rank revealing decomposition of the matrix $\mathbb{M}=\left( \begin{smallmatrix} A & \0 \cr C & \Id_n \end{smallmatrix}\right)$ as
\begin{equation}\label{eq:MRRD}
	\mathbb{M}\Pi_M = Q_MR_M,\;\;Q_M = \left(\begin{array}{c|c}
	\mathbf{0} & Q_A \\ \hline
	\mathbb{I}_n & \mathbf{0}
	\end{array}\right),\;\; \Pi_M = \left(\begin{array}{c|c}
	\mathbf{0} & {\Pi_A} \\ \hline
	\mathbb{I}_n & \mathbf{0}
	\end{array}\right),\;\;R_M = \left(\begin{array}{c|c}
	\mathbb{I}_n & C{\Pi_A} \\ \hline
	\mathbf{0} & R_A
	\end{array}\right).
\end{equation}
{The truncation of $R_M$ is done by truncating $R_A$, and the truncation can be pushed back into a backward perturbation of $A$; see \cite[\S 2.1, \S 2.3]{KVADeig-arxiv}.}

Similarly, the rank revealing factorization of the matrix $\mathbb{K}=\left(\begin{smallmatrix} \0 & -\Id_n\cr E & \0\end{smallmatrix}\right)$ is
\begin{equation}\label{KRRD}
	\mathbb{K}\Pi_K = Q_KR_K,\;\;Q_K = \left(\begin{array}{c|c}
	\Id_n & \0 \\ \hline
	\0 & Q_E
	\end{array}\right), \;\; \Pi_K = \left(\begin{array}{c|c}
	\0 & \Pi_E \\ \hline
	\Id_n & \0
	\end{array}\right),\;\;R_K = \left(\begin{array}{c|c}
	-\mathbb{I}_n & \0 \\ \hline
	\mathbf{0} & R_E
	\end{array}\right).
\end{equation}
Notice that the permutation of the column blocks only ensures that the matrix $R_K$ is upper triangular. If this structure is not important for the process, we can skip the permutation step and just make the following transformation
\begin{equation}
\begin{pmatrix}
\Id_n & \0\\
\0 & Q^*_E
\end{pmatrix}\Kbb \begin{pmatrix}
\Pi_E &\0\\
\0 & \Id_n
\end{pmatrix} = \begin{pmatrix}
\0 & -\Id_n \\
 R_E & \0
\end{pmatrix}.
\end{equation}
\begin{remark}\label{REM-drop-off}
{To determine the numerical rank using the rank revealing QR factorization, we use the thresholding strategies as in \cite[\S 2.3.1]{KVADeig-arxiv}. For a softer thresholding we look for a drop-off of absolute values of two consecutive diagonal entries in the upper triangular form. {In general, determination of the numerical rank (thresholding strategy and thresholds for truncating the triangular factor) should take into account the size and the structure of the initial uncertainty in the data. Such an additional information is application specific.}
}
\end{remark}

\subsection{The decision tree of \texttt{kvarteig}}\label{SS=decision-tree}
The algorithm is designed to remove zero eigenvalues; the infinities are removed by switching to the reversed pencil.
{
Similarly as in\footnote{Here, some familiarity with the reduction/deflation in the \texttt{kvadeig} algorithm is helpful for understanding the details of \texttt{kvarteig}. } \cite{KVADeig-arxiv}, the deflation process is an adaptation of the algorithm by 
\cite{van1979computation}
for computing the structure of the eigenvalues $0$ and $\infty$. The first two steps are modified using the structure of the linearization  (\ref{eq:FinalLinearization}), and for possible additional steps the algorithm proceeds with the rank revealing QR factorizations and carefully implemented URV decompositions.}

As in the \texttt{kvadeig}, there are three main {cases}: both $A$ and $E$ regular; only one of $A$ and $E$ is singular; and both $A$ and $E$ are singular.

\subsubsection{\underline{Both matrices $A$ and $E$ regular}}
If both matrices $A$ and $E$ are regular, we can use the factorization (\ref{eq:MRRD}) to reduce the matrix $\mathbb{B}$ from (\ref{eq:FinalLinearization}) to upper triangular form, since this is already the first step of the QZ algorithm. 
\begin{eqnarray}
&& \begin{pmatrix} Q_M^* & \0 \cr \0 & \Id_{2n} \end{pmatrix} \left\{ \begin{pmatrix} \mathbb{C} & -\Id_{2n} \cr \mathbb{K} & \0\end{pmatrix} -\lambda \begin{pmatrix} -\mathbb{M} & \0 \cr \0 & -\Id_{2n} \end{pmatrix}\right\}\begin{pmatrix} \Pi_M & \0 \cr \0 & \Id_{2n} \end{pmatrix}\nonumber \\
&=& 
{\footnotesize\left(\begin{array}{c ? c} 
\begin{array}{c|c}
\0 & D{\Pi_A}\cr \hline
\0 & Q^*_AB{\Pi_A}
\end{array} & \begin{array}{c| c}
\mathbf{0} & -\Id_{n}\cr\hline 
-Q^*_A & \mathbf{0}
\end{array} \cr \thickhline
\begin{array}{c|c}
-\Id_n & \mathbf{0}\cr\hline
\mathbf{0} & E{\Pi_A}
\end{array} & \huge\mathbf{0}_{2n}
\end{array}\right) - \lambda \left(\begin{array}{c?c}
\begin{array}{c|c}
-\Id_n & -C {{\Pi_A}}\cr \hline
\0 & -R_A
\end{array} & \0_{2n} \cr \thickhline
\0_{2n} & -\Id_{2n}
\end{array}\right)}.\label{eq:ABRegularCase}
\end{eqnarray}
The rest of the computation depends on the QZ algorithm. Note that the special structure of the pencil (\ref{eq:ABRegularCase}) can be exploited for designing a more efficient Hessenberg-triangular decomposition. This is a separate issue that we will not tackle in this work.
\subsubsection{\underline{Only one matrix is singular}}\label{SSS=one-singular}
Assume first that $E$ is singular, $r_E<n$, and thus there are at least $n-r_E$ zero eigenvalues which can be deflated. If our setup is to remove only the block of zero eigenvalues that is revealed by the null space of $E$, then we can achieve that and, at the same time, transform the matrix $\mathbb{B}$ to upper triangular form by the equivalence transformation 
\begin{eqnarray}
&& \begin{pmatrix} Q_M^* & \0 \cr \0 & Q_K^* \end{pmatrix} \left\{ \begin{pmatrix} \mathbb{C} & -\Id_{2n} \cr \mathbb{K} & \0\end{pmatrix} -\lambda \begin{pmatrix} -\mathbb{M} & \0 \cr \0 & -{\Id_{2n}} \end{pmatrix}\right\}\begin{pmatrix} \Pi_M & \0 \cr \0 & Q_K \end{pmatrix} \nonumber\\
&=& 
{\small\left(\begin{array}{c?c}
\begin{array}{c|c}
\0 & D{\Pi_A}\cr \hline
\0 & Q^*_AB{\Pi_A}
\end{array} & \begin{array}{cc}
\0 & -Q_E\\
-Q^*_A & \0
\end{array} \cr \thickhline
\begin{array}{c|c}
-\Id_n & \0 \cr \hline
\0 & \widehat{R}_E {\Pi^*_E}{\Pi_A} \cr \0 & \0
\end{array} & \0_{2n}
\end{array}\right) - \lambda \left(\begin{array}{c?c}
\begin{array}{c|c}
-\Id_n & -C {\Pi_A}\cr \hline
\0 & -R_A
\end{array} & \0_{2n} \cr \thickhline
\0_{2n} & -\Id_{2n}
\end{array}\right)}.\label{eq:ABOneSingularOneJordanBlock}
\end{eqnarray}
The $n-r_E$ zero eigenvalues are now deflated implicitly by working with the leading $(3n+r_E)\times (3n+r_E)$ sub-pencil of (\ref{eq:ABOneSingularOneJordanBlock}).
If we want to check for the existence of further blocks corresponding to $\lambda=0$, then it is convenient to use the following transformation:
\begin{eqnarray}\label{eq:FirstTransformJordanTest}
&& \begin{pmatrix} Q_K^* & \0 \cr \0 & Q_K^* \end{pmatrix} \left\{ \begin{pmatrix} \mathbb{C} & -\Id_{2n} \cr \mathbb{K} & \0\end{pmatrix} -\lambda \begin{pmatrix} -\mathbb{M} & \0 \cr \0 & -\Id_{2n} \end{pmatrix}\right\}\begin{pmatrix} \Id_{2n} & \0 \cr \0 & Q_K \end{pmatrix}\nonumber \\
&=& 
{\left(\begin{array}{c?c}
\begin{array}{c|c}
B & \0\cr \hline
Q^*_ED & \0
\end{array} & \begin{array}{c}
-\Id_{2n}
\end{array} \cr \thickhline
\begin{array}{c|c}
0 & -\Id_n \cr \hline
\widehat{R}_E{\Pi^*_E} & \0 \cr
\0 & \0
\end{array} & \0_{2n}
\end{array}\right) - \lambda \left(\begin{array}{c?c}
\begin{array}{c|c}
-A & \0\cr \hline
-Q^*_EC & -Q^*_E
\end{array} & \0_{2n} \cr \thickhline
\0_{2n} & -\Id_{2n}
\end{array}\right)}.\label{eq:ABOneSingular}
\end{eqnarray}
The deflated pencil of order $3n+r_E$ reads
\begin{equation}\label{eq:EsingularTruncatedPencil}
	\mathbb{A}_{22} - \lambda \mathbb{B}_{22} = {\small\left(\begin{array}{c?c}
	\begin{array}{c|c}
	B & \0\cr \hline
	Q^*_{E,1}D & \0\\ \hline
	Q^*_{E,2}D & \0
	\end{array} & \begin{array}{c| c}
	-\Id_{n} & \\ \hline 
	 & -\Id_{r_E}\\ \hline
	\0 & \0
	\end{array} \cr \thickhline
	\begin{array}{c|c}
	\0 & -\Id_n \cr \hline
	\widehat{R}_E{\Pi^*_E} & \0 \cr
	\end{array} & \0_{n+r_E}
	\end{array}\right) - \lambda \left(\begin{array}{c?c}
	\begin{array}{c|c}
	-A & \0\cr \hline
	-Q^*_EC & -Q^*_E
	\end{array} & \0_{(2n)\times (n+r_E)} \cr \thickhline
	\0_{(n+r_E)\times(2n)} & -\Id_{n+r_E}
	\end{array}\right)},
\end{equation}
where $Q^*_{E,1}=Q^*_E(1:r_E,:)$ and $Q^*_{E,2} = Q^*_E(r_E+1:n,:)$. Note that 
$\mathbb{A}_{22} - \lambda \mathbb{B}_{22}$ is the block at the position $(1,1)$ of a block-upper triangular pencil (\ref{eq:FirstTransformJordanTest}); the block position $(2,2)$ corresponds to the deflated $n-r_E$ zeros.
Denote the left and the right transformation matrices from (\ref{eq:FirstTransformJordanTest}) with $\mathbf{P}_1$ and $\mathbf{Q}_1$ respectively, and the linearization pencil with $\mathbb{A} - \lambda \mathbb{B} = \mathbb{A}_{11} - \lambda \mathbb{B}_{11}$. After the first deflation step we have\footnote{See \cite[\S 5.2]{KVADeig-arxiv} for more details.}
\begin{equation}\label{eq:pik}
\mathbf{P}_1(\mathbb{A}_{11} - \lambda \mathbb{B}_{11})\mathbf{Q}_1 = \begin{pmatrix}
\mathbb{A}_{22} - \lambda \mathbb{B}_{22} & \spadesuit\\
\0 & -\lambda \breve{\mathbb{B}}_{11}
\end{pmatrix},\;\; \breve{\mathbb{B}}_{11}=-\Id_{n-r_E}.
\end{equation}
The next step in the deflation process is to determine the rank of the matrix $\mathbb{A}_{22}$. From the structure of the matrix, we conclude that the rank of $\mathbb{A}_{22}$ is equal to $2n+r_E+$ \textit{"the rank of the $n\times n$ matrix }
$	\left(\begin{smallmatrix}
		Q^*_{E,2}D\\ \hline
		\widehat{R}_E{\Pi^*_E}
	\end{smallmatrix}\right)$",
which is defined in  terms of the coefficient matrices $D$ and $E$ of the original problem.
So, we compute the rank revealing factorization
\begin{equation}\label{eq:TestJordanZero}
	\left(\begin{array}{c}
		Q^*_{E,2}D\\ \hline
		\widehat{R}_E{\Pi^*_E}
	\end{array}\right)\Pi_{A_{22}} = Q_{A_{22}}R_{A_{22}}.
\end{equation}
If (\ref{eq:TestJordanZero}) is of full rank $n$, then $\mathbb{A}_{22}$ is regular, there are no more zeros in the spectrum, and the single deflation step is done by removing the trailing $n-r_E$ rows and columns in (\ref{eq:ABOneSingularOneJordanBlock}).
If, on the other hand, (\ref{eq:TestJordanZero}) is rank deficient with $\mathrm{rank}(R_{A_{22}})=r_2<n$, the corresponding number of $n-r_2$ zero eigenvalues can be deflated. To that end, note that $R_{A_{22}}=\left(\begin{smallmatrix} \widehat{R}_{A_{22}}\cr \0_{n-r_2,n}\end{smallmatrix}\right)$ and transform the pencil (\ref{eq:EsingularTruncatedPencil}) to get zero rows at the bottom of $\mathbb{A}_{22}$. 
%
This is done by the permutation 
$\pi = \begin{pmatrix}
1:n+r_E, & 2n+1:3n, & n+r_E+1:2n, & 3n+1:3n+r_E
\end{pmatrix}.$
If $\Pi$ is the corresponding row permutation matrix, and if we set 
\begin{equation}\label{eq:hatP2}
	\widehat{P}_2 = \begin{pmatrix}
	\Id_{2n+r_E} & \\
	 & Q^*_{A_{22}}
	\end{pmatrix}\Pi,
\end{equation}
then the transformed pencil is
\begin{equation}\label{eq:ABOneSingularSecondStep}
	\widehat{P}_2\mathbb{A}_{22} =\! \left(\begin{array}{c?c}
	\begin{array}{c|c}
	B & \0 \cr \hline
	Q^*_{E,1}D & \0 \\
	\0 & -\Id_n
	\end{array} & \begin{array}{c| c}
	-\Id_n & \cr \hline
	& -\Id_{r_E} \\
	\0 & \0
	\end{array} \\ \thickhline
	\begin{array}{c | c}
	\widehat{R}_{A_{22}}\Pi^T_{A_{22}} & \0 \cr 
	\0 & \0
	\end{array} & \0_{n\times (n+r_E)}
	\end{array}\right)\! , \;
	\widehat{P}_2\mathbb{B}_{22} = \!
\left(\begin{array}{c?c}
\begin{array}{c|c}
-A & \0\\ \hline
\begin{array}{c}- {Q}_{E,1}^*C\cr
\0_n \end{array} & \begin{array}{c} -{Q}_{E,1}^* \cr \0_n \end{array}
\end{array} & \begin{array}{c} \0_{n+{r}_E} \cr\hline \0_{{r}_E\times (n+{r}_E)} \cr\hline \begin{array}{c|c}-\Id_n & \0_{n\times {r}_E} \end{array}\end{array} \\ \thickhline
\begin{array}{c|c} -{N}_{[1]} & -{N}_{[2]}\end{array} & \begin{array}{c|c} {N}_{[3]} & {N}_{[4]}\end{array}
\end{array}\right)\! .	
%
\end{equation}
To deflate the additional $n-r_2$ zeros, {we reduce the trailing $n-r_2$ rows of the blocks $-{N}_{[1]}$, $-{N}_{[2]}$ and ${N}_{[3]}$ to zero}. This is done by the complete orthogonal decomposition
\begin{equation}\label{eq:URV-B22}
	(\widehat{P}_2\mathbb{B}_{22})(2n+r_E+r_2+1:3n+r_E,:) 
	= U_{BB}R_{BB}V^*_{BB},
\end{equation}
so that $(\widehat{P}_2\mathbb{B}_{22})(2n+r_E+r_2+1:3n+r_E,:)V_{BB}\! =\! \begin{pmatrix}
\0 & \breve{\mathbb{B}}_{22}
\end{pmatrix}$. Finally, the deflated pencil is
\begin{equation}\label{eq-penicil-AB33}
	\widehat{P}_2\mathbb{A}_{22}V_{BB} - \lambda \widehat{P}_2\mathbb{B}_{22}V_{BB} = \begin{pmatrix}
	\mathbb{A}_{33} - \lambda \mathbb{B}_{33} & \blacksquare \\	
	\0 & -\lambda\breve{\mathbb{B}}_{22}
	\end{pmatrix}.
\end{equation}
This reduction process continues by forwarding $\mathbb{A}_{33} - \lambda \mathbb{B}_{33}$ to the next step of reduction toward an upper triangular KCF, as described in \cite{KVADeig-arxiv}.

\begin{remark}
For a more structured backward error in case of graded matrices, the complete orthogonal (URV) decomposition (\ref{eq:URV-B22}) should be computed as in \cite[\S 2.2]{KVADeig-arxiv}.
\end{remark}

\begin{remark}
	If the matrix $A$ is rank deficient, and $E$ is full rank, we process the reversed problem $(\mu^4E + \mu^3D + \mu^2C + \mu B + A)x = \0$, $\mu=1/\lambda$, and the corresponding truncated linearization pencil of order $3n+r_A$ reads
	\begin{equation}
	\mathbb{A}_{22} - \lambda \mathbb{B}_{22} = \left(\begin{array}{c?c}
	\begin{array}{c|c}
	D & \0\cr \hline
	Q^*_{A,1}B & \0\\ \hline
	Q^*_{A,2}B & \0
	\end{array} & \begin{array}{c| c}
	-\Id_{n} & \\ \hline 
	& -\Id_{r_A}\\ \hline
	\0 & \0
	\end{array} \cr \thickhline
	\begin{array}{c|c}
	\0 & -\Id_n \cr \hline
	\widehat{R}_AP^*_A & \0 \cr
	\end{array} & \0_{n+r_A}
	\end{array}\right) - \lambda \left(\begin{array}{c?c}
	\begin{array}{c|c}
	-E & \0\cr \hline
	-Q^*_AC & -Q^*_A
	\end{array} & \0_{(2n)\times (n+r_A)} \cr \thickhline
	\0_{(n+r_A)\times(2n)} & -\Id_{n+r_A}
	\end{array}\right),
	\end{equation}
	and the rank of matrix $\mathbb{A}_{22}$ is now $2n+r_A+$ \textit{the rank of the $n\times n$ matrix}
$\left(\begin{smallmatrix}
	Q^*_{A,2}B\\ \hline
	\widehat{R}_AP^*_A
	\end{smallmatrix}\right)$.
\end{remark}

\subsubsection{\underline{Both matrices $A$ and $E$ are singular}}\label{SSS=both-singular}
When both matrices $A$ and $E$ are rank deficient, then, following the discussion from \S \ref{SSS=one-singular},  the key information is in the numerical ranks of the matrices

	\begin{equation}\label{eq:QuarticJordanCheck}
\Phi =	\left(\begin{array}{c}
	Q^*_{A,2}B\\ \hline
	\widehat{R}_A{\Pi^*_A}
	\end{array}\right),\;\;\Psi = \left(\begin{array}{c}
	Q^*_{E,2}D\\ \hline
	\widehat{R}_E{\Pi^*_E}
	\end{array}\right).
	\end{equation}
	
\paragraph{\underline{Both $\Phi$ and $\Psi$ are full rank}}
	In this case, in the KCF the zero and the infinite eigenvalue occupy single block each, induced by the rank deficiency of $E$ and $A$. The deflation process starts by creating $n-r_E$ and $n-r_A$ zero rows in the coefficients of the corresponding linearization as follows:
\begin{eqnarray}
&& \begin{pmatrix} Q_{M}^* & \0 \cr \0 & Q_{K}^* \end{pmatrix} \left\{ \begin{pmatrix} \Cbb & -\Id_{2n} \cr \Kbb & \0\end{pmatrix} -\lambda \begin{pmatrix} -\Mbb & \0 \cr \0 & -\Id_{2n} \end{pmatrix}\right\}\begin{pmatrix} \Id_{2n} & \0 \cr \0 & Q_{K} \end{pmatrix}\nonumber \\
&=& \!\!\!\!
\left(\begin{array}{c?c}
\begin{array}{c | c}
\0_n & D\\ \hline
\0_{n\times r_A} & Q_A^*(1:r_A,:)B\\
\0_{n\times (n-r_A)} & Q_A^*(r_A+1:n,:)B
\end{array} & \begin{array}{c | c c}
\0_n & -Q_E(:,1\!:\! r_E) & -Q_E(:,r_E\! +\! 1\!:\! n)\\ \hline
-Q^*_A(1:r_A,:) & \0_{r_A\times r_E} & \0_{r_A\times (n-r_E)}\\
-Q^*_A(r_A\! +\! 1\! :\! n,:) & \0_{(n-r_A)\times r_E} & \0_{(n-r_A)\times(n-r_E)} \end{array}\cr \thickhline
\begin{array}{c|c}
-\Id_{n} & \0_n\\ \hline
\0_{r_E\times n} & \widehat{R}_E{\Pi^*_E} \\
\0_{(n-r_E)\times n} & \0_{(n-r_E)\times n}
\end{array} & \begin{array}{c | c c}
\0_{n} & \0_{n\times r_E} & \0_{n\times (n-r_E)}\\ \hline
\0_{r_E\times n} & \0_{r_E} & \0_{r_E\times (n-r_E)}\\
\0_{(n-r_E)\times n} & \0_{(n-r_E)\times r_E} & \0_{(n-r_E)}
\end{array}
\end{array}\right) \nonumber \\
&-\lambda& \left(\begin{array}{c?c}
\begin{array}{c | c}
-\Id_n & C\\ \hline
\0_{r_A\times n} & -\widehat{R}_A{\Pi^*_A}\\ 
\0_{(n-r_A)\times n} & \0_{(n-r_A)\times n}
\end{array} & \begin{array}{c|c c}
\0_n & \0_{n\times r_E} & \0_{n\times (n-r_E)}\\ \hline
\0_{r_A\times n} & \0_{r_A\times r_E} & \0_{r_A\times (n-r_E)}\\
\0_{(n-r_A)\times n} & \0_{(n-r_A)\times r_E} & \0_{(n-r_A)\times (n-r_E)}
\end{array} \\ \thickhline
\begin{array}{c|c}
\0_n & \0_n \\ \hline
\0_{r_E\times n} & \0_{(n-r_E)\times n}\\
\0_{(n-r_E)\times n} & \0_{(n-r_E)\times n}
\end{array} & \begin{array}{c | c c}
-\Id_n & \0_{n\times r_E} & \0_{n\times (n-r_E)}\\\hline
\0_{r_E\times n} & -\Id_{r_E} & \0_{r_E\times (n-r_E)}\\
\0_{(n-r_E)\times n} & \0_{(n-r_E)\times r_E} & -\Id_{n-r_E}
\end{array}
\end{array}\right).
\label{eq:ABBothSingular}
\end{eqnarray} 
The next step is to compute the complete orthogonal decomposition
\begin{equation}
	\left(\begin{array}{c c c}Q_A^*(r_A+1:n,:)B & Q_A^*(r_A+1:n,:) & \0_{(n-r_A)\times r_E}\end{array} \right) = Q_X\left( \begin{array}{c c}
	R_X & \0_{(n-r_A)\times (n+r_E+r_A)}
	\end{array} \right)Z_X ,
\end{equation}
and permute the first $(n-r_A)$ and the last $(n+r_E+r_A)$ columns to get
\begin{align*}
	&Q^*_X\left(\begin{array}{c c c}Q_A^*(r_A+1:n,:)B & Q_A^*(r_A+1:n,:) & \0_{(n-r_A)\times r_E}\end{array} \right)Z^*_X \left(\begin{array}{c c}
	\0 & \Id_{n-r_E}\\
	\Id_{n+r_A+r_E} & \0
	\end{array}\right)\\
	& = \left(\begin{array}{c c}
	\0_{(n-r_A)\times (n+r_E+r_A)} & R_X
	\end{array}\right).
\end{align*}
Finally, to complete the deflation process, the following left and right transformation matrices must be applied on the pencil \eqref{eq:ABBothSingular}:
\begin{equation*}
	\left(\begin{array}{c c c c c}
	\Id_{n+r_A} & \0 & \0 & \0 &\0\\\hline
	\0 & \0 & \0 & \Id_{r_E} & \0 \\
	\0 & \0 & \Id_n & \0 & \0\\
	\0 & Q^*_X &\0 &\0 &\0\\\hline
	\0 & \0 &\0 &\0 & \Id_{r_E}
	\end{array}\right),\;\;\; \left(\begin{array}{c |c}
	Z^*_X\left(\begin{array}{c c}
	\0 & \Id_{n-r_E}\\
	\Id_{n+r_A+r_E} & \0
	\end{array}\right) & \0 \\ \hline
	\0 & \Id_{2n-r_E}
	\end{array}\right).
\end{equation*}
After the transformation step, the deflation is finished by removing the last $2n-r_E-r_A$ rows and columns from the obtained pencil. The resulting pencil of dimension  $2n+r_A+r_E$ is forwarded to the QZ algorithm.

\paragraph{\underline{Only one matrix in (\ref{eq:QuarticJordanCheck}) is singular}}
	 This means that there are at least two KCF blocks for the zero (if $\Psi$ is singular) or the infinite (if $\Phi$ is singular) eigenvalue. In either case, we deflate two  blocks for the zero eigenvalue using the structure described in \S \ref{SSS=one-singular} ({see also \cite[\S 5.2, \S 6.1]{KVADeig-arxiv}}), meaning that the reversed problem is considered if there are more  blocks for the infinite eigenvalues. 
	 
	 After deflating two blocks of zero eigenvalues, we obtain the pencil (\ref{eq-penicil-AB33}). Now, the existence of additional zero eigenvalues depends on the rank of the matrix $\mathbb{A}_{22}$. To deflate possible additional zeros, the pencil $\mathbb{A}_{22} - \lambda \mathbb{B}_{22}$ is forwarded to the algorithm for computing  the KCF \cite[\S 3.2]{KVADeig-arxiv}. As the output we get the pencil $\mathbb{A}_{\ell+1,\ell+1}-\lambda \mathbb{B}_{\ell+1,\ell+1}$ and transformation matrices $Q_p$ and $P_p$, with $\mathbb{A}_{\ell+1,\ell+1}$ regular. Denote with $n_{\ell+1}$ the dimension of the resulting pencil.
	 
	 Finally, we have to deflate one block of infinite eigenvalues, which have been detected at the beginning. This is done by forwarding the reversed pencil $\mathbb{B}_{\ell+1,\ell+1} - \lambda \mathbb{A}_{\ell+1,\ell+1}$ to the procedure described in \cite[\S 3.2]{KVADeig-arxiv}. As the input to the algorithm we supply the information that there is only one block to be deflated, so that only one step of the algorithm is needed. In addition, we also send the number of infinite eigenvalues so that the rank determination of the matrix $\mathbb{B}_{\ell+1,\ell+1}$ is omitted. As an output, we get the pencil $\mathbb{A}_{\ell + \ell_1,\ell+\ell_1} - \lambda \mathbb{B}_{\ell + \ell_1,\ell+\ell_1}$ with both $\mathbb{A}_{\ell + \ell_1,\ell+\ell_1}$ and $\mathbb{B}_{\ell + \ell_1,\ell+\ell_1}$ regular, and the corresponding transformation matrices $P_{p1}$ and $Q_{p1}$.
	 The final transformation matrices $Q$ and $P$ are
	 
	 \begin{align*}
	 	Q &= \left(\begin{smallmatrix}
	 	\Id_{2n} & \0 \\
	 	\0 & Q_{\Kbb}
	 	\end{smallmatrix}\right)
	 	\left(\begin{smallmatrix} \Id_{2n} & \0 & \0 \cr \0 & V^*_{BB}P_{BB} & \0 \cr \0 & \0 &\Id_{n-r_E}\end{smallmatrix}\right)
	 	\left(\begin{smallmatrix} Q_p & \0 \cr \0 & \Id_{4n-r_E-r_2}\end{smallmatrix}\right)
	 	\left( \begin{smallmatrix} Q_{p1} & \0 \cr \0 & \Id_{4n-n_{\ell+1}}\end{smallmatrix}\right)\\
	 	P &= \left(\begin{smallmatrix} P_{p1} & \0 \cr \0 & \Id_{4n-n_{\ell+1}}\end{smallmatrix}\right)
	 	\left(\begin{smallmatrix} P_p & \0 \cr \Id_{4n-r_E-r_2} & \0\end{smallmatrix}\right)
	 	\left(\begin{smallmatrix} \Id_{n+r_E} & \0 & \0 \cr \0 & Q^*_{A_{22}} & \0 \cr \0 & \0 & \Id_{2n-r_E}\end{smallmatrix}\right) \left(\begin{smallmatrix}
	 	Q^*_{\Kbb} & \0\\
	 	\0 & Q^*_{\Kbb}
	 	\end{smallmatrix}\right).
	 \end{align*}

\paragraph{\underline{Both matrices in (\ref{eq:QuarticJordanCheck}) are singular}} 
This case is analogous to the previous one. The only difference is that, when we call the algorithm on the reversed pencil $B_{\ell+1,\ell+1} - \lambda A_{\ell+1,\ell+1}$, we provide additional information that there are least two steps of deflation ahead, as well as the dimensions of the first two blocks which were previously determined by the rank revealing decompositions of $A$ and $\Phi$.\\

{
\subsubsection{On making more reduction steps}
After all detected zero and/or infinite eigenvalues have been deflated, as described above,  we check the ranks of the matrices in the resulting pencil in order to determine whether there are more blocks of these eigenvalues. If these matrices are rank deficient, then another step of deflation must take place. Unfortunately, with that step, the structure of the linearization is lost, so we use the standard deflation process for generalized eigenvalue problem as in \cite{van1979computation} and \cite{KVADeig-arxiv}. It remains an interesting problem to determine an equivalence transformation to restore the structure for more steps, while working on an equivalent representation of the original problem.

There is, of course, a trade-off between this increased numerical robustness and computational cost (complexity), and, in a software implementation, the number of reduction/deflation steps will be limited.
But, even with these few steps we can make some critical decisions on the zero and infinite eigenvalues, with backward error in terms of the coefficients of the original problem; see \S \ref{S=BackError} and \S \ref{S=KVATREIG-NUMEX}.
}

\subsubsection{{An illustrative example}}\label{NUMEX-4} Let us illustrate the action of the additional reduction steps toward the KCF.  We use the \texttt{mirror} example from the NLEVP library \cite{betcke2010nlevp}; it originates  from the calibration of catadioptric vision system \cite{Zhang-vision-4608140}. The problem is of order $n=9$.
	
	Both $A$ and $E$ are rank deficient, with the rank $r_E = r_A = 2$, which means that there are at least $7$ zero {and $7$} infinite eigenvalues. They were correctly identified and deflated in the preprocessing in \texttt{quadeig}\footnote{For the purpose of testing and comparisons, we apply quadratic solvers to the quadratification  (\ref{eq:SecondComapnionFormGrade2}) of the quartic problem.}; in the next step, the QZ algorithm found an additional zero eigenvalue, and two more infinite eigenvalues. On the other hand, \texttt{polyeig}\footnote{We use \texttt{polyeig} from Matlab, version 7.11.0.584 (R2010b).} identified in total only $2$ zero  and $9$ infinite eigenvalues. This shows the advantage of the preprocessing introduced in \texttt{quadeig} for early revealing of zeros and infinities. {These  numbers of computed zero and infinite eigenvalues were independent of whether the parameter scaling was on or off before calling \texttt{quadeig} and \texttt{polyeig}.}
	
	On the other hand, the preprocessing in both \texttt{kvadeig} and \texttt{kvarteig} found additional {two zero and two infinite eigenvalues}, making the total of $9$ zero and $9$ infinite eigenvalues deflated before calling the QZ. {Again, the same numbers of $9$ zero and $9$ infinite eigenvalues were found with and without parameter scaling. This almost agrees with the result of \texttt{quadeig}, up to one zero eigenvalue. 
	
	Next, we check the norm-wise backward errors \eqref{eq:back-err-0} for all computed eigenpairs (for all four algorithms). The details of computing the eigenvectors in \texttt{kvarteig} are given in \S \ref{S=Vectors}.
	The computed residuals, shown in Figure \ref{fig:mirrorBE}, seem to indicate that all results are acceptable up to small norm-wise backward errors (separate for each eigenpair) 
	of the order of machine precision. (The eigenvalues are indexed in non-decreasing absolute values.)
	}
	
	
	\begin{figure}[ht]
		\centering
		\begin{minipage}{.5\textwidth}
			\includegraphics[width=1\textwidth]{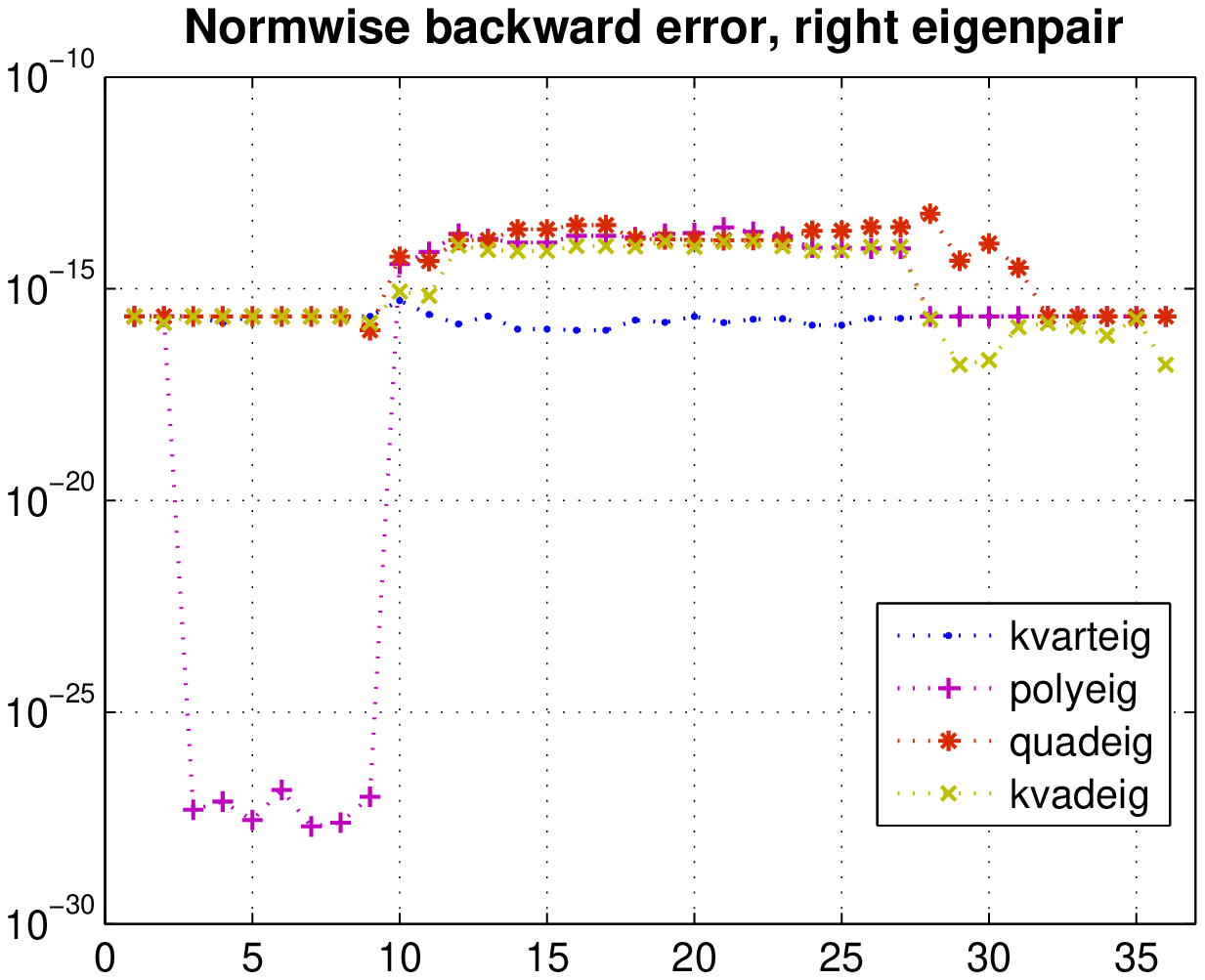}
			
		\end{minipage}%
		\begin{minipage}{.5\textwidth}
			\centering
			\includegraphics[width=1\textwidth]{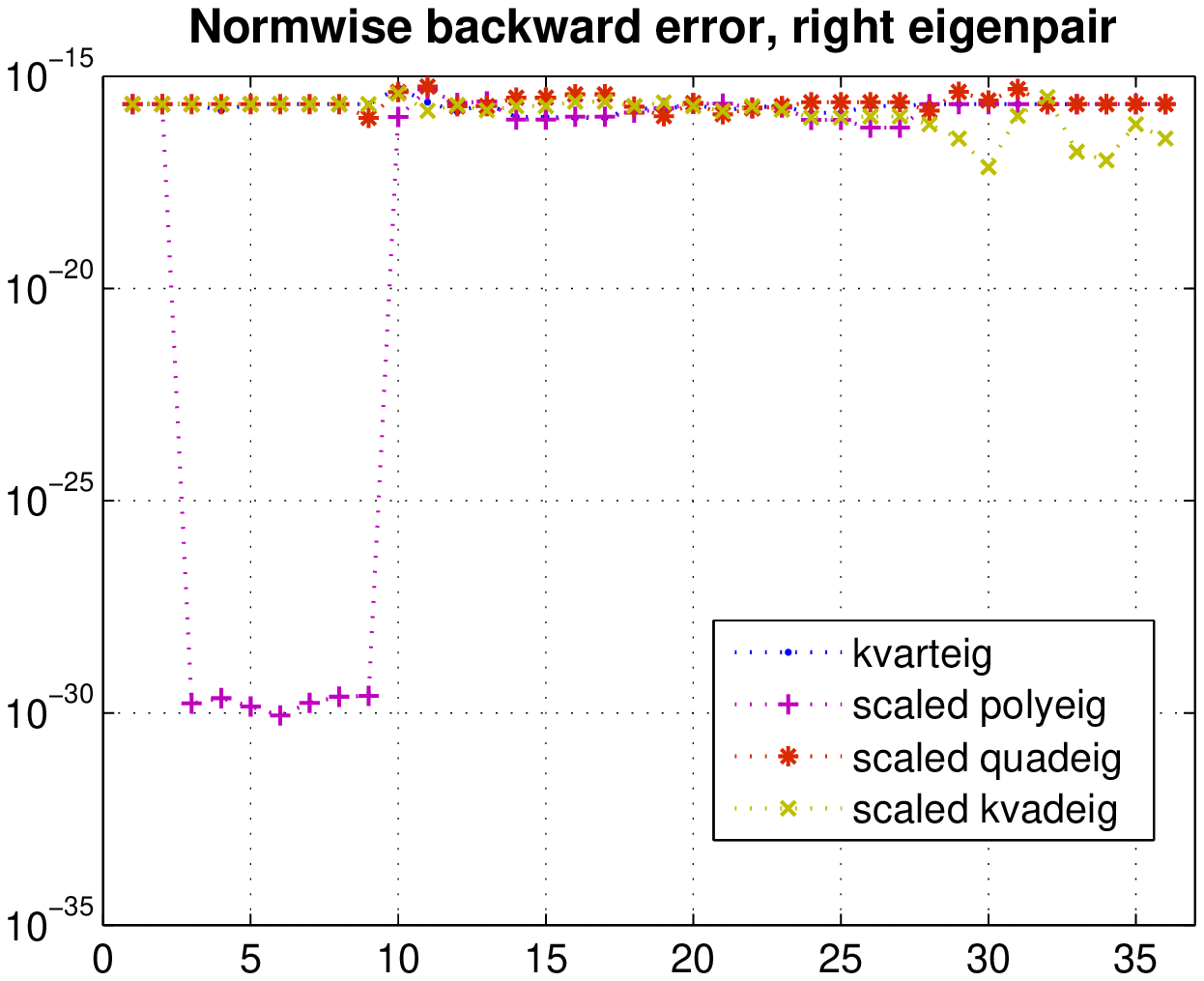}	
			
		\end{minipage}
		\caption{Norm-wise backward errors for all eigenvalues with the corresponding right eigenvectors in the \texttt{mirror} NLEVP benchmark example. \emph{Left panel:} No parameter scaling in \texttt{polyeig} and the quadratic solvers. \emph{Right panel:} The coefficients of the problem are scaled as described in \S \ref{SSS=param-scaling}.}
		\label{fig:mirrorBE}
	\end{figure}
	

{Hence, with all backward errors at the level of the round-off, and with different numbers of zero and infinite eigenvalues computed by different algorithms, how can we tell which one is correct? What assurance is given in a particular algorithm concerning the existence of infinite eigenvalues of a perturbed matrix polynomial in a vicinity of the given one? The difficulty is best illustrated in \cite[Example 2]{van1979computation}, which actually contains the key idea pursued in \texttt{quadeig}, \texttt{kvadeig} and \texttt{kvarteig}.}

{
If we look at the structure of the matrices $A$ and $E$ for this particular problem, we see that their ranks can be determined exactly because each has  $7$ zero columns, and the independence of the remaining columns is easy to check. Further, the block matrices (\ref{eq:QuarticJordanCheck}), which are used to determine the existence of more than one block for zero and infinite eigenvalues, also have two zero columns each, and the remaining $9\times 7$ submatrices are well conditioned. Thus we can argue that \texttt{kvarteig} has determined the correct numbers of zero and infinite eigenvalues.\\
\indent {We call the reader to revisit this example after reading Example \ref{EX-mirror-transposed}.}
}

\section{Computing the eigenvectors}\label{S=Vectors}
In the computation of the eigenvectors, we have two main computational tasks: \emph{(i)} restore the eigenvectors of the quartic problem from the eigenvectors of its \emph{linearization} (\S \ref{SS:quart-vectors-from-lin}); \emph{(ii)} assemble the eigenvectors  of the linearization from the eigenvectors of the deflated (linearization) pencil, using the transformation matrices (\S \ref{SS=Assemble-of-linear}).

\subsection{Quartic eigenvectors from the eigenvectors of the linearization}\label{SS:quart-vectors-from-lin}
For an eigenvalue $\lambda$, the  eigenvectors of the original problem (\ref{eq:QuarticEP}) and the final linearization pencil (\ref{eq:FinalLinearization}) can be related using  explicit formulas. For the reader's convenience, we briefly outline the crux of this connection. 

We use $z\in \mathbb{C}^{4n}$ and $w\in \mathbb{C}^{4n}$ to denote the right and the left eigenvector for the linearization, and $x\in \mathbb{C}^{n}$, $y\in \mathbb{C}^{n}$ to denote the right and the left eigenvector for the original problem. The eigenvalue $\lambda \in \mathbb{C}$ is now fixed as assumed nonzero and finite. 

Let $z = \left( \begin{array}{cccc}
z_1^T & z_2^T & z_3^T & z_4^T
\end{array} \right)^T \in \mathbb{C}^{4n}$, $z_i\in \mathbb{C}^n$, $i = 1,2,3,4$ be a right eigenvector for the eigenvalue $\lambda$ ($0 < |\lambda| < \infty$) of the linearized problem, i.e. $(\mathbb{A}-\lambda \mathbb{B})z = \0$ :
\begin{equation}\label{eq:4-2-1}
(\mathbb{A}-\lambda \mathbb{B}) z = \left\{\left(\begin{array}{cc?cc}
B & \0 & -\Id & \0 \\
D & \0 & \0 & -\Id \\ \thickhline
\0 & -\Id & \0 & \0 \\
E & \0 & \0 & \0
\end{array}\right)-\lambda \left(\begin{array}{cc?cc}
-A & \0 & \0 & \0 \\
-C & -\Id & \0 & \0 \\\thickhline
\0 & \0 & -\Id & \0 \\
\0 & \0 & \0 & -\Id
\end{array}\right)\right\} \left(\begin{array}{c}
z_1\\
z_2\\
z_3\\
z_4
\end{array}\right) = \left(\begin{array}{c}
\0\\
\0 \\
\0\\
\0
\end{array}\right).
\end{equation}
By equating the corresponding block components on the left and on the right we get
\begin{eqnarray}
Bz_1-z_3+\lambda A z_1 &=&\0 \Leftrightarrow z_3 = (\lambda A + B)z_1, \label{eq:z-1}\\
Dz_1-z_4+\lambda Cz_1 + \lambda z_2 &=& \0 \Leftrightarrow Dz_1+ (1/\lambda)Ez_1+\lambda Cz_1 + \lambda ^2(\lambda A + B)z_1 = \0, \label{eq:z-2}\\
-z_2 + \lambda z_3 &=& \0 \Leftrightarrow z_2 = \lambda z_3, \label{eq:z-3}\\
Ez_1 +\lambda z_4 &=& \0 \Leftrightarrow \lambda z_4 = -Ez_1.\label{eq:z-4}
\end{eqnarray}
It follows immediately that $z_1\neq \0$; if $\mathrm{det(\lambda A +B)}\neq 0$, then, in addition, $z_3\neq \0$ and $z_2\neq\0$; if $\mathrm{det}(E)\neq 0$, then also $z_4\neq\0$. 
Using (\ref{eq:z-2}) we easily check that $x = z_1/\lambda$ is an eigenvector of the original quartic problem. Further, (\ref{eq:z-1}) implies that $x$ satisfies $z_3 = \lambda (\lambda A + B)x$, and (\ref{eq:z-3}) yields $z_2 = \lambda^2(\lambda A + B)x$, and finally from (\ref{eq:z-4}) it follows that  $z_4 = -Ex$.  Similarly, if we initially assume that $x$ is an eigenvector of the quartic problem, these formulas for the $z_i$'s give an eigenvector of (\ref{eq:4-2-1}). 

An analogous computation reveals a left eigenvector $y$, using the partitioned left eigenvector of the linearization, as $w = \begin{pmatrix}
w^T_1 & w^T_2 & w^T_3 & w^T_4
\end{pmatrix}^T$, $w_i \in \mathbb{C}^n,i=1,2,3,4$. Altogether, we obtain the following relations between the two sets of eigenvectors:
\begin{equation}\label{eq:RightEigenvectorRecovery} 
z = \begin{pmatrix}
z_1\\
z_2\\
z_3\\
z_4
\end{pmatrix} = \begin{pmatrix}
\lambda x\\
\lambda^2(\lambda A + B)x\\
\lambda (\lambda A + B) x\\
- E x
\end{pmatrix} ,  \;\:
w = \begin{pmatrix}
w_1\\
w_2\\
w_3\\
w_4
\end{pmatrix} = \begin{pmatrix}
\lambda^3y\\
\lambda^2 y\\
\lambda y\\
y
\end{pmatrix}.
\end{equation}
For both the right and the left eigenvector there are four choices to recover $x$ and $y$. 
Reconstruction of the left eigenvector seems easier. We just choose one of the block components $w_1,w_2,w_3$ or $w_4$ and rescale appropriately.

For the right eigenvector we can choose $z_1$, $(\lambda A + B)^{-1}z_2$, $(\lambda A + B)^{-1}z_3$ or $E^{-1}z_4$. Notice that, for the last three choices we have to solve system of linear equations in order to compute the wanted vector. Given all the difficulties in numerical solution of nonlinear eigenvalue problem, we ought to use all alternatives in order to obtain better output -- in this case, for instance, we can solve all systems and select the vector with smallest residual. 

\begin{remark}\label{REM-zero-eig}
If $\lambda = 0$, for the corresponding right eigenvector we have $Ex = \0$ and $\mathbb{A}x = \0$. By the same reasoning as above, we conclude the following connection $z = \left(\begin{array}{cccc}
x^T & \0 & (Bx)^T & (Dx)^T
\end{array}\right)^T$.
\end{remark}
\subsubsection{\underline{Computing  $(\lambda A + B)^{-1}z_2$, $(\lambda A + B)^{-1}z_3$ or $E^{-1}z_4$ multiple times}}
Inverting $E$ (assuming $\mathrm{det}(E)\neq 0$) multiple times can be done by reusing initially computed LU decomposition. On the other hand computing  $(\lambda A + B)^{-1}z_2$, $(\lambda A + B)^{-1}z_3$ for $4n$ values of $\lambda$ is not that simple because the coefficients of the linear system change with $\lambda$; $O(n^3)$ \emph{flops} per eigenvalue to compute the corresponding eigenvector is prohibitive complexity. Fortunately, this can be reduced using a bag of tricks for solving shifted linear systems. {In particular, this problem is similar to evaluating the transfer function of a descriptor LTI dynamical system at multiple frequencies \cite[\S 4]{descriptor-sys-Mehrmann-2006}.} 

 We can compute the triangular-Hessenbeg form of $(A,B)$, i.e. a unitary $Q$, an upper triangular $T$ and an upper Hessenberg matrix $H$ can be constructed in $O(n^3)$ time so that $A = Q T Q^*$, $B=QHQ^*$. Hence, for any vector $v$
$$
(\lambda A + B)^{-1}v = Q [ (\lambda T + H)^{-1}(Q^* v) ] ,
$$
which has $O(n^2)$ complexity because $\lambda T + H$ is upper Hessenberg. This means that the total work (for all $4n$ eigenvalues) of choosing the eigenvectors with smallest residuals remains $O(n^3)$. (Here, the tacit assumption is that $A+\lambda B$ is nonsingular and well conditioned with respect to inversion.) {These details can be taken into account for a development of an optimized software for multicore architectures; for more information see \cite{BBD-Gen-Hess-2013},  \cite{BBD-shifted-2018}.}
\begin{remark}
In some applications, such as e.g. computing deformation modes of thin-walled structures \cite{VIEIRA2014575}, the cubic term is zero, $B=\0$, so that the shifted systems can be replaced with linear system matrix $A$ for all $\lambda$'s. Other details include e.g. the case of real data and using the complex conjugate eigenpairs to save unnecessary computation. Here we omit those details and leave them for the detailed description of a software implementation, which is a subject of our future  work.
\end{remark}

\subsubsection{Least squares reconstruction of the eigenvectors}\label{SSS=LS-vectors}
Since in a finite precision computation the computed eigenvector $z$ is only an approximation (thus noisy), and since  $A+\lambda B$ is not guaranteed to be well conditioned, it makes sense to turn the conditions (\ref{eq:RightEigenvectorRecovery}) into a least squares problem, but keeping in mind than we may have to solve it $4n$ times (i.e. we may take e.g. only two conditions to form the least squares problem).
	So, for instance, we can compute $x$ by solving the least squares problem
	\begin{equation}\label{eq:LS-for-x}
	\left\| \begin{pmatrix} \lambda\Id_n \cr E \end{pmatrix} {x} - \begin{pmatrix} z_1\cr -z_4\end{pmatrix}\right\|_2
	\longrightarrow \min \;\;
	(\;\mbox{or e.g.}\;\;\left\| \begin{pmatrix} \Id_n \cr E \end{pmatrix} {x} - \begin{pmatrix} z_1/\lambda \cr -z_4\end{pmatrix}\right\|_2
	\longrightarrow \min\;)
	\end{equation}
	{Actually, the second (and any additional) condition serves as a regularization that can be given a positive weight. Such a strategy can be used  for a deeper study of selected eigenpairs. We omit the details for the sake of brevity.}

	If the data is well scaled, then for some eigenvalues (semi-)normal equations can be used.
	In general, the least squares problem (\ref{eq:LS-for-x}) can be solved efficiently for any eigenvalue $\lambda\neq 0$ by pre-computing the SVD $E=U_E \Sigma_E V_E^*$ (which actually may be available if we used it for a strong rank revealing of $E$) and then, for each triple $\lambda$, $z_1$, $z_4$, solving in $O(n^2)$ \emph{flops} the equivalent problem
\begin{equation}\label{eq:LS-for-x-simpler}
\left\| \begin{pmatrix} \lambda\Id_n \cr \Sigma_E \end{pmatrix} {V_E^*x} - \begin{pmatrix} V_E^*z_1\cr -U_E^*z_4\end{pmatrix}\right\|_2
\longrightarrow \min . \;\;(\;\mbox{or}\;\;
\left\| \begin{pmatrix} \Id_n \cr \Sigma_E \end{pmatrix} {V_E^*x} - \begin{pmatrix} V_E^*z_1/\lambda\cr -U_E^*z_4\end{pmatrix}\right\|_2
\longrightarrow \min. \;)
\end{equation}
If $\lambda=0$, then, based on Remark \ref{REM-zero-eig}, the corresponding eigenvector can be found from either of the following least squares problems
\begin{equation}
\left\| \begin{pmatrix} \Id_n \cr B \end{pmatrix} {x} - \begin{pmatrix} z_1\cr z_3\end{pmatrix}\right\|_2
\longrightarrow \min,\;\;\;
 \left\| \begin{pmatrix} \Id_n \cr D \end{pmatrix} {x} - \begin{pmatrix} z_1\cr z_4\end{pmatrix}\right\|_2
 \longrightarrow \min ,
\end{equation}
which can be efficiently solved for all eigenvectors $z$ of $\lambda=0$, using one of the approaches discussed above. Other possibilities include e.g. using the bidiagonalization instead of the SVD (of $B$ or $D$).
\subsection{Assembling the eigenvectors of the linearization }\label{SS=Assemble-of-linear}
Let $\widetilde{z}$ and $\widetilde{w}$ be the computed right and left eigenvector for the linearization pencil (\ref{eq:FinalLinearization}). {Both right and left eigevectors will have $4n$ elements if no deflation occurred, otherwise the number of elements will be $4n-d$}, where $d$ is the total number of zero and infinite eigenvalues deflated. $4n-d$ is also the dimension of the truncated pencil $\widetilde{\mathbb{A}}_{22} - \lambda \widetilde{\mathbb{B}}_{22} = P(\mathbb{A} - \lambda \mathbb{B})Q$ which is passed to the QZ algorithm for computation of finite nonzero eigenvalues. 

\subsubsection{Case 1: No deflation has occurred} Let $\widetilde{z}$ and $\widetilde{w}$ be the right and the left eigenvector of the transformed pencil $P(\mathbb{A}-\lambda \mathbb{B})Q$. The corresponding right and the left eigenvectors for the original linearization pencil are $z = Q\widetilde{z}$ and $w = P^T\widetilde{w}$. The right and the left eigenvector for the quartic problem are computed as described in \S \ref{SS:quart-vectors-from-lin}.
\subsubsection{Case 2: Deflation has occurred} 
Let $n_{\ell+1}$ be the dimension of the deflated linearization $A_{\ell+1,\ell+1}-\lambda B_{\ell+1,\ell+1}$, i.e. both $A_{\ell+1,\ell+1}$ and $B_{\ell+1,\ell+1}$ are regular. Let $\widetilde{z}\in \mathbb{C}^{n_{\ell+1}}$ and $\widetilde{w}\in \mathbb{C}^{n_{\ell+1}}$ be the right and the left eigenvector for a finite nonzero eigenvalue $\lambda$.

To recover eigenvectors of the initial linearization,  we must lift $\widetilde{z}$ and $\widetilde{w}$ to the $4n$-dimensional space. For the right eigenvector this is easy; we just append  $4n-n_{\ell+1}$ zeros to $\widetilde{z}$ to get $z = Q\begin{pmatrix}
\widetilde{z}^T &
\0_{1\times(4n-n_{\ell+1}) }
\end{pmatrix}^T$.

For the left eigenvector, let $\widetilde{w}_2 \in \mathbb{C}^{n-n_{\ell+1}}$ be the vector satisfying $\begin{pmatrix}
\widetilde{w}^T & \widetilde{w}_2^T
\end{pmatrix} P(\mathbb{A}-\lambda \mathbb{B})Q = \0$. From
\begin{equation}
\begin{pmatrix}
\widetilde{w}^T & \widetilde{w}_2^T
\end{pmatrix} P(\mathbb{A} - \lambda \mathbb{B})Q = 	\begin{pmatrix}
\widetilde{w}^T & \widetilde{w}_2^T
\end{pmatrix} \begin{pmatrix}
\mathbb{A}_{\ell+1,\ell+1} - \lambda \mathbb{B}_{\ell+1,\ell+1} & X\\
\0 & Y
\end{pmatrix}
\end{equation}
we conclude that $\widetilde{w}^T_2 = -\widetilde{w}^T XY^{-1}$. (Note that it follows from the procedure in \S \ref{SSS=one-singular} that $Y$ is nonsingular. {Namely, $Y$ is a block upper triangular matrix with the diagonal blocks that  are by construction nonsingular.}) Now, the left eigenvector for the original linearization is $w = P^T \begin{pmatrix}
	\widetilde{w}^T & \widetilde{w}_2^T
\end{pmatrix}^T$. 

The right eigenvectors for the zero eigenvalue span the nullspace of the matrix $E$. The corresponding basis is computed for the orthogonal complement of the range of $E^*$. To compute this basis we can use the already computed QR factorization of $E$ (\ref{eq:RankRevealingQRAE}) as follows. First compute the QR factorization of $\Pi_E\widehat{R}_E^* = Q_{\widehat{R}^*_{E}}R_{\widehat{R}^*_E}$. Now, the last $n-r_E$ columns of $Q_{\widehat{R}^*_E}$ represent the basis for the nullspace of the matrix $E$. Similarly, the right eigenvectors of the infinite eigenvalue span the nullspace of the matrix $A$. The basis is computed using the already computed QR factorization (\ref{eq:RankRevealingQRAE}). Again, compute the QR factorization of $\Pi_A\widehat{R}_A^* = Q_{\widehat{R}^*_{A}}R_{\widehat{R}^*_A}$, and the last $n-r_A$ columns of $Q_{\widehat{R}^*_A}$ represent the basis for the nullspace of $A$.

The left eigenvectors for the zero eigenvalue are determined as the last $n-r_E$ columns of the unitary matrix $Q_E$ from the corresponding QR factorization, and the left eigenvectors for the infinite eigenvalue are selected as the last $n-r_A$ columns of the unitary matrix $Q_A$ from the QR factorization of $A$.

\section{Backward error analysis}\label{S=BackError}
{As we discussed in \S \ref{SS=Scaling} and \S \ref{SS=Quandary-infinite-eigs}, it is important that the computational errors in the preprocessing phase correspond (in a backward error sense) to small perturbations of the initial coefficients, i.e. that we have strong norm-wise backward stability. Moreover, if the initial coefficient matrices have graded columns, then it is advantageous to have backward error in each column to be small relative to the size of that column (instead of relative to the norm of the whole matrix). In this section, we analyze the backward errors in the proposed preprocessor, where we use a framework of mixed error analysis. As expected, such an analysis is rather technical.}

{We provide a backward/mixed error analysis for the first two steps of the deflation procedure described in \S \ref{SSS=one-singular} (only one of $A$ and $E$ is singular),  which includes the key details and principles of the analysis.  Extending this further  to the first two steps of the general case \S \ref{SSS=both-singular}  (both $A$ and $E$ singular) would follow the same steps and it is omitted for the sake of brevity.} 

The following proposition deals with the first step, that is, the deflation of the first batch of $n-r_E$ zero eigenvalues.

\begin{proposition}\label{PROP1}
	Let $E \widetilde{\Pi}_E\approx \widetilde{Q}_E \left(\begin{smallmatrix}\widetilde{R}_E\cr \0\end{smallmatrix}\right)$ be the computed rank revealing QR factorization of $E$, and let $\widetilde{r}_E$ be the computed numerical rank of $E$. Further, let $\widetilde{X}=computed(\widetilde{Q}_E^* D)$, $\widetilde{Y}=computed(\widetilde{Q}_E^* C)$. 
	Let 
	\begin{equation}\label{eq:LGEVP:n+rE:computed}
	\widetilde{\mathbb{A}}_{22}-\lambda\widetilde{\mathbb{B}}_{22}= \left( \begin{array}{c?c} 
	\begin{array}{c|c}
	B & \0\\ \hline
	\widetilde{X} & \0
	\end{array} & \begin{array}{c|c}
	-\Id_n & \0 \\ \hline
	\0 & -\Id_{\widetilde{r}_E}\\
	\0 & \0
	\end{array} \\ \thickhline
	\begin{array}{c|c}
	\0 & -\Id_n \\ \hline
	\widetilde{R}_E\widetilde{\Pi}^T_E & \0
	\end{array} & \0_{n+\widetilde{r}_E}
	\end{array}\right) - \lambda \left(\begin{array}{c?c}
	\begin{array}{c|c}
	-A & \0\\ \hline
	-\widetilde{Y} & -\widetilde{Q}^*_E
	\end{array} & \0_{2n \times (n+\widetilde{r}_E)} \\ \thickhline
	\0_{(n+\widetilde{r}_E)\times 2n} & -\Id_{n+\widetilde{r}_E}
	\end{array}\right)
	\end{equation} 
	be the computed reduced pencil (\ref{eq:EsingularTruncatedPencil}), extracted from the transformed linearization (\ref{eq:FirstTransformJordanTest}).
	There exists small structured perturbation
	$$
	\delta\widetilde{\mathbb{B}}_{22} = 
	\left(\begin{array}{c?c}
	\begin{array}{c|c}
	\0 & \0\cr \hline
	\0 & -\delta{Q}^*_E
	\end{array} & \0_{2n\times (n+\widetilde{r}_E)} \cr \thickhline
	\0_{(n+\widetilde{r}_E)\times 2n} & \0_{n+\widetilde{r}_E}
	\end{array}\right),
	$$
	such that $\widetilde{\mathbb{A}}_{22}-\lambda (\widetilde{\mathbb{B}}_{22}+\delta \widetilde{\mathbb{B}}_{22})$ 
corresponds to an exact reduced linearization of a quartic matrix polynomial
	$$\lambda^4 A + \lambda^3 B + \lambda^2 (C+\delta C) + \lambda(D+\delta D) + (E+\delta E + \Delta E)$$
	with at least $n-\widetilde{r}_E$ zero eigenvalues,
	where, for all $i=1,\ldots, n$, 
	\begin{equation}\label{eq:dMdCdK:cols}
	\|\delta C(:,i)\|_2 \leq \epsilon_{C} \|C(:,i)\|_2,\;\;	\|\delta D(:,i)\|_2 \leq \epsilon_D \|D(:,i)\|_2,\;\; \|\delta E(:,i)\|_2 \leq \epsilon_{qr} \|E(:,i)\|_2 ,
	\end{equation}	
	and the truncation error from the determination of the numerical rank of $E$ is\footnote{See Remark \ref{REM-drop-off}.} 
	\begin{equation}
	\max_{j=1:n-k}\|(\Delta E)\widetilde{\Pi}_E(:,k+j)\|_2 \leq \tau \min_{i=1:k}\| (E+\delta E)\widetilde{\Pi}_E(:,i)\|_2; \;\;
	(\Delta E)\widetilde{\Pi}_E(:,1:k) = \0_{n,k}.
	\end{equation}
Here $\epsilon_C$, $\epsilon_D$, $\epsilon_{qr}$ are bounded by a moderate function of $n$ times the machine precision $\roff$, and $\tau$ is  prescribed threshold parameter.
\end{proposition}
\begin{proof}
The computed QR factorization of $E$, $E \widetilde{\Pi}_E\approx \widetilde{Q}_E \left(\begin{smallmatrix}\widetilde{R}_E\cr \0\end{smallmatrix}\right)$ can be represented as $(E+\delta E + \Delta E)\widetilde{\Pi}_E = \widehat{Q}_E \left(\begin{smallmatrix}\widetilde{R}_E\cr \0\end{smallmatrix}\right)$, where $\widehat{Q}_E$ is exactly unitary and $\| \widetilde{Q}_E - \widehat{Q}_E\|_F \leq \epsilon_{qr}$; the backward error $\delta E$ is induced by rounding errors during the factorization, and $\Delta E$ is the truncation error from the numerical rank. If we set $\delta Q_E = \widetilde{Q}_E - \widehat{Q}_E$, then $\widetilde{Q}_E = \widehat{Q}_E (\Id_n + \widehat{Q}_E^*\delta Q_E) = (\Id_n + \delta Q_E \widehat{Q}_E^*)\widehat{Q}_E$.
We can also write
$(E+\delta_1 E + \Delta E)\widetilde{\Pi}_E = \widetilde{Q}_E \left(\begin{smallmatrix}\widetilde{R}_E\cr \0\end{smallmatrix}\right)$, where 
$\delta_1 E = \delta E + \delta{Q}_E \left(\begin{smallmatrix}\widetilde{R}_E\cr \0\end{smallmatrix}\right)$, and thus
$\begin{pmatrix}\widetilde{R}_E\cr \0\end{pmatrix}\widetilde{\Pi}_E^T =
\widehat{Q}_E^* (E+\delta E + \Delta E) = \widetilde{Q}_E^{-1} (E+\delta_1 E + \Delta E) .$

There is an important subtlety here, and it is instructive to discuss it in more detail. In the actually computed matrix $\widetilde{\mathbb{B}}_{22}$, stored in the computer memory, one of its blocks is the numerically computed numerically orthogonal $\widetilde{Q}_E$. The backward stability of the QR factorization is usually stated in terms of an exactly unitary matrix $\widehat{Q}_E$, which is an inaccessible object as it is artificially constructed in the proof of backward stability. This is motivated by the desire to be able to say that we have computed the exact QR factorization of a nearby matrix. 
The matrices $\widetilde{X}$ and $\widetilde{Y}$ are computed by using the floating point matrix $\widetilde{Q}_E$, possibly implicitly as in the LAPACK subroutine \texttt{xORMQR}, or by explicit matrix multiply (\texttt{xGEMM} from BLAS) using explicitly formed $\widetilde{Q}_E$, using \texttt{xORGQR} (LAPACK).
The computed $\widetilde{X}$, $\widetilde{Y}$ can be represented as 
%
\begin{eqnarray*}
\!\!\!\!\!\!\! computed(\widetilde{Q}_E^* D) &=& \widetilde{Q}_E^* D + \delta_0 {D} = \widetilde{Q}_E^* (D+\delta D),\;\;\delta D = \widetilde{Q}_E^{-*}\delta_0 D,\;\;
|\delta_0 {D}| \leq \epsilon |\widetilde{Q}_E^*| |D|, \\
&=& \widehat{Q}_E^* (D + \Delta D),\;\;\Delta D = \delta D + \widehat{Q}_E (\delta Q_E)^* D + \widehat{Q}_E (\delta Q_E)^* \delta D,\;\epsilon\leq O(n)\roff ; \\ 
\!\!\!\!\!\!\! computed(\widetilde{Q}_E^* C) &=& \widetilde{Q}_E^* C + \delta_0 {C} = \widetilde{Q}_E^* (C+\delta C),\;\;\delta C = \widetilde{Q}_E^{-*}\delta_0 C,\;\;
|\delta_0 {C}| \leq \epsilon |\widetilde{Q}_E^*| |C|, \\
&=& \widehat{Q}_E^* (C + \Delta C),\;\;\Delta C = \delta C + \widehat{Q}_E (\delta Q_E)^* C + \widehat{Q}_E (\delta Q_E)^* \delta C .
\end{eqnarray*}
%
%
  On the other hand, the unit blocks $\Id_n\oplus \Id_{\widetilde{r}_E}$ in $\widetilde{\mathbb{A}}_{22}$ and $\Id_{n+\widetilde{r}_E}$ in $\widetilde{\mathbb{B}}_{22}$ assume exact orthogonality of $\widetilde{Q}_E$, which is not feasible in finite precision arithmetic.
If we set $\Delta_{\Sigma_1}E=\delta_1 E + \Delta E$, then we can represent the computed linearization (\ref{eq:FirstTransformJordanTest}) as 
\begin{eqnarray*}
	\!\!\!&&\!\!\! \left(\begin{smallmatrix}\Id_n & \0&\0 &\0 \cr \0 & \widetilde{Q}_E^* & \0 & \0 \cr \0 & \0 & \Id_n & \0 \cr \0 & \0 & \0 & \widetilde{Q}_E^{-1} \end{smallmatrix}\right)\! \left\{ \left(\begin{smallmatrix}
		B & \0_n & -\Id_n & \0 \\
		D+\delta D & \0_n & \0_n & -\Id_n \\
		0 & -\Id_n & \0_n & \0_n \\
		E+\Delta_{\Sigma_1}E & \0_n & \0_n & \0_n 
	\end{smallmatrix}\right)- \lambda \left(\begin{smallmatrix}
		-A & \0_n & \0_n & \0_n\\
		-(C+\delta C) & -\Id_n & \0_n & \0_n\\
		\0_n & \0_n & -\Id_n & \0_n \\
		\0_n & \0_n & \0_n & -\Id_n
	\end{smallmatrix}\right)\right\}\! \left(\begin{smallmatrix}\Id_n & \0&\0 &\0 \cr \0 & \Id_n & \0 & \0 \cr \0 & \0 & \Id_n & \0 \cr \0 & \0 & \0 & \widetilde{Q}_E \end{smallmatrix}\right) \\
	&& =
	{\small\left(\begin{array}{c?c}
			\begin{array}{c|c}
				B & \0\cr \hline
				\widetilde{X} & \0
			\end{array} & \begin{array}{c}
			\begin{array}{c|c} - \Id_{n} & \0 \cr\hline \0 & -\Id_n\end{array}
		\end{array} \cr \thickhline
		\begin{array}{c|c}
			0 & -\Id_n \cr \hline
			\widetilde{R}_E \widetilde{\Pi}^T_E & \0 \cr
			\0 & \0
		\end{array} & \0_{2n}
	\end{array}\right) + 
	\left(\begin{array}{c?c}
		\begin{array}{c|c}
			\0 & \0\cr \hline
			\0 & \0
		\end{array} & \begin{array}{c}
		\begin{array}{c|c} \0 & \0 \cr\hline \0 & \Xi\end{array}
	\end{array} \cr \thickhline
	\begin{array}{c|c}
		0 & \0 \cr \hline
		\0 & \0 \cr
		\0 & \0
	\end{array} & \0_{2n}
\end{array}\right)
	 - \lambda \left(\begin{array}{c?c}
	\begin{array}{c|c}
		-A & \0\cr \hline
		-\widetilde{Y} & -\widetilde{Q}^*_E
	\end{array} & \0_{2n} \cr \thickhline
	\0_{2n} & -\Id_{2n}
\end{array}\right)}.
\end{eqnarray*}
Now we see at the block position $(2,4)$ in the left matrix, $\widetilde{Q}_E^* (-\Id_n)\widetilde{Q}_E = -\Id_n + \Xi\neq -\Id_n$. Hence, (\ref{eq:LGEVP:n+rE:computed}) can be justified by a mixed stability scenario -- if the computed pencil is changed by $\|\Xi\|_2 \leq \epsilon_{qr}$ to restore identity at the $(2,4)$ position in the left matrix, then it can be interpreted as an exact transformation of a slightly changed initial pencil.

Alternatively, we can set $\Delta_{\Sigma}E=\delta E + \Delta E$ and model (\ref{eq:FirstTransformJordanTest}) as 
\begin{eqnarray}
	\!\!\!&&\!\!\! \left(\begin{smallmatrix}\Id_n & \0&\0 &\0 \cr \0 & \widehat{Q}_E^* & \0 & \0 \cr \0 & \0 & \Id_n & \0 \cr \0 & \0 & \0 & \widehat{Q}_E^{*} \end{smallmatrix}\right)\! \left\{ \left(\begin{smallmatrix}
		B & \0_n & -\Id_n & \0 \\
		D+\Delta D & \0_n & \0_n & -\Id_n \\
		0 & -\Id_n & \0_n & \0_n \\
		E+\Delta_{\Sigma}E & \0_n & \0_n & \0_n 
	\end{smallmatrix}\right)- \lambda \left(\begin{smallmatrix}
	-A & \0_n & \0_n & \0_n\\
	-(C+\Delta C) & -\Id_n & \0_n & \0_n\\
	\0_n & \0_n & -\Id_n & \0_n \\
	\0_n & \0_n & \0_n & -\Id_n
\end{smallmatrix}\right)\right\}\! \left(\begin{smallmatrix}\Id_n & \0&\0 &\0 \cr \0 & \Id_n & \0 & \0 \cr \0 & \0 & \Id_n & \0 \cr \0 & \0 & \0 & \widehat{Q}_E \end{smallmatrix}\right) \nonumber \\
&&\!\! =\!
{\small\left(\begin{array}{c?c}
		\begin{array}{c|c}
			B & \0\cr \hline
			\widetilde{X} & \0
		\end{array} & \begin{array}{c}
		\begin{array}{c|c} - \Id_{n} & \0 \cr\hline \0 & -\Id_n\end{array}
	\end{array} \cr \thickhline
	\begin{array}{c|c}
		0 & -\Id_n \cr \hline
		\widetilde{R}_E \widetilde{\Pi}^T_E & \0 \cr
		\0 & \0
	\end{array} & \0_{2n}
\end{array}\right)\! - \!\lambda \left\{\! \left(\begin{array}{c?c}
\begin{array}{c|c}
	-A & \0\cr \hline
	-\widetilde{Y} & -\widetilde{Q}^*_E
\end{array} & \0_{2n} \cr \thickhline
\0_{2n} & -\Id_{2n}
\end{array}\right) \! + \!
\left(\begin{array}{c?c}
\begin{array}{c|c}
	\0 & \0\cr \hline
	\0 & -\delta{Q}^*_E
\end{array} & \0_{2n} \cr \thickhline
\0_{2n} & -\0_{2n}
\end{array}\right)\!\right\}}\! . \label{eq-back-model-2}
\end{eqnarray}
In this case, the $(2,2)$ block in the right matrix in (\ref{eq:LGEVP:n+rE:computed}) should be changed from $-\widetilde{Q}_E^*$ to $-\widehat{Q}_E^*$, by adding $\delta Q_{E}^*$,  to establish exact equivalence with a slightly perturbed initial pencil. 
\hfill \end{proof}
\begin{remark}
The forward error introduced in (\ref{eq-back-model-2}) (thus making the model of the analysis of mixed forward-backward type) is due to the fact that in finite precision computation unitarity/orthogonality cannot be guaranteed.\footnote{For that reason the QR factorization can only be mixed stable, and in general it is not backward stable.} Note that this error is localized to one block of the linearization; its structure can be easily seen from the backward analysis of the e.g. Householder QR factorization.	
\end{remark}
We now consider the first two steps and show that the algorithm remains mixed stable. The proof is technically more involved, but it is important to see how the reduced linear pencil after small forward modification exactly corresponds to a quartic pencil with backward errors in the initial coefficient matrices. Also, the proof nicely illustrates the benefits of well scaled data.

\begin{theorem} Assume the notation of Proposition \ref{PROP1}, and let
{\small	
\begin{eqnarray}\label{eq:ABOneSingularSecondStep-computed}
\widetilde{\widehat{P}_2\mathbb{A}_{22}} &=&  \left(\!\!\begin{array}{c?c}
\begin{array}{c|c}
B & \0 \cr \hline
\widehat{Q}^*_{E,1}(D+\Delta D) & \0 \\
\0 & -\Id_n
\end{array} & \begin{array}{c| c}
-\Id_n & \cr \hline
& -\Id_{\widetilde{r}_E} \\
\0 & \0
\end{array} \\ \thickhline
\begin{array}{c | c}
\widetilde{R}_{A_{22}}\widetilde{\Pi}^T_{A_{22}} & \0 \cr 
\0 & \0
\end{array} & \begin{array}{c} \0_{\widetilde{r}_2\times (n+\widetilde{r}_E)}\cr\hline
\0_{(n-\widetilde{r}_2)\times (n+\widetilde{r}_E)}\end{array}
\end{array}\!\!\right)\!, 
\end{eqnarray}
\begin{eqnarray}
\widetilde{\widehat{P}_2\mathbb{B}_{22}}  &=&  
\left(\begin{array}{c?c}
\begin{array}{c|c}
-A & \0\\ \hline
\begin{array}{c}- \widehat{Q}_{E,1}^*(C+\Delta C)\cr
\0_n \end{array} & \begin{array}{c} -\widetilde{Q}_{E,1}^* \cr \0_n \end{array}
\end{array} & \begin{array}{c} \0_{n+\widetilde{r}_E} \cr\hline \0_{\widetilde{r}_E\times (n+\widetilde{r}_E)} \cr\hline \begin{array}{c|c}-\Id_n & \0_{n\times\widetilde{r}_E} \end{array}\end{array} \\ \thickhline
\begin{array}{c|c} -\widetilde{N}_{[1]} & -\widetilde{N}_{[2]}\end{array} & \begin{array}{c|c} \widetilde{N}_{[3]} & \widetilde{N}_{[4]}\end{array}
\end{array}\right)
%
\end{eqnarray}	
}
be the computed version of (\ref{eq:ABOneSingularSecondStep}). There exist small structured forward perturbation  
$$
\mathcal{F}_{\mathbb{B}_{22}}=\left(\begin{array}{c?c}
\begin{array}{c|c}
\0 & \0\\ \hline
\begin{array}{c} \0\cr
\0_n \end{array} & \begin{array}{c} -\delta\widetilde{Q}_{E,1}^* \cr \0_n \end{array}
\end{array} & \begin{array}{c} \0_{n+\widetilde{r}_E} \cr\hline \0_{\widetilde{r}_E\times (n+\widetilde{r}_E)} \cr\hline \begin{array}{c|c} \0_n & \0_{n\times\widetilde{r}_E} \end{array}\end{array} \\ \thickhline
\begin{array}{c|c} \Delta\widetilde{N}_{[1]} & \Delta\widetilde{N}_{[2]}\end{array} & \begin{array}{c|c} \0 & \Delta\widetilde{N}_{[4]}\end{array} 
\end{array}\right),\;\;\|\delta\widetilde{Q}_{E,1}^*\|_2\leq \epsilon_{qr},\;\;
\|\Delta\widetilde{N}_{[4]} \|_2 \leq \epsilon_{qr},
$$
of $\widetilde{\widehat{P}_2\mathbb{B}_{22}}$, and backward errors $\Delta C$, $\Delta_{\Sigma} D$, $\Delta_{\Sigma} E$  such that $\widetilde{\widehat{P}_2\mathbb{A}_{22}}-\lambda (\widetilde{\widehat{P}_2\mathbb{B}_{22}} + \mathcal{F}_{\mathbb{B}_{22}})$
corresponds to an exactly reduced linearization of a quartic matrix polynomial 
$$\lambda^4 A + \lambda^3 B + \lambda^2 (C+\Delta C) + \lambda(D+\Delta_{\Sigma} D) + (E+\Delta_{\Sigma} E),$$
with the exact transformation given in 
(\ref{eq:backw-golbal-A}) and (\ref{eq:backw-golbal-B}) below. Under mild technical assumption (on the size of $n\roff$), {we have, for each column index $i$,
\begin{equation}
  \|\Delta\widetilde{N}_{[1]}(:,i)\|_2 \leq f_1(n)\roff \|\widetilde{N}_{[1]}(:,i)\|_2,\;\;
 \|\Delta\widetilde{N}_{[2]}\|_2 \leq f_2(n)\roff  \|\widetilde{N}_{[2]}\|_2,
\end{equation}
where $f_1(n)$, $f_2(n)$ are mildly growing functions.}
%
%
{Further, with $\Delta C$, $\Delta D$, $\delta E$, $\Delta E$ as in Proposition \ref{PROP1}, it holds that   $\Delta_{\Sigma} D=\Delta D + \widehat{Q}_{E,2}\Gamma_1$ and $\Delta_{\Sigma} E = \delta E+\Delta E+\widehat{Q}_{E,1}\Gamma_2$, where  
\begin{equation}
    \left\| \begin{pmatrix}
\Gamma_1\\
\Gamma_2
\end{pmatrix} (:,i) \right\|_2 \leq \epsilon_{qr} \left\| \begin{pmatrix}
(D+\Delta D)\\ \hline
E+\delta E + \Delta E
\end{pmatrix} (:,i) \right\|_2.
\end{equation}
The latter shows the benefits of well scaled and balanced $D$ and $E$ (\S \ref{SSS=param-scaling}, \S \ref{SSS=graded-m}, Remark \ref{REM=Balancing}).
}
\end{theorem}
\begin{proof}
We continue based on the details and the notation from the proof of Proposition \ref{PROP1}.
The next step is computation of the rank revealing factorization of the block matrix $\left(\begin{smallmatrix}
\widetilde{Q}_{E,2}^*(D+\delta D)\\ \hline
\widetilde{R}_E\widetilde{\Pi}_E^T
\end{smallmatrix}\right)$.
It is convenient to consider the left matrix in (\ref{eq:LGEVP:n+rE:computed}) with the relevant blocks already swapped (see (\ref{eq:ABOneSingularSecondStep}))
$$
{
\left( \begin{array}{c?c} 
\begin{array}{c|c}
B & \0\\ \hline
\begin{array}{c} \widehat{Q}_{E,1}^*(D+\Delta D)\cr \widehat{Q}_{E,2}^*(D+\Delta D)\end{array} & \begin{array}{c} \0\cr \0\end{array}
\end{array} & \begin{array}{c|c}
-\Id_n & \0 \\ \hline
\0 & -\Id_{\widetilde{r}_E}\\
\0 & \0
\end{array} \\ \thickhline
\begin{array}{c|c}
\0_n & -\Id_n \\ \hline
\widetilde{R}_E\widetilde{\Pi}^T_E & \0
\end{array} & \0_{n+\widetilde{r}_E}
\end{array}\right) \longrightarrow
\left( \begin{array}{c?c} 
\begin{array}{c|c}
B & \0\\ \hline
\begin{array}{c} \widehat{Q}_{E,1}^*(D+\Delta D)\cr \0\end{array} & \begin{array}{c} \0\cr -\Id_n\end{array}
\end{array} & \begin{array}{c|c}
-\Id_n & \0 \\ \hline
\0 & -\Id_{\widetilde{r}_E}\\
\0 & \0
\end{array} \\ \thickhline
\begin{array}{c|c}
\widehat{Q}_{E,2}^*(D+\Delta D) & \0 \\ \hline
\widetilde{R}_E\widetilde{\Pi}^T_E & \0
\end{array} & \0_{n\times (n+\widetilde{r}_E)}
\end{array}\right) .
}
$$
For the computed factors $\widetilde{\Pi}_{A_{22}},\widetilde{Q}_{A_{22}},\widetilde{R}_{A_{22}}$ it holds that
\begin{equation}\label{eq:QRA22}
\left[ \begin{pmatrix}
\widetilde{Q}_{E,2}^*(D+\delta D)\\ \hline
\widetilde{R}_E\widetilde{\Pi}_E^T
\end{pmatrix} + \begin{pmatrix}
\Gamma_1\\
\Gamma_2
\end{pmatrix} \right] \widetilde{\Pi}_{A_{22}} 
\equiv
\left[ \begin{pmatrix}
\widehat{Q}_{E,2}^*(D+\Delta D)\\ \hline
\widetilde{R}_E\widetilde{\Pi}_E^T
\end{pmatrix} + \begin{pmatrix}
\Gamma_1\\
\Gamma_2
\end{pmatrix} \right] \widetilde{\Pi}_{A_{22}} 
= \widehat{Q}_{A_{22}}\begin{pmatrix} \widetilde{R}_{A_{22}}\cr \0\end{pmatrix},
\end{equation}	 
where $\widehat{Q}_{A_{22}}$ is exactly unitary and $\widehat{Q}_{A_{22}} \approx \widetilde{Q}_{A_{22}}$, $\widetilde{R}_{A_{22}}$ is $\widetilde{r}_2\times n$ of full row rank,\footnote{The zero block beneath of $\widetilde{R}_{A_{22}}$ may be void.} and $\Gamma=\left(\begin{smallmatrix}\Gamma_1\cr\Gamma_2\end{smallmatrix}\right)$ is the backward error of the QR factorization that can be estimated by
\begin{equation}
\left\| \begin{pmatrix}
\Gamma_1\\
\Gamma_2
\end{pmatrix} (:,i) \right\|_2 \leq \epsilon_{qr} \left\| \begin{pmatrix}
\widehat{Q}_{E,2}^*(D+\Delta D)\\ \hline
\widetilde{R}_E\widetilde{\Pi}_E^T
\end{pmatrix} (:,i) \right\|_2.
\end{equation}
We can push $\Gamma_1$ and $\Gamma_2$ backward in $D$ and $E$, respectively, as follows. 
First, $\widehat{Q}_{E,2}^*(D+\Delta D)+\Gamma_1 = \widehat{Q}_{E,2}^*(D+\Delta D + \widehat{Q}_{E,2}\Gamma_1)$ and
$$
\widehat{Q}_E^* (D+\Delta D + \widehat{Q}_{E,2}\Gamma_1) = \begin{pmatrix}
\widehat{Q}_{E,1}^* (D+\Delta D) \cr
\widehat{Q}_{E,2}^* (D+\Delta D) + \Gamma_1 \end{pmatrix}.
$$
If $D$ and $E$ are so scaled that their norms are nearly of the same order, then $\Gamma_1$, $\Gamma_2$ will be, respectively, their relatively small perturbations. Further, an analogous conclusion holds also column-wise, which motivates scaling the initial data by diagonal matrices to equilibrate on the matrix elements level.\footnote{See Example \ref{NUMEX:balance-butterfly}.} 
Hence, if the additive perturbation $\Delta D$ is replaced with $\Delta_{\Sigma} D=\Delta D + \widehat{Q}_{E,2}\Gamma_1$, $\widetilde{X}_1$ remains unchanged, and $\widetilde{X}_2$ is precisely as in (\ref{eq:QRA22}). (Here $\widetilde{X}=\left(\begin{smallmatrix}\widetilde{X}_1\cr
\widetilde{X}_2\end{smallmatrix}\right)$.) Similarly,
{using $\delta E$, $\Delta E$ from Proposition \ref{PROP1}, we have }
\begin{equation}
\widehat{Q}_E^* (E + \delta E + \Delta E + \widehat{Q}_{E,1}\Gamma_2) = \begin{pmatrix} \widetilde{R}_E\widetilde{\Pi}_E^T + \Gamma_2 \cr \0\end{pmatrix}.
\end{equation}
Now define $\widetilde{\widehat{P}_2}=(\Id_{2n+\widetilde{r}_E}\oplus \widehat{Q}_{A_{22}}^*) \widetilde{\Pi}$ analogously to (\ref{eq:hatP2}).
The  matrix in (\ref{eq:ABOneSingularSecondStep-computed}) can be interpreted as an exact transformation of type (\ref{eq-back-model-2}), followed by the transformation of type (\ref{eq:ABOneSingularSecondStep}) with $\widetilde{\widehat{P}_2}$, but with initial matrices that are changed as $D\leadsto D+\Delta_{\Sigma} D$; $E\leadsto E + \Delta_{\Sigma}E$, $\Delta_{\Sigma}E=\delta E + \Delta E + \widehat{Q}_{E,1}\Gamma_2$. The transformation reads:

\begin{equation}\label{eq:backw-golbal-A}
\left(\begin{smallmatrix}\Id_{2n+\widetilde{r}_E} & \0 & \0 \cr
\0 & \widehat{Q}_{A_{22}}^* & \0 \cr
\0 & \0 & \Id_{n-\widetilde{r}_E} \end{smallmatrix}\right)
\left(\begin{smallmatrix} 
\widetilde{\Pi} & \0  \cr
\0 & \Id_{n-\widetilde{r}_E}  \end{smallmatrix}\right)
\left(\begin{smallmatrix}\Id_n & \0&\0 &\0 \cr \0 & \widehat{Q}_E^* & \0 & \0 \cr \0 & \0 & \Id_n & \0 \cr \0 & \0 & \0 & \widehat{Q}_E^{*} \end{smallmatrix}\right)\!  
\left(\begin{smallmatrix}
B & \0_n & -\Id_n & \0 \\
D+\Delta_{\Sigma} D & \0_n & \0_n & -\Id_n \\
0 & -\Id_n & \0_n & \0_n \\
E+\Delta_{\Sigma}E & \0_n & \0_n & \0_n 
\end{smallmatrix}\right)
\left(\begin{smallmatrix}\Id_n & \0&\0 &\0 \cr \0 & \Id_n & \0 & \0 \cr \0 & \0 & \Id_n & \0 \cr \0 & \0 & \0 & \widehat{Q}_E \end{smallmatrix}\right)
\end{equation}
Consider now the second coefficient. The block swapping on the right matrix (pre-multiplication of the rows  by the $(3n+\widetilde{r}_E)\times (3n+\widetilde{r}_E)$  permutation matrix $\widetilde{\Pi}$) reads
$$
{\small 
\left(\begin{array}{c?c}
\begin{array}{c|c}
-A & \0\\ \hline
 \begin{array}{c}- \widehat{Q}_{E,1}^*(C+\Delta C)\cr
-\widehat{Q}_{E,2}^*(C+\Delta C) \end{array} & \begin{array}{c} -\widetilde{Q}_{E,1}^* \cr -\widetilde{Q}_{E,2}^* \end{array}
\end{array} & \0_{2n \times (n+\widetilde{r}_E)} \\ \thickhline
\0_{(n+\widetilde{r}_E)\times 2n} & -\Id_{n+\widetilde{r}_E}
\end{array}\right)\!\! \rightarrow \!\!
\left(\begin{array}{c?c}
\begin{array}{c|c}
-A & \0\\ \hline
\begin{array}{c}- \widehat{Q}_{E,1}^*(C+\Delta C)\cr
\0_n \end{array} & \begin{array}{c} -\widetilde{Q}_{E,1}^* \cr \0_n \end{array}
\end{array} & \begin{array}{c} \0_{n+\widetilde{r}_E} \cr\hline \0_{\widetilde{r}_E\times (n+\widetilde{r}_E)} \cr\hline \begin{array}{c|c}-\Id_n & \0_{n\times\widetilde{r}_E} \end{array}\end{array} \\ \thickhline
\begin{array}{c|c} -\widehat{Q}_{E,2}^*(C+\Delta C) & -\widetilde{Q}_{E,2}^* \cr \0 & \0 \end{array}& \begin{array}{c|c} \0 & \0 \cr \0 & -\Id_{\widetilde{r}_E}\end{array}
\end{array}\right)
}
$$
Recall, $\widetilde{Y}=computed(\widetilde{Q}_E^* C)=\widehat{Q}_E^* (C+\Delta C)$; introduce block-row partition $\widetilde{Y}=\left(\begin{smallmatrix}\widetilde{Y}_1\cr
\widetilde{Y}_2\end{smallmatrix}\right)$ with $\widetilde{Y}_2=\widehat{Q}_{E,2}^*(C+\Delta C)$. Similarly, introduce block-column partitions $\widetilde{Q}_{A_{22}}^* = (\widetilde{\Omega}_1\; \widetilde{\Omega}_2)$, $\widehat{Q}_{A_{22}}^* = (\widehat{\Omega}_1\; \widehat{\Omega}_2)$.
The last column block $\widetilde{N}_{[4]}$ in the matrix
\begin{equation}
\widetilde{N}\!\! =\!\! computed(\!\widetilde{Q}_{A_{22}}^*\! \begin{pmatrix} -\widehat{Q}_{E,2}^*(C+\Delta C) & -\widetilde{Q}_{E,2}^* & \0_{(n-\widetilde{r}_E)\times n} & \0 \cr
\0_{\widetilde{r}_E\times n} & \0_{\widetilde{r}_E\times n} & \0_{\widetilde{r}_E\times n} & -\Id_{\widetilde{r}_E} \end{pmatrix}\!) \!\!=\!\!\begin{pmatrix} -\widetilde{N}_{[1]} & -\widetilde{N}_{[2]}  & \widetilde{N}_{[3]} & \widetilde{N}_{[4]}\end{pmatrix}
\end{equation}
is simply $-\widetilde{\Omega}_2$. Since we used $\widehat{Q}_{A_{22}}$ in the backward error analysis of the left-hand matrix, here we will have to use a mixed error analysis: $-\widetilde{\Omega}_2$ will be changed by a forward error into $-\widehat{\Omega}_2$. Recall that our model of the analysis (using exactly unitary instead of the computed numerically unitary matrices) will also require small forward perturbation to change $-\widetilde{Q}_{E,1}^*$ into $-\widehat{Q}_{E,1}^*$.

Consider now the first two blocks in $\widetilde{N}$. 
\begin{equation}
\widetilde{N}_{[1]}=computed(\widetilde{\Omega}_1 \widetilde{Y}_2) = \widetilde{\Omega}_1 \widetilde{Y}_2 + \delta\widetilde{N}_{[1]} = \widehat{\Omega}_1 \widetilde{Y}_2 + \delta\widetilde{\Omega}_1 \widetilde{Y}_2 + \delta\widetilde{N}_{[1]} ,\;\; |\delta\widetilde{N}_{[1]}|\leq \epsilon 
|\widetilde{\Omega}_1| | \widetilde{Y}_2| .
\end{equation}
{(Here $\epsilon$ estimates the backward error for matrix multiplication, $0\leq\epsilon\leq O(n)\roff$.)}
In this block too, we will commit a forward error and replace it with 
\begin{equation}
\widehat{Q}_{A_{22}}^* \begin{pmatrix} \widehat{Q}_{E,2}^*(C+\Delta C) \cr \0_{\widetilde{r}_E\times n}\end{pmatrix} =  \widehat{\Omega}_1 \widetilde{Y}_2 = \widetilde{N}_{[1]} -\Delta\widetilde{N}_{[1]},\;\; \Delta\widetilde{N}_{[1]} =  \delta\widetilde{\Omega}_1 \widetilde{Y}_2 + \delta\widetilde{N}_{[1]} .
\end{equation}
To estimate this forward change we first note that, for each column index $i$,
$$
\|\widetilde{Y}_2(:,i)\|_2 \leq \frac{\|\widetilde{N}_{[1]}(:,i)\|_2}{1-\|\delta\widetilde{\Omega}_1\|_2-\epsilon \| |\widetilde{\Omega}_1| \|_2} .
$$
Hence 
\begin{equation}\label{eq:DN1}
\|\delta\widetilde{\Omega}_1 \widetilde{Y}_2(:,i)\|_2 \leq \frac{\|\delta\widetilde{\Omega}_1\|_2\|\widetilde{N}_{[1]}(:,i)\|_2}{1-\|\delta\widetilde{\Omega}_1\|_2-\epsilon \| |\widetilde{\Omega}_1| \|_2},\;\;
\|\delta\widetilde{N}_{[1]}(:,i)\|_2 \leq \frac{\epsilon \| |\widetilde{\Omega}_1|\|_2\|\widetilde{N}_{[1]}(:,i)\|_2}{1-\|\delta\widetilde{\Omega}_1\|_2-\epsilon \| |\widetilde{\Omega}_1| \|_2} , 
\end{equation}
and we conclude that $\Delta\widetilde{N}_{[1]}$ is a column-wise small perturbation of $\widetilde{N}_{[1]}$.
Computation of $\widetilde{N}_{[2]}=computed(\widetilde{\Omega}_1\widetilde{Q}_{E,2}^*)$ is analogous, but for the purpose of mixed stability interpretation, $\widetilde{Q}_{E,2}^*$ has to be replaced with $\widehat{Q}_{E,2}^* = \widetilde{Q}_{E,2}^* - \delta\widetilde{Q}_{E,2}^*$, which yields
\begin{equation}
\widetilde{N}_{[2]}= \widehat{\Omega}_1 \widehat{Q}_{E,2}^* + 
\widehat{\Omega}_1 \delta\widetilde{Q}_{E,2}^* + \underbrace{\delta\widetilde{\Omega}_1\widehat{Q}_{E,2}^* + \delta\widetilde{\Omega}_1\delta\widetilde{Q}_{E,2}^*}_{\delta\widetilde{\Omega}_1\widetilde{Q}_{E,2}^*} + \delta\widetilde{N}_{[2]},\;\; |\delta\widetilde{N}_{[2]}|\leq\epsilon |\widetilde{\Omega}_1| |\widetilde{Q}_{E,2}^*| .
\end{equation}
Hence, if the computed right-hand matrix is changed by a forward perturbation as
\begin{eqnarray}
&& \left(\begin{array}{c?c}
\begin{array}{c|c}
-A & \0\\ \hline
\begin{array}{c}- \widehat{Q}_{E,1}^*(C+\Delta C)\cr
\0_n \end{array} & \begin{array}{c} -\widetilde{Q}_{E,1}^* \cr \0_n \end{array}
\end{array} & \begin{array}{c} \0_{n+\widetilde{r}_E} \cr\hline \0_{\widetilde{r}_E\times (n+\widetilde{r}_E)} \cr\hline \begin{array}{c|c}-\Id_n & \0_{n\times\widetilde{r}_E} \end{array}\end{array} \\ \thickhline
 \begin{array}{c|c} -\widetilde{N}_{[1]} & -\widetilde{N}_{[2]}\end{array} & \begin{array}{c|c} \widetilde{N}_{[3]} & \widetilde{N}_{[4]}\end{array}
\end{array}\right) + 
\left(\begin{array}{c?c}
\begin{array}{c|c}
\0 & \0\\ \hline
\begin{array}{c} \0\cr
\0_n \end{array} & \begin{array}{c} -\delta\widetilde{Q}_{E,1}^* \cr \0_n \end{array}
\end{array} & \begin{array}{c} \0_{n+\widetilde{r}_E} \cr\hline \0_{\widetilde{r}_E\times (n+\widetilde{r}_E)} \cr\hline \begin{array}{c|c} \0_n & \0_{n\times\widetilde{r}_E} \end{array}\end{array} \\ \thickhline
\begin{array}{c|c} \Delta\widetilde{N}_{[1]} & \Delta\widetilde{N}_{[2]}\end{array} & \begin{array}{c|c} \0 & \Delta\widetilde{N}_{[4]}\end{array} 
\end{array}\right)  \nonumber \\
&& = \left(\begin{array}{c?c}
\begin{array}{c|c}
-A & \0\\ \hline
\begin{array}{c}- \widehat{Q}_{E,1}^*(C+\Delta C)\cr
\0_n \end{array} & \begin{array}{c} -\widehat{Q}_{E,1}^* \cr \0_n \end{array}
\end{array} & \begin{array}{c} \0_{n+\widetilde{r}_E} \cr\hline \0_{\widetilde{r}_E\times (n+\widetilde{r}_E)} \cr\hline \begin{array}{c|c}-\Id_n & \0_{n\times\widetilde{r}_E} \end{array}\end{array} \\ \thickhline
\begin{array}{c|c} \widehat{Q}_{A_{22}}^* \begin{pmatrix} -\widehat{Q}_{E,2}^*(C+\Delta C) \cr \0_{\widetilde{r}_E\times n}\end{pmatrix} & \widehat{Q}_{A_{22}}^* \begin{pmatrix} -\widehat{Q}_{E,2}^* \cr \0_{\widetilde{r}_E\times n}\end{pmatrix} \end{array} & \begin{array}{c|c} \widehat{Q}_{A_{22}}^*\left(\begin{array}{c}\0\cr\0\end{array}\right) & \widehat{Q}_{A_{22}}^* \left(\begin{array}{c}\0\cr -\Id_{\widetilde{r}_E}\end{array}\right)\end{array}
\end{array}\right) ,
\end{eqnarray}
the resulting matrix is the $(3n+\widetilde{r}_E)\times (3n+\widetilde{r}_E)$ main submatrix of  
\begin{equation}\label{eq:backw-golbal-B}
\left(\begin{smallmatrix}\Id_{2n+\widetilde{r}_E} & \0 & \0 \cr
\0 & \widehat{Q}_{A_{22}}^* & \0 \cr
\0 & \0 & \Id_{n-\widetilde{r}_E} \end{smallmatrix}\right)
\left(\begin{smallmatrix} 
\widetilde{\Pi} & \0  \cr
 \0 & \Id_{n-\widetilde{r}_E}  \end{smallmatrix}\right)
\left(\begin{smallmatrix}\Id_n & \0&\0 &\0 \cr \0 & \widehat{Q}_E^* & \0 & \0 \cr \0 & \0 & \Id_n & \0 \cr \0 & \0 & \0 & \widehat{Q}_E^{*} \end{smallmatrix}\right)\!  \left(\begin{smallmatrix}
-A & \0_n & \0_n & \0_n\\
-(C+\Delta C) & -\Id_n & \0_n & \0_n\\
\0_n & \0_n & -\Id_n & \0_n \\
\0_n & \0_n & \0_n & -\Id_n
\end{smallmatrix}\right) \left(\begin{smallmatrix}\Id_n & \0&\0 &\0 \cr \0 & \Id_n & \0 & \0 \cr \0 & \0 & \Id_n & \0 \cr \0 & \0 & \0 & \widehat{Q}_E \end{smallmatrix}\right).
\end{equation}
\end{proof}

\section{Numerical examples}\label{S=KVATREIG-NUMEX}
In this section, we present numerical examples and compare our new algorithm \texttt{kvarteig} with \texttt{polyeig} from MATLAB,  \texttt{quadeig} \cite{Hammarling:QUADEIG} and \texttt{kvadeig} (including \texttt{balanced\_kvadeig}\footnote{\texttt{balanced\_kvadeig} denotes the \texttt{kvadeig} algorithm enhanced with balancing by diagonal scaling matrices, see Remark \ref{REM=Balancing}.}) \cite{KVADeig-arxiv} applied to the linearization (\ref{eq:FinalLinearization}) and the quadratification (\ref{eq:SecondComapnionFormGrade2}), respectively. 

Our goal is to illustrate the potential of the techniques introduced in \texttt{kvadeig} and \texttt{kvarteig}, and to motivate further development. {We in particular stress the benefits of additional balancing of the coefficient matrices, that is used in combination with the well developed parameter scaling. The analysis of backward errors in \S \ref{S=BackError} indicates, and numerical experiments in this section provide empirical evidence of importance of well balanced data. Balancing is applicable, \emph{mutatis mutandis},  to other solvers as well.}

All test examples are taken from the NLEVP benchmark collection \cite{betcke2010nlevp}. 
{
In all examples and figures shown in this section, the eigenpairs are indexed so that the eigenvalues are sorted increasingly in modulus.
}
\begin{example}\label{NUMEX-1}
We first test \texttt{kvarteig} on three examples with the default input values: \texttt{butterfly} ($n = 64$); \texttt{orr\_sommerfeld} ($n=64$);  \texttt{planar waveguide} ($n=129$). The results are tested using the norm-wise (residual) backward error {\eqref{eq:back-err-0}}.

In the first run of the experiment, 
\texttt{polyeig, quadeig}, and  both variants of \texttt{kvadeig} worked with raw data $A$, $B$, $C$, $D$, $E$, and the parameter scaling is applied in \texttt{quadeig, kvadeig} only to the quadratic pencil $\lambda^2\Mbb+\lambda\Cbb+\Kbb$ from (\ref{eq:SecondComapnionFormGrade2}). In \texttt{balanced kvadeig}, parameter scaling is combined with diagonal balancing of $\Mbb, \Cbb, \Kbb$. For the sake of the experiment, \texttt{kvarteig} worked 
in two modes: \emph{(i)} with parameter scaling (designated as \texttt{kvarteig}); and \emph{(ii)} without parameter scaling, but with the diagonal balancing switched on (designated as \texttt{balanced kvarteig (-s)}). 

{Switching the scaling off  may simulate the case of an unsuccessful parameter scaling of the initial matrix  coefficients.}
Also, this may serve as a simulation of a genuine quadratic problem in which the coefficients are composed of blocks with different parameter dependencies, possibly on different scales -- then parameter scaling cannot resolve different scales inside $\Mbb$, $\Cbb$, $\Kbb$.  Hence, this is primarily a test of the quadratic solvers as potential tools for quadratification based solution of quartic problems. We refer the reader to \S \ref{SSS=graded-m},  Remark \ref{REM=Balancing}, \S  \ref{SSS=quadr-not-enough}, and \cite{KVADeig-arxiv}. Further, with \texttt{balanced kvarteig (-s)} we want to check whether diagonal balancing can make a difference in the absence (or failure) of the parameter scaling.

The extreme values of $\eta$ over all computed  right eigenpairs are given in Table \ref{tt:BEKAVRTeig}:
\vspace{-3mm}
\begin{table}[H]
	\centering
	\caption{Comparison of backward errors for \texttt{polyeig}, \texttt{quadeig}, \texttt{kvadeig} and \texttt{kvarteig}}
	\resizebox{\columnwidth}{!}{
		\begin{tabular}{|l|| c|c|c| c|c|c|}
			\hline
			& \multicolumn{2}{c|}{\texttt{\textbf{butterfly}}} &
			\multicolumn{2}{c|}{\texttt{\textbf{orr\_sommerfeld}}} & \multicolumn{2}{c|}{\texttt{\textbf{planar waveguide}}}\\
			\hline
			\multicolumn{1}{|c||}{\textbf{Algorithm}} & $\min\eta$ & $\max \eta$ &  $\min\eta$ & $\max \eta$ &$\min\eta$ & $\max \eta$\\ \hline
			\texttt{polyeig} & \texttt{2.04e-016} & \texttt{8.61e-016} &\texttt{1.36e-017}  & \texttt{8.01e-006} & \texttt{1.60e-016} &  \texttt{3.08e-012} \\ \hline
			\texttt{quadeig} & \texttt{6.56e-017} & \texttt{2.03e-015} & \texttt{6.11e-015} & \texttt{4.07e-004} & \texttt{4.99e-016} & \texttt{2.03e-009} \\ \hline
			\texttt{kvadeig} & \texttt{6.56e-017} & \texttt{2.03e-015} & \texttt{6.25e-021} & \texttt{2.12e-007} & \texttt{4.75e-016} & \texttt{1.67e-009}\\
			\hline
			\texttt{balanced kvadeig} & \texttt{6.56e-017} & \texttt{2.03e-015} & \texttt{2.81e-021} & \texttt{2.06e-012} & \texttt{1.49e-016} & \texttt{2.32e-012}\\
			\hline
			\texttt{balanced kvarteig (-s)} & \texttt{3.18e-017} & \texttt{9.11e-016} & \texttt{3.40e-021} & \texttt{5.25e-008} & \texttt{3.20e-016} & \texttt{5.16e-012}\\
			\hline
			\texttt{kvarteig} & \texttt{5.84e-017} & \texttt{1.13e-015} & \texttt{6.37e-021} & \texttt{1.76e-015} & \texttt{4.32e-016} & \texttt{1.75e-013}\\
			\hline
	\end{tabular}}
		\label{tt:BEKAVRTeig}
	\end{table}	
\vspace{-3mm}
{Note that \texttt{quadeig} and \texttt{kvadeig} had relatively large relative errors for \texttt{orr\_sommerfeld} and \texttt{planar waveguide}, while \texttt{balanced kvadeig} performed well despite the fact that it received unscaled original matrices. Although contrived, this example illustrates the main point well -- parameter scaling (here applied to $\lambda^2\Mbb+\lambda\Cbb+\Kbb$) combined with diagonal balancing is better than parameter scaling alone. }

{We complete this experiment with the initial parameter scaling included in all methods. It performed well and,  as a result, all measured backward errors were in all five methods at most $O(10^{-12})$. This is of the order of the machine precision multiplied by a low order polynomial of the dimension $n$ of the problem.}

\end{example}	
\begin{example}\label{NUMEX-2}
Structured backward errors provide a better insight into the numerical quality of the computed solutions.  
	For the data of \texttt{orr\_sommerfeld} in Example \ref{NUMEX-1}, we compute for each right eigenpair $\lambda, x$ the component-wise backward error
	 \begin{eqnarray}
&&	  \omega(\lambda,x) \!=\! \min\{ \epsilon  :  (\lambda^4 \widetilde{A}+\lambda^3 \widetilde{B}+ \lambda^2\widetilde{C}+\lambda\widetilde{D}+\widetilde{E})x=\0,\;\;|\delta A|\leq \epsilon |A|, \ldots, |\delta E|\leq \epsilon |E|
	 \} \nonumber \\  
	\!\! &=& \! \max_{i=1:n} \frac{|(\lambda^4 A + \lambda^3 B + \lambda^2 C + \lambda D + E )x|_i}{((|\lambda|^4|A| + |\lambda|^3|B| + |\lambda|^2 |C| + |\lambda||D|+|E|)|x|)_i}\! , \;\widetilde{A}\!=\! A\! +\!\delta A, \ldots, \widetilde{E}\! =\! E\! +\! \delta E. \label{eq:ComponentwiseBE}
	 \end{eqnarray}
{The corresponding error $\omega'(\lambda, y)$ for a left eigenpair $\lambda, y$ is defined analogously. }
We examine the component-wise backward errors in the \texttt{orr\_sommerfeld} example with two sets of defining parameters. 	 
Recall, the function from the NLEVP library for generating this quartic eigenvalue problem has three optional input arguments, $n$, $\omega$ and $R$: $n$ represents the dimension of the problem, $\omega$ is the frequency, and $R$ is the Reynolds number. The default values are: $n=64$, $\omega = 0.26943$ and $R=5772$ (these values are used in Table \ref{tt:BEKAVRTeig}). {In the first test, we use these default values.}

The values of $\omega(\lambda,x)$ and $\omega'(\lambda, y)$  are shown for all computed eigenpairs in Figure \ref{fig:orr1000_default_component_right_left}.

		\begin{figure}[ht]
			\centering
			\begin{minipage}{.5\textwidth}
				\includegraphics[width=1\textwidth]{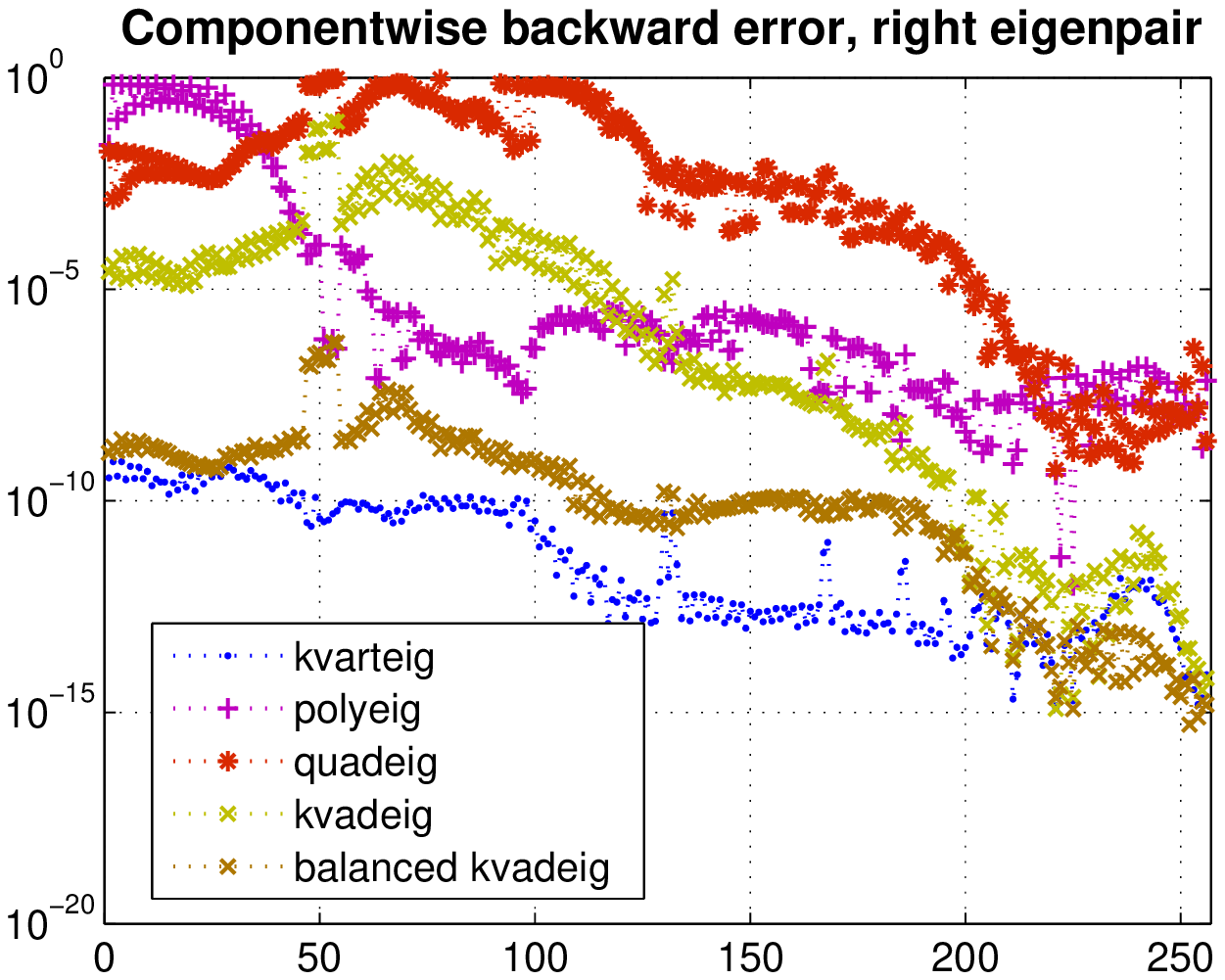}
				
			\end{minipage}%
			\begin{minipage}{.5\textwidth}
				\centering
				\includegraphics[width=1\textwidth]{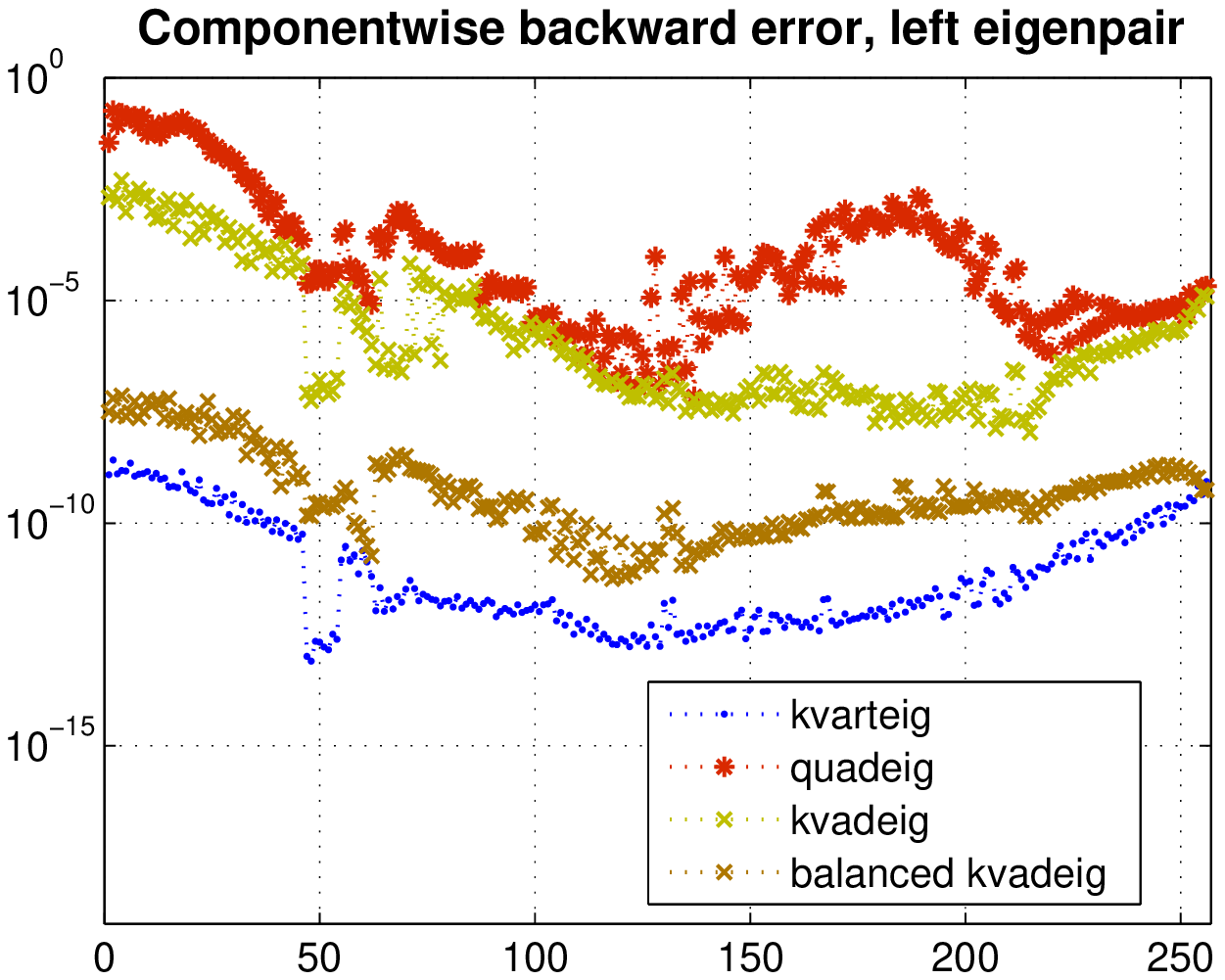}	
				
			\end{minipage}
			\caption{(Example \ref{NUMEX-2}.) Component-wise backward errors (\ref{eq:ComponentwiseBE}) for all computed eigenpairs of the \texttt{orr\_sommerfeld} example with $n=64$, $\omega = 0.26943$ and $R=5772$.  }
			\label{fig:orr1000_default_component_right_left}
		\end{figure}	
	
In the second run of the test, we increase the Reynolds number to  $R=10000$. The norm-wise and the component-wise backward errors for all eigenvalues, are shown in Figure \ref{fig:orr1000_R10000_norm_right_left} and Figure \ref{fig:orr1000_R10000_component_right_left}, respectively. Note how the backward error for \texttt{kvarteig} in Figure \ref{fig:orr1000_R10000_norm_right_left} remains nearly flat at the roundoff level, and how \texttt{kvadeig} also performs well (even with structured backward error for the right eigenpairs), despite being oblivious to the underlying structure of the quadratification and receiving unscaled original coefficients.
	
	\begin{figure}[ht]
		\centering
		\begin{minipage}{.5\textwidth}
			\includegraphics[width=0.99\textwidth]{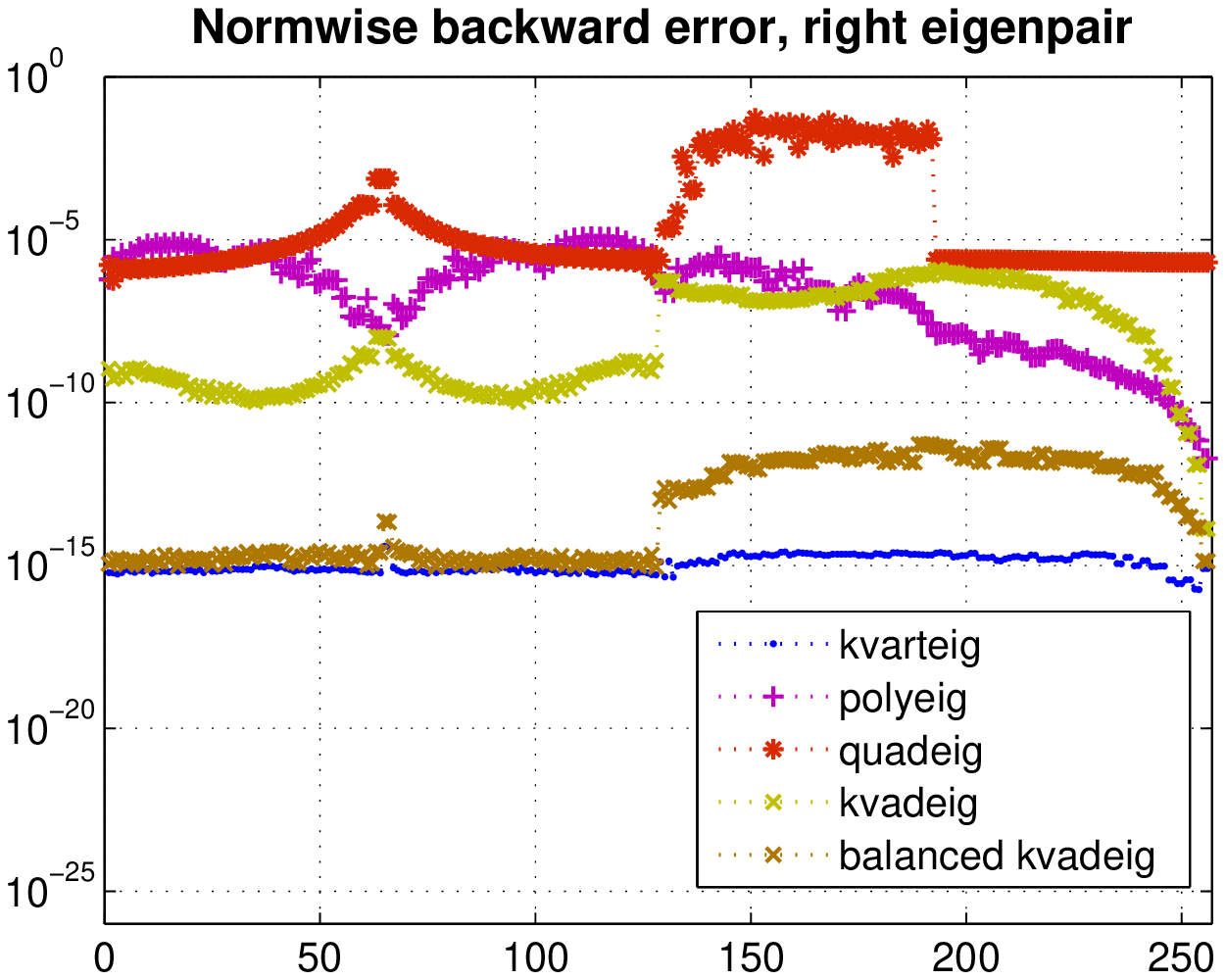}
			
		\end{minipage}%
		\begin{minipage}{.5\textwidth}
			\centering
			\includegraphics[width=0.99\textwidth]{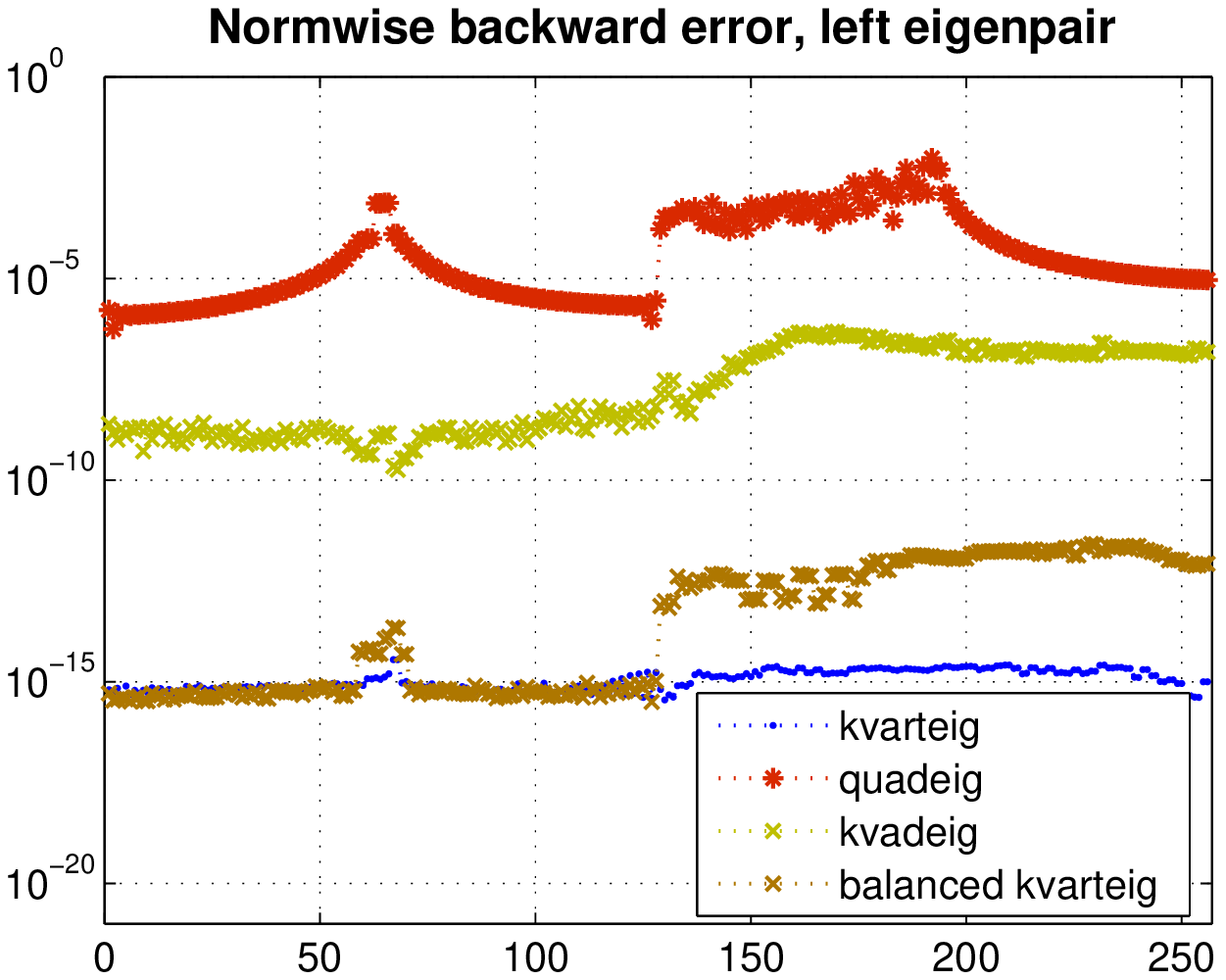}	
			
		\end{minipage}
		\caption{(Example \ref{NUMEX-2}.) The norm-wise backward errors for all computed eigenpairs of the \texttt{orr\_sommerfeld} example with $n=64$, $\omega = 0.26943$ and $R=10000$.}
		\label{fig:orr1000_R10000_norm_right_left}
	\end{figure}
	{A conclusion of this and Example \ref{NUMEX-1} is that quadratic solver equipped with parameter scaling and diagonal balancing might work reasonably well on a quadratification of the quartic problem, even when the scaling of the coefficients of the original quartic problem is omitted or unsuccessful. }
	\begin{figure}[H]
		\centering
		\begin{minipage}{.5\textwidth}
			\includegraphics[width=0.99\textwidth]{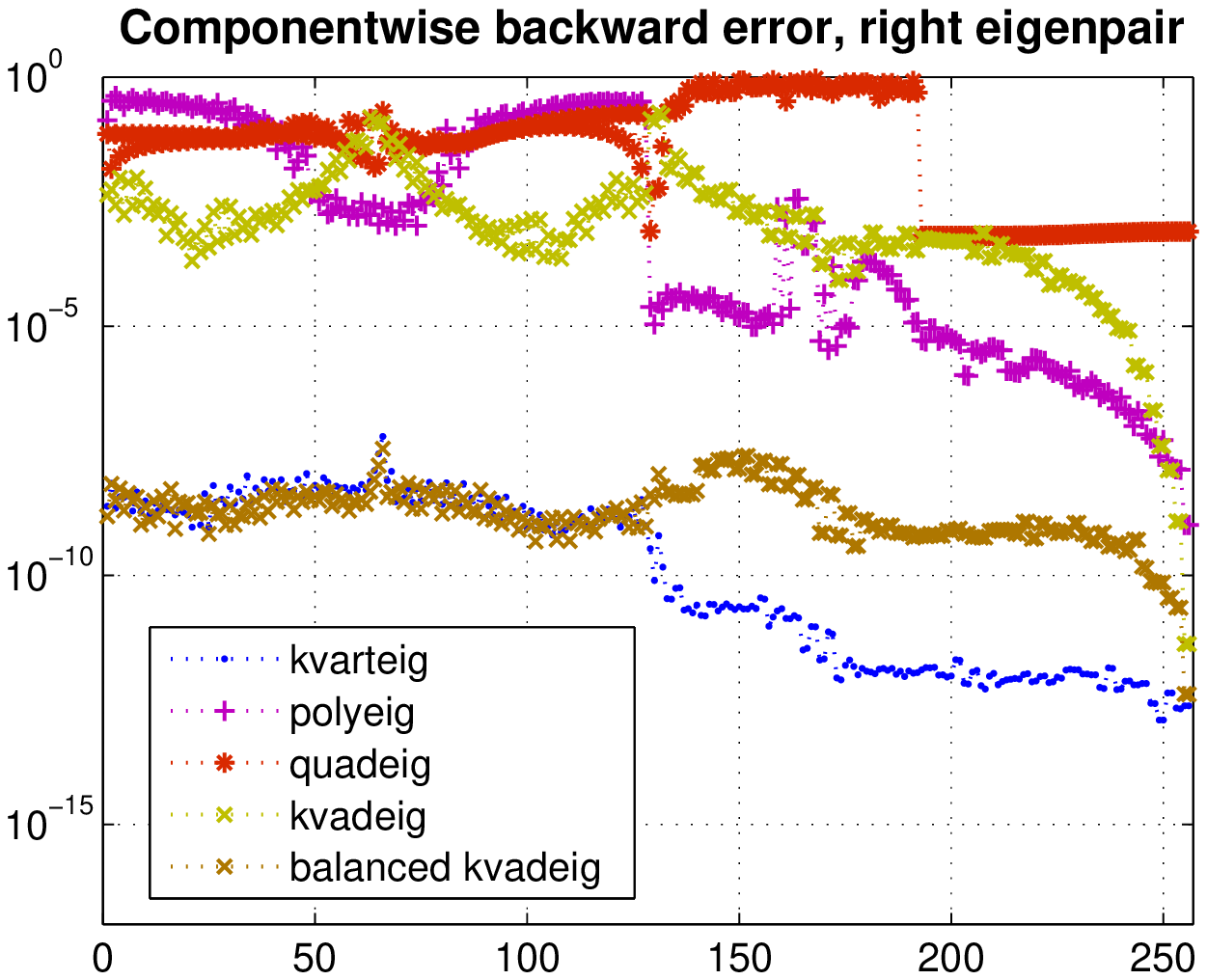}
			
		\end{minipage}%
		\begin{minipage}{.5\textwidth}
			\centering
			\includegraphics[width=0.99\textwidth]{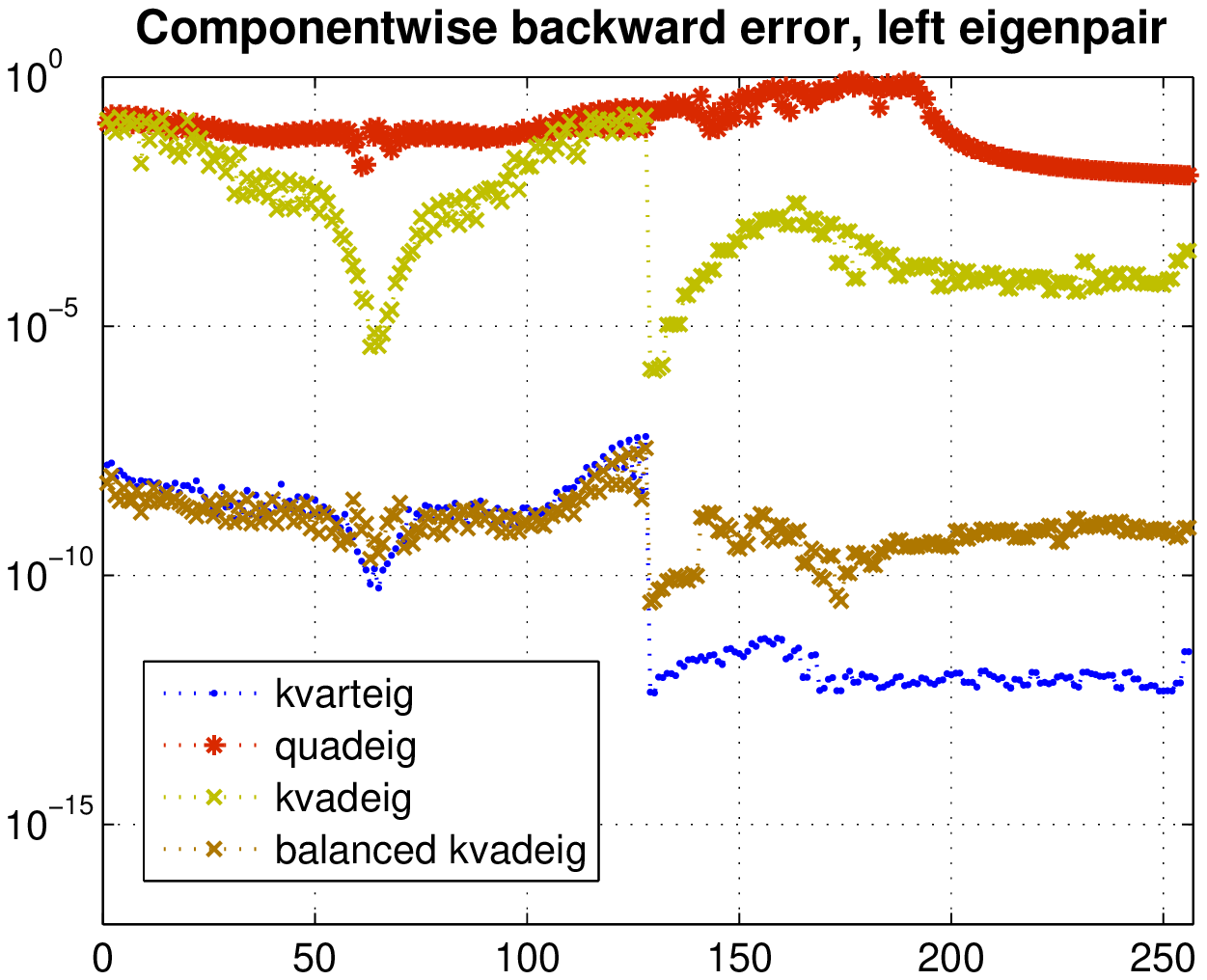}	
			
		\end{minipage}
		\caption{(Example \ref{NUMEX-2}.) The component-wise backward errors of all computed eigenpairs of the \texttt{orr\_sommerfeld} example with $n=64$, $\omega = 0.26943$ and $R=10000$.}
		\label{fig:orr1000_R10000_component_right_left}
	\end{figure}
\end{example}	
\vspace{-6mm}
\begin{remark}
{
The results of this experiment, with the computed backward errors shown in Figure \ref{fig:orr1000_R10000_norm_right_left} and Figure \ref{fig:orr1000_R10000_component_right_left}, are instructive.
First, in this example \texttt{quadeig} deflated $64$ infinite eigenvalues of the quadratic pencil $\lambda^2 \Mbb+\lambda \Cbb+\Kbb$ (see (\ref{eq:SecondComapnionFormGrade2})), because in the preprocessing stage of the algorithm, the numerical rank of the matrix $\Mbb=\left(\begin{smallmatrix}
A & \mathbf{0}\\
C & \mathbb{I}_n
\end{smallmatrix}\right)$ of order $128$ is computed as $64$.  
On the other hand, the existence of infinite eigenvalues in the original quartic eigenvalue problem depends on the rank of the leading coefficient matrix $A$. {If we inspect the singular values\footnote{Singular values are indexed in non-increasing order, $\sigma_i(\cdot)\geq\sigma_{i+1}(\cdot)$.} $\sigma_i(\Mbb)$ of $\Mbb$ and $\sigma_i(A)$ of $A$, then, as clearly shown in Figure \ref{fig:orr1000_R10000_svd}, $A$ is numerically of full rank (its condition number is below $10^6$, so \texttt{kvarteig} safely removed the possibility of infinite eigenvalues). On the other hand, $\sigma_{65}(\Mbb)/\sigma_1(\Mbb)$ is at the level of \emph{dimension of $\Mbb$ times machine precision} and in many algorithms this is the truncation threshold for numerical rank deficiency.}
%
Note that parameter scaling of the quadratic pencil (\ref{eq:SecondComapnionFormGrade2}) cannot remove this problem. 
	
\begin{figure}[ht]
\centering
\begin{minipage}{.5\textwidth}
\includegraphics[width=0.99\textwidth,height=1.8in]{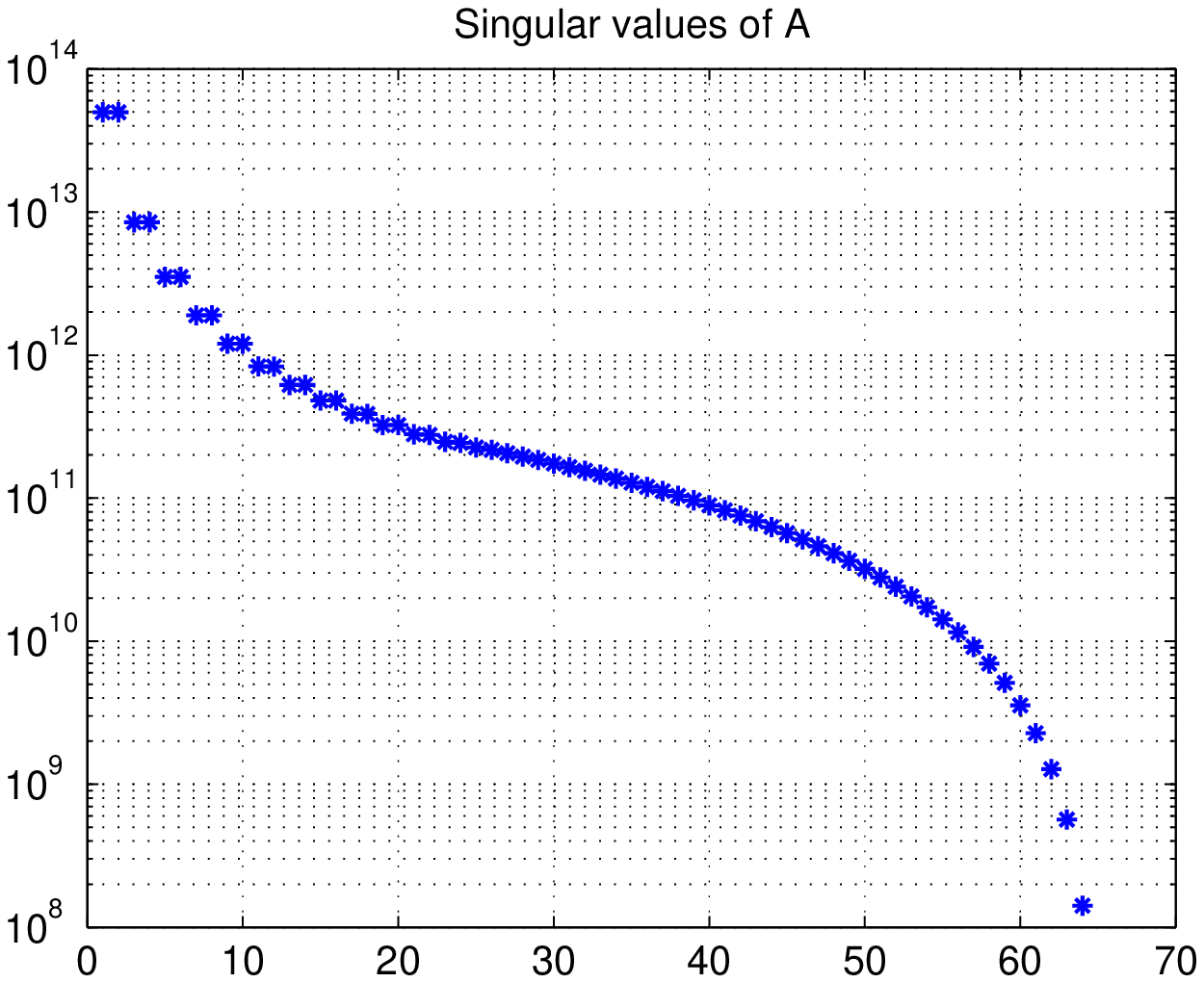}
			
\subcaption{Leading coefficient $A$ of the quartic problem}
\end{minipage}%
\begin{minipage}{.5\textwidth}
\centering
\includegraphics[width=0.99\textwidth,height=1.8in]{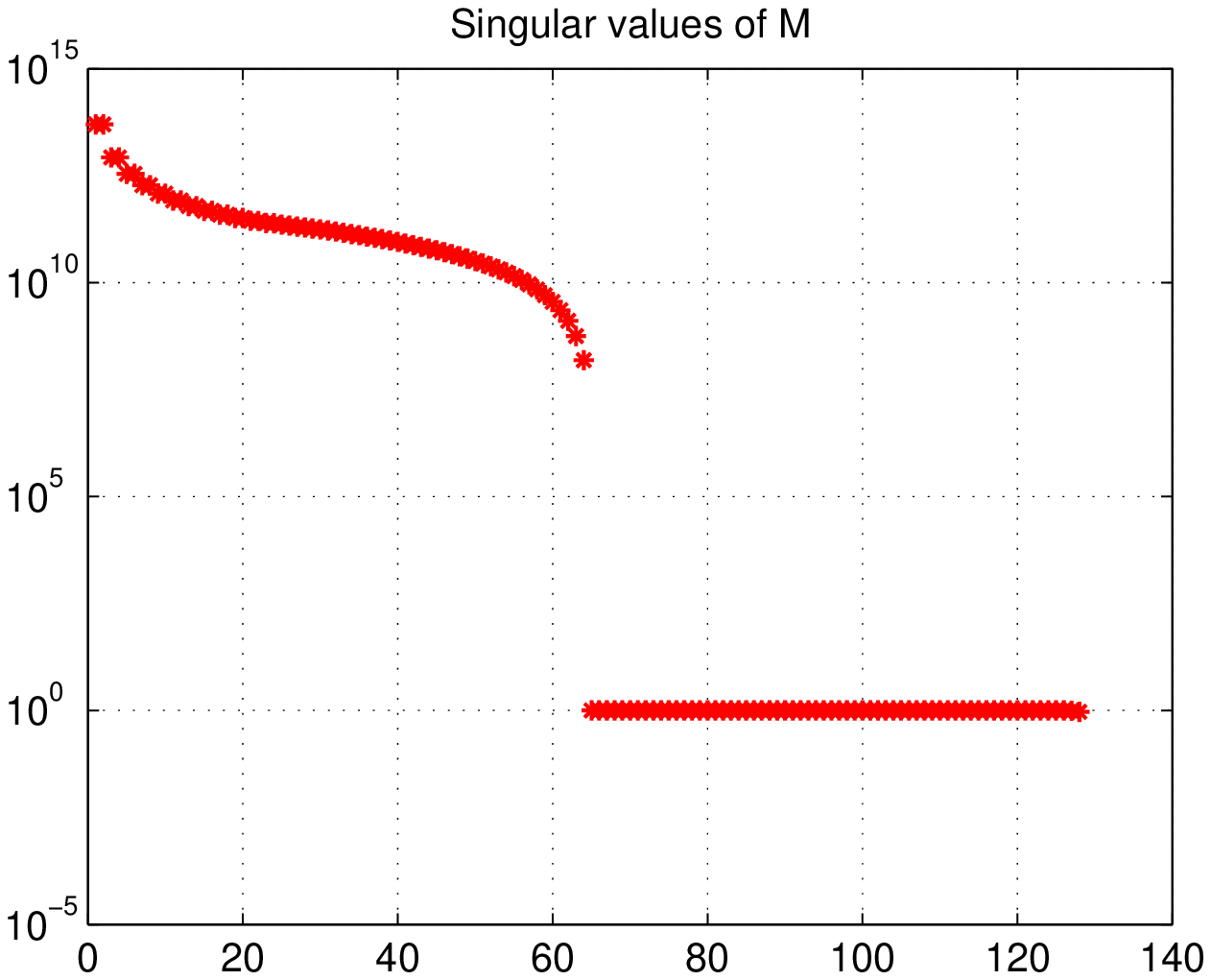}	
			
\subcaption{Leading coefficient $\mathbb{M}$ of the quadratification }
\end{minipage}
\caption{(Example \ref{NUMEX-2}.) Singular values of the leading coefficient matrices of the original quartic problem (\ref{eq:QuarticEP}) and the quadratification (\ref{eq:SecondComapnionFormGrade2}). Note that $\sigma_{\max}(A)/\sigma_{\min}(A)=O(10^6)$, $\sigma_{65}(\Mbb)/\sigma_1(\Mbb)=O(n)\roff=O(10^{-14})$,
$\sigma_{65}(\Mbb)/\sigma_{64}(\Mbb)=O(\sqrt{\roff})$. Here  $\roff\approx 2.2\cdot 10^{-16}$ is the machine precision.}
\label{fig:orr1000_R10000_svd}
\end{figure}
\vspace{-1mm}	
{On the other hand, \texttt{kvadeig} (applied to the same $\lambda^2 \Mbb+\lambda \Cbb+\Kbb$) declared the matrix $\Mbb$ nonsingular, and thus no infinite eigenvalues where deflated nor found by the QZ algorithm. This is because \texttt{kvadeig} uses more local truncation strategy; it truncates at index $i$ if $\sigma_{i+1}(\Mbb)/\sigma_i(\Mbb)$ is estimated to be small;   see Remark \ref{REM-drop-off} and Example \ref{NUMEX-3}. Good results by \texttt{balanced\_kvadeig} are due to the additional  balancing \cite[\S 4.2]{KVADeig-arxiv}, and this example once more justifies our approach in \texttt{kvadeig} (using local truncation strategy and balancing in combination with parameter scaling).}
%
}
\end{remark}

{Now, we turn on the parameter scaling, which is a necessary tool for numerical stability of a polynomial eigensolver. Although the scaling described in \S \ref{SSS=param-scaling} is a simple combination of the existing and well known formulas, it seems that it works well. In particular, in many cases it works well for \texttt{polyeig}, as we already showed in Example \ref{NUMEX-1}. This is illustrated in the next two numerical experiments with the \texttt{orr\_sommerfeld} example of dimensions $n=64$ and $n=1000$.}

\begin{example}\label{NUMEX-OS-R-1000-scaled}
 We use the same benchmark problem as in the second part of Example \ref{NUMEX-2} (\texttt{orr\_sommerfeld} example with $n=64$, $\omega = 0.26943$ and $R=10000$.), but initially we scale the matrices as described in \S \ref{SSS=param-scaling}, so that all algorithms start with scaled data. This example has no infinite eigenvalues.
 The results of all algorithms depend on the QZ algorithm, thus the similar results. {(It seems that \texttt{polyeig} and \texttt{kvarteig} are a little bit better than \texttt{kvadeig} and \texttt{quadeig}, which makes sense because both work on the original coefficients, while the quadratic solvers work on $\Mbb$, $\Cbb$, $\Kbb$ from the quadratification.)}
	
		\begin{figure}[ht]
			\centering
			\begin{minipage}{.5\textwidth}
				\includegraphics[width=0.99\textwidth]{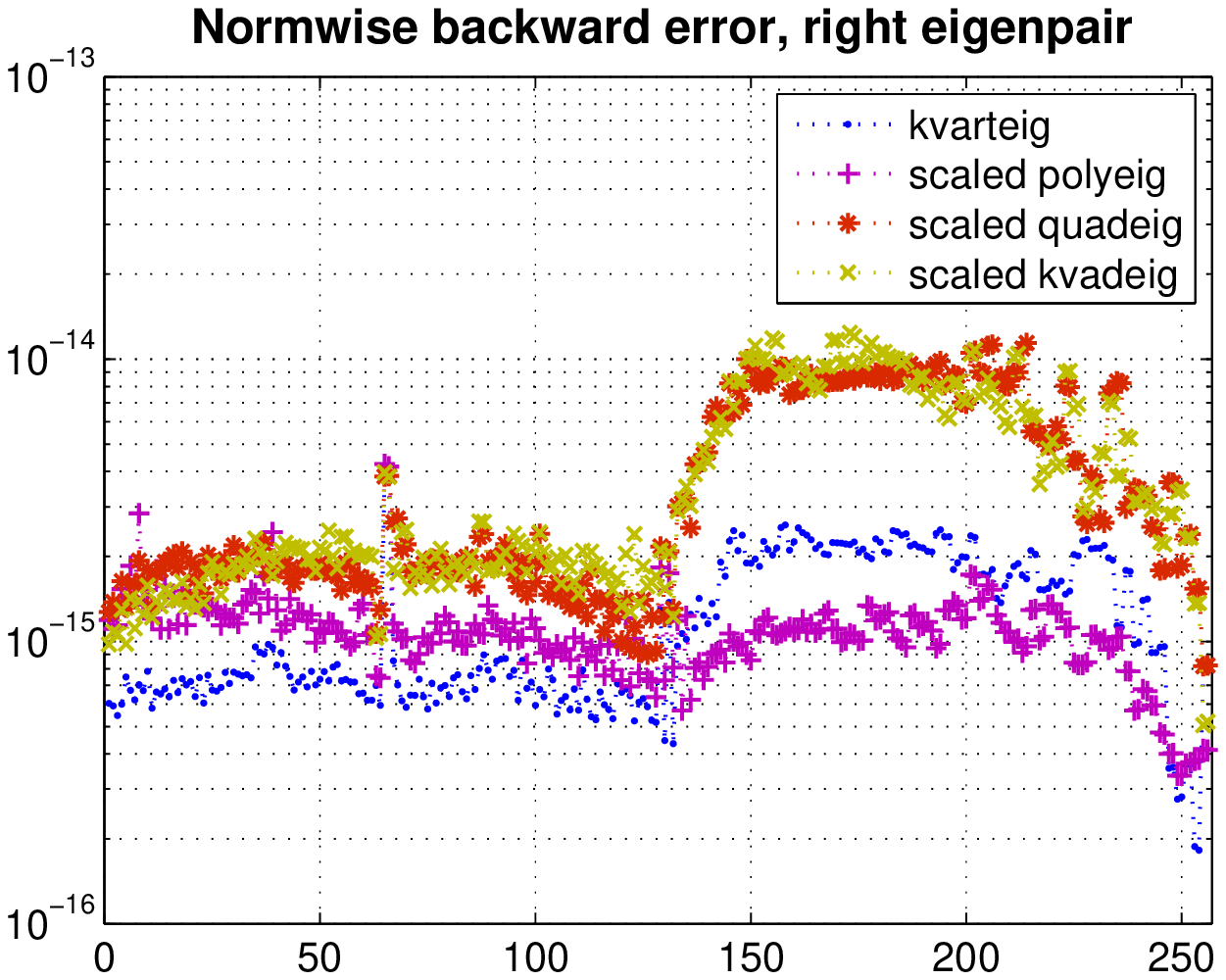}
				
			\end{minipage}%
			\begin{minipage}{.5\textwidth}
				\centering
				\includegraphics[width=0.99\textwidth]{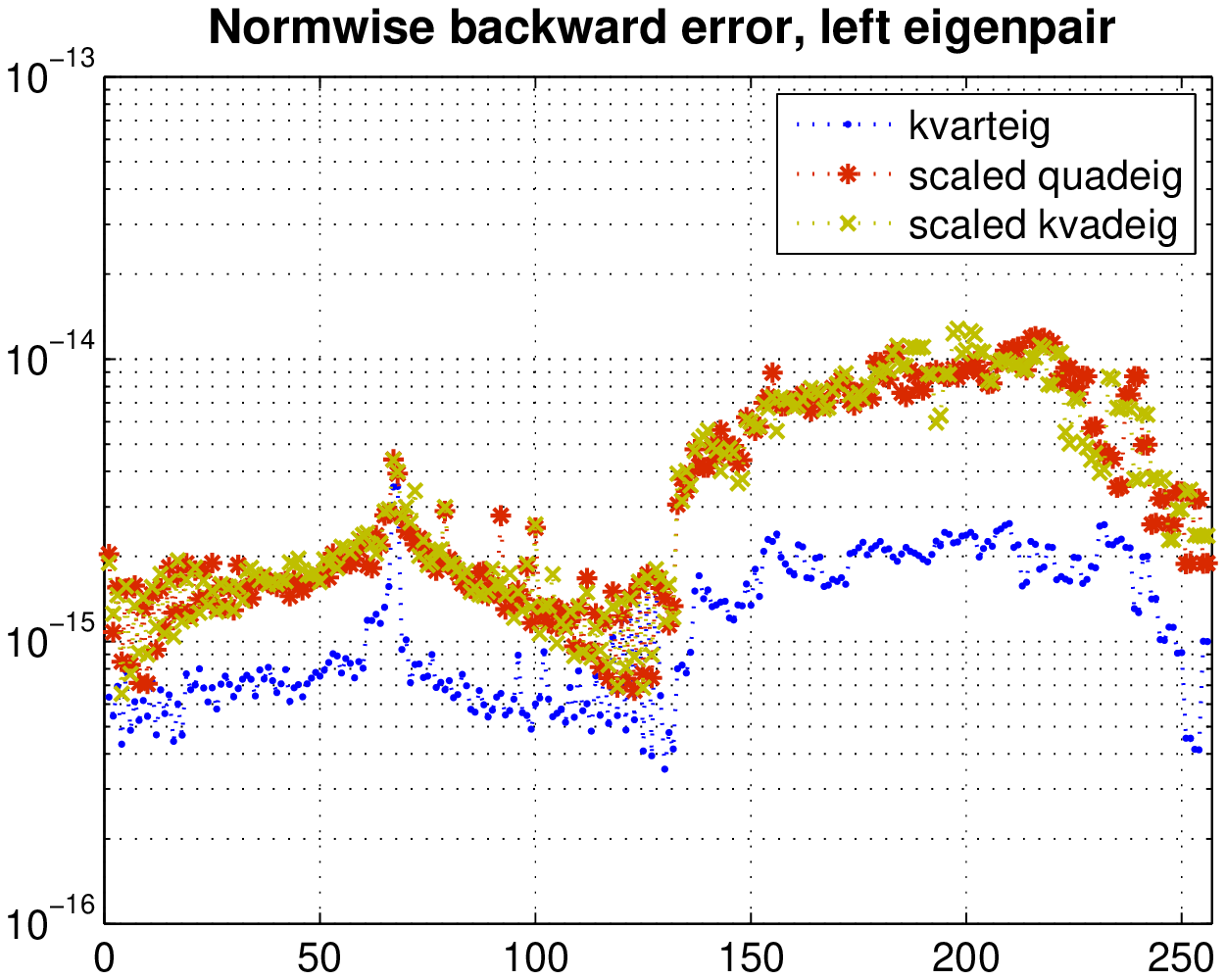}	
				
			\end{minipage}
			\caption{(Example \ref{NUMEX-OS-R-1000-scaled}.) The norm-wise backward errors for all computed eigenpairs for the \texttt{orr\_sommerfeld} example with $n=64$, $\omega = 0.26943$ and $R=10000$.}
			\label{fig:orr1000_R10000_norm_right_left-scaled}
		\end{figure}
\vspace{-2mm}		
		\begin{figure}[ht]
			\centering
			\begin{minipage}{.5\textwidth}
				\includegraphics[width=0.99\textwidth]{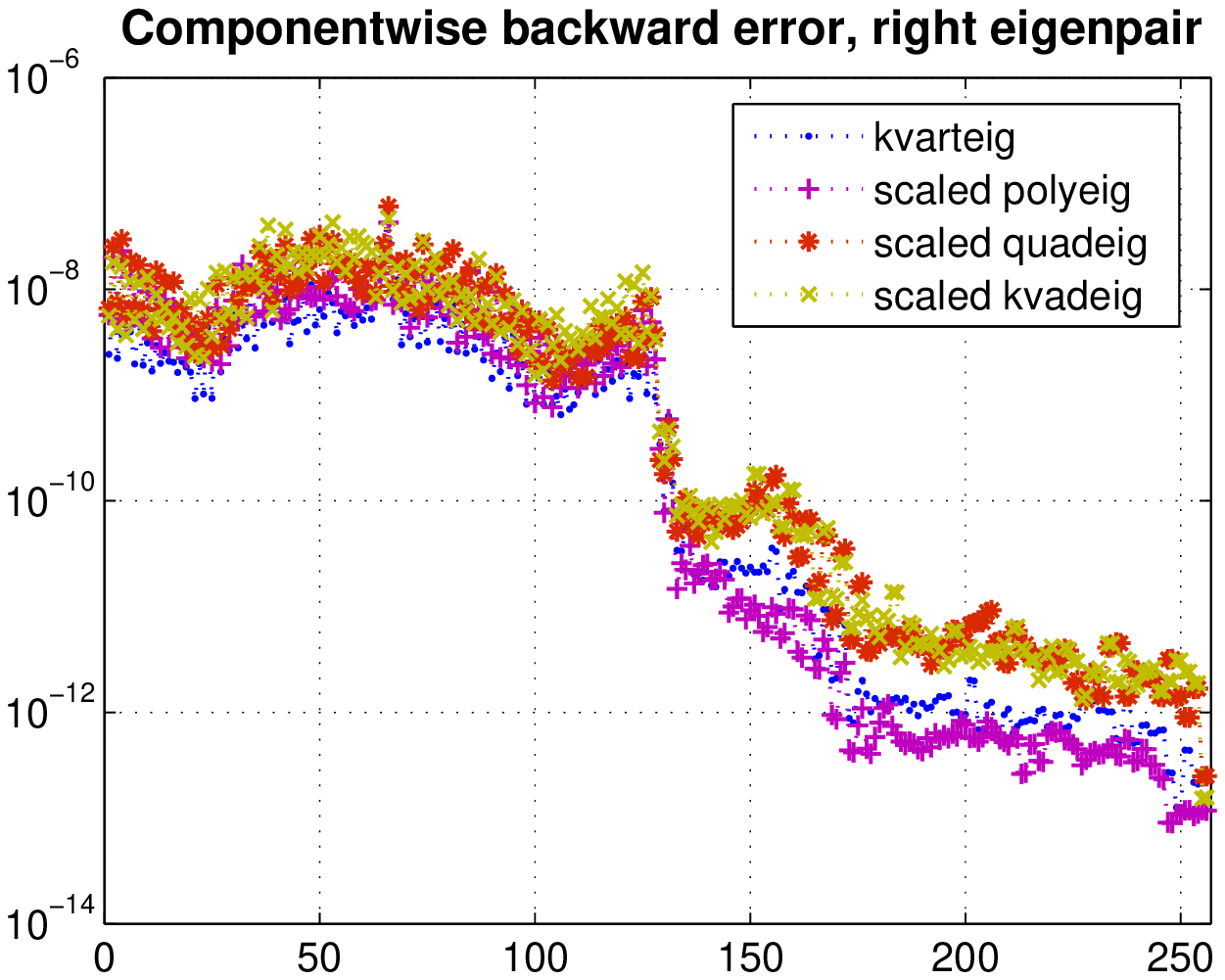}
				
			\end{minipage}%
			\begin{minipage}{.5\textwidth}
				\centering
				\includegraphics[width=0.99\textwidth]{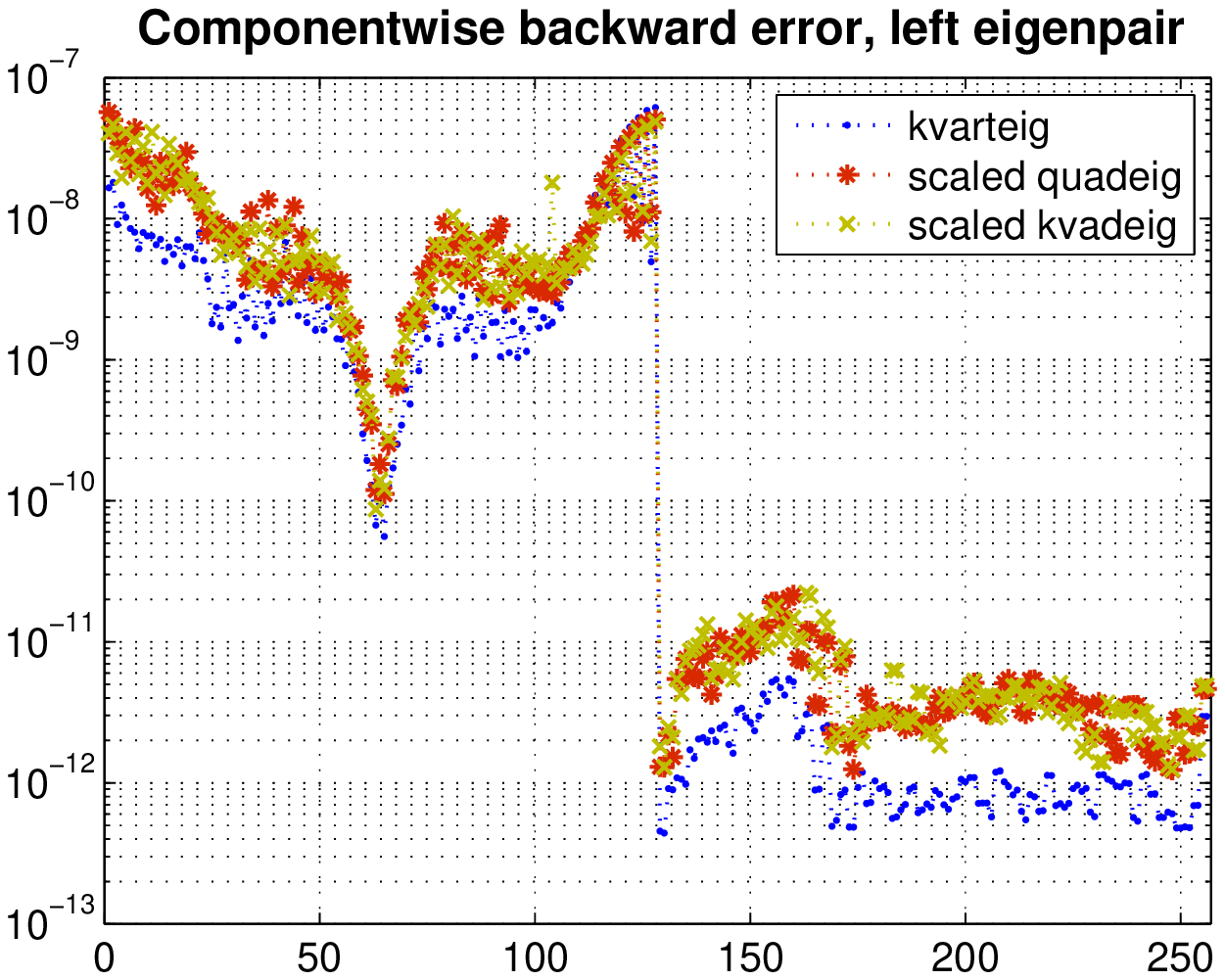}	
				
			\end{minipage}
			\caption{(Example \ref{NUMEX-OS-R-1000-scaled}.) The component-wise backward errors for all computed eigenpairs for the \texttt{orr\_sommerfeld} example with $n=64$, $\omega = 0.26943$ and $R=10000$.}
			\label{fig:orr1000_R10000_component_right_left-scaled}
		\end{figure}
	
	\end{example}	
\vspace{-3mm}
\begin{example}\label{NUMEX-3}
We continue experimenting with the \texttt{orr\_sommerfeld} example; we choose the default values of the Reynolds number $R$ and the frequency $\omega$, but increase the dimension to $n=1000$, and compute all $4000$ eigenpairs. {The matrix coefficients are scaled using the same strategy as in the previous example.\footnote{Without parameter scaling of the initial data, the Matlab function \texttt{polyeig} failed completely -- all computed eigenvalues were of the form $\pm \texttt{Inf} \pm \texttt{Inf}\iu$.}}

An application of \texttt{quadeig} to the  quadratification (\ref{eq:SecondComapnionFormGrade2}) returned $282$ infinite eigenvalues. {With \texttt{kvadeig} and the same quadratification, $31$ infinite eigenvalues are detected.} 
On the other hand, if we use balancing (\texttt{balanced\_kvadeig}), the leading coefficient matrix is declared regular, and no infinite eigenvalues are detected. The difference is mainly due to the softer drop-off truncation in the rank revealing QR factorization.
The result of \texttt{kvarteig} also depends on the truncation strategy. If the truncation of the pivoted QR factorization is done relative to the norm of $A$, the numerical rank is $988$, meaning that $12$ infinite eigenvalues are deflated immediately in the preprocessing phase. In the case of drop-off strategy, the matrix $A$ is not numerically  rank deficient. 
	
	The  norm-wise and component-wise backward errors for the computed right and left eigenpairs are shown in  Figures \ref{fig:orr_sommerfeld_1000_scaled_nbe_right}, \ref{fig:orr_sommerfeld_1000_scaled_nbe_left}, \ref{fig:orr_sommerfeld_1000_scaled_cbe_right}, \ref{orr_sommerfeld_1000_scaled_cbe_left}.
	

	\begin{figure}[ht]
			\includegraphics[scale=0.50,width=0.99\textwidth]{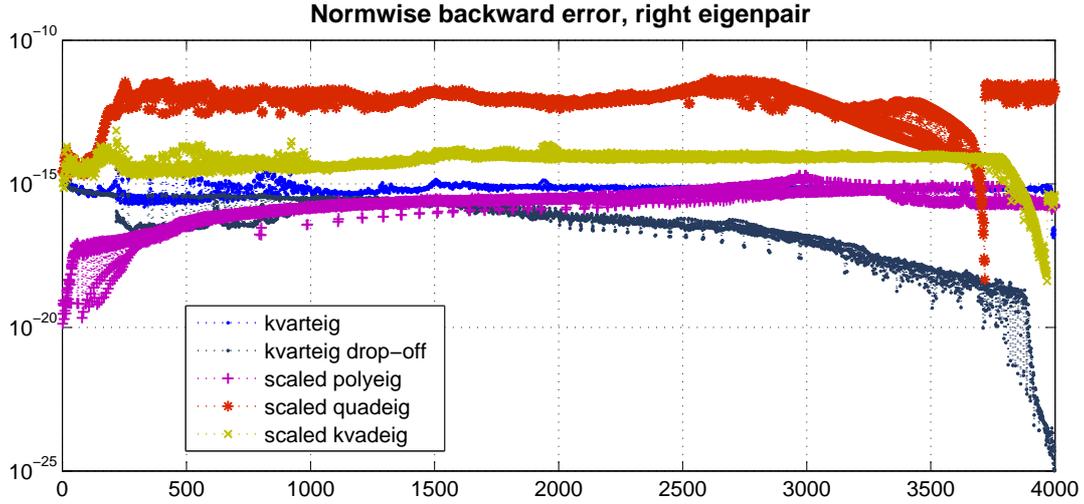}
		\caption{(Example \ref{NUMEX-3}.) Norm-wise  backward errors for the  right eigenpairs. \label{fig:orr_sommerfeld_1000_scaled_nbe_right}}
		\end{figure}
		\begin{figure}[ht]
			\includegraphics[scale=0.50,width=0.99\textwidth]{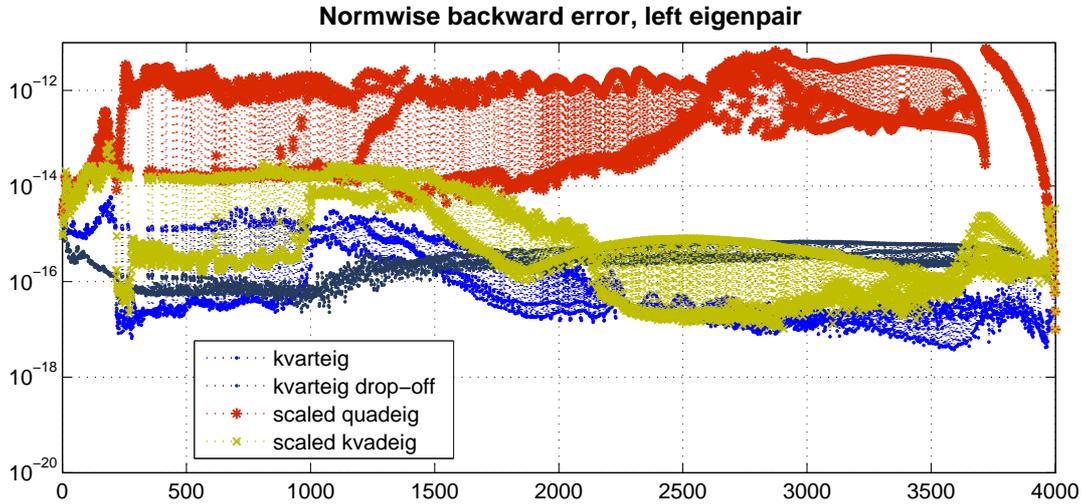}
		\vspace{-1mm}
		\caption{(Example \ref{NUMEX-3}) Norm-wise  backward errors for the left eigenpairs. \label{fig:orr_sommerfeld_1000_scaled_nbe_left}}
		\end{figure}
		\begin{figure}[ht]
			\includegraphics[scale=0.50,width=0.99\textwidth]{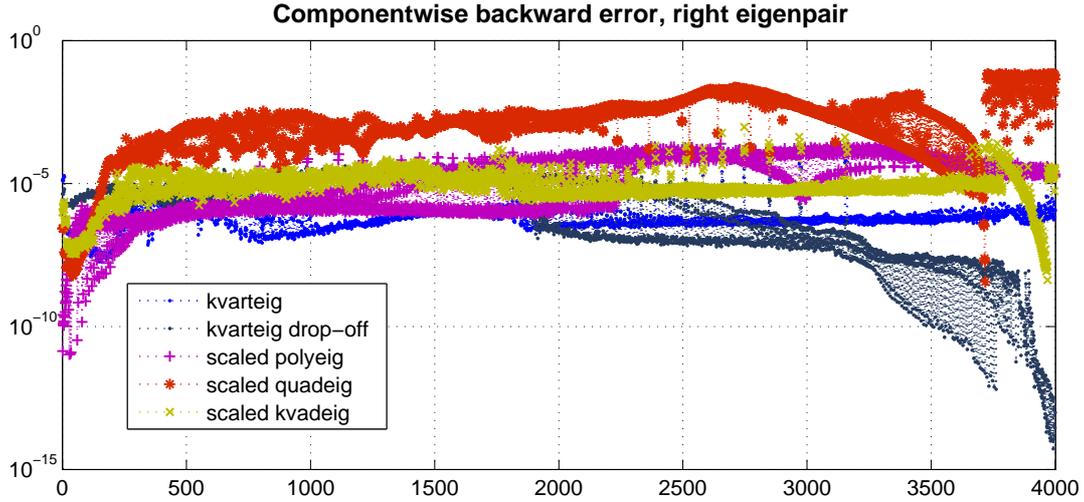}
		\vspace{-1mm}
		\caption{(Example \ref{NUMEX-3}.) Component-wise backward errors for the right eigenpairs. \label{fig:orr_sommerfeld_1000_scaled_cbe_right}}
		\end{figure}

		\begin{figure}[ht]
			\includegraphics[scale=0.50,width=0.99\textwidth]{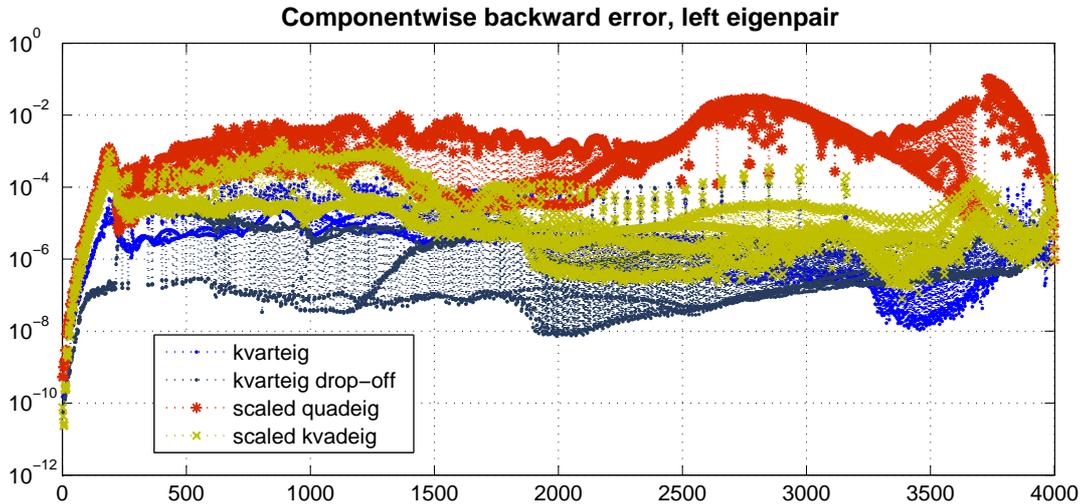}
		\caption{(Example \ref{NUMEX-3}.) Component-wise backward errors for the left eigenpairs. \label{orr_sommerfeld_1000_scaled_cbe_left}}	
	\end{figure}

\end{example}

\begin{example}\label{EX-mirror-transposed}
In this example, we use transposed matrices from the \texttt{mirror} example\footnote{Recall, we analyzed this example in \S \ref{NUMEX-4}, where we argued that there are nine zero and nine infinite eigenvalues.} and scale them as described in \S \ref{SSS=param-scaling}. The number of zero and infinite eigenvalues 
found by the four algorithms were: \texttt{polyeig} (no zeros and $5$ infinities);
\texttt{quadeig} ($7$ zeros and $9$ infinities); \texttt{kvadeig} and \texttt{kvarteig} $9$ zeros and $9$ infinite eigenvalues. The component-wise {and the norm-wise} backward errors are given in Figure \ref{fig:mirror-transposed-cbe-right}.
 

	\begin{figure}[ht]
		\centering
		\begin{minipage}{.5\textwidth}
			\includegraphics[width=0.99\textwidth]{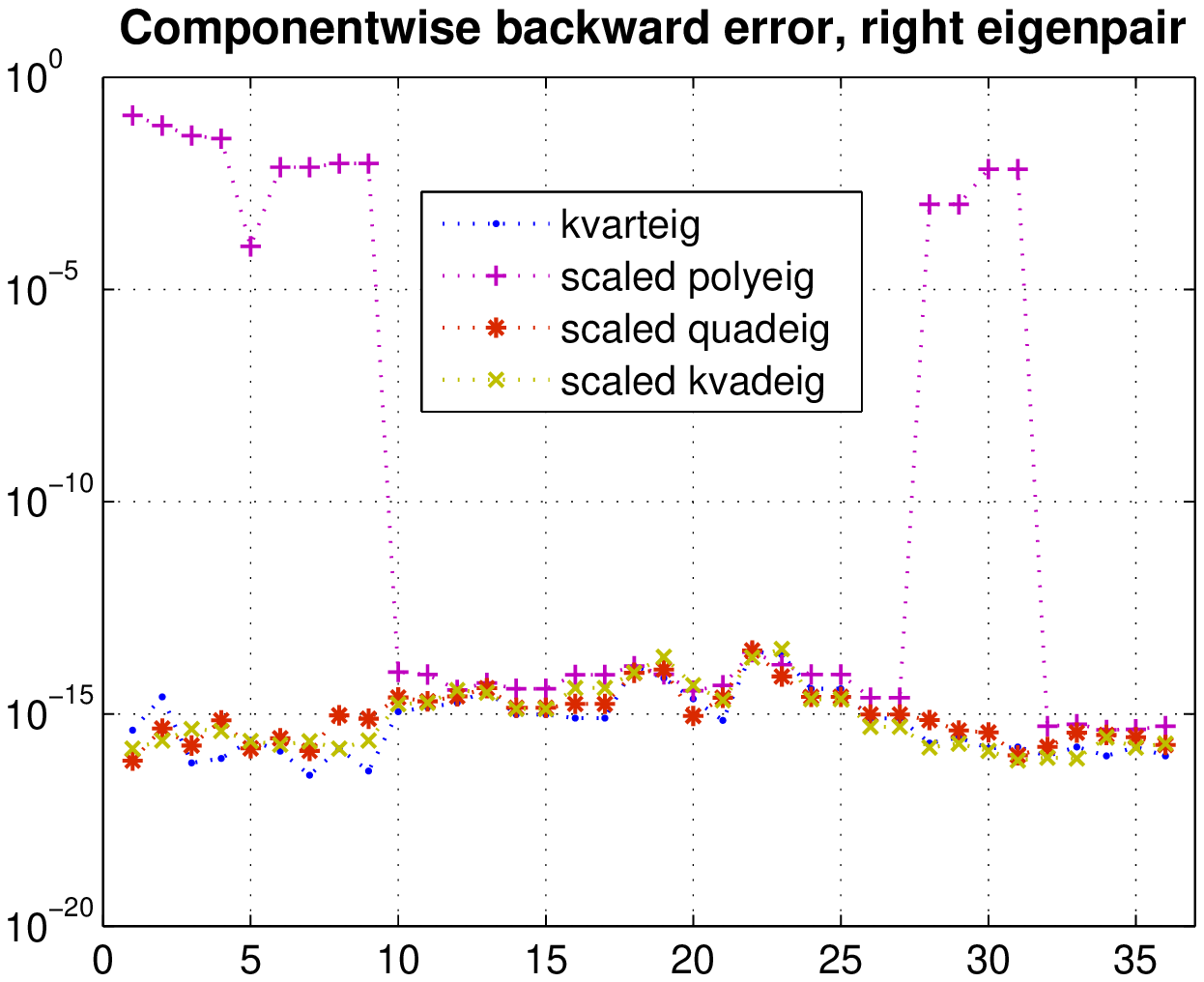}	
		\end{minipage}%
		\begin{minipage}{.5\textwidth}
			\centering
			\includegraphics[width=0.99\textwidth]{mirror_nbe_right_scaled.eps}	
		\end{minipage}
		\caption{(Example \ref{EX-mirror-transposed}) \emph{Left panel}: Component-wise backward errors for the transposed \texttt{mirror} example. \emph{Right panel}: Norm-wise backward error for the original \texttt{mirror} example (This is the right panel from Figure \ref{fig:mirrorBE}, here given for comparison.)}
		\label{fig:mirror-transposed-cbe-right}
	\end{figure}	
\end{example}	
		
\begin{example}\label{NUMEX:balance-butterfly}
	 In this example we illustrate potential benefits of equilibration of the coefficient matrices on the element level, mentioned in Remark \ref{REM=Balancing}. Such diagonal scalings \emph{balance} the absolute values of nonzero entries over all matrices. 
	 
	 We take the \texttt{butterfly} example and pre-multiply its coefficient matrices   with diagonal matrix $\varDelta$ with randomly permuted powers $2^i$, $i=1,\ldots, n=64$ on the diagonal. This is an entirely artificial step to simulate a situation with ill-conditioning caused by removable scaling (that may originate in an inappropriate scale of physical units). We obtain an equivalent problem, but numerical algorithms may be more or less sensitive to this change of representation.  
	
	 Then, we compute balancing matrices $\varDelta_{\ell}$, $\varDelta_r$ (see Remark \ref{REM=Balancing}) and examine how this preprocessing ($(A,B,C,D,E)\rightsquigarrow \varDelta_{\ell}(A,B,C,D,E)\varDelta_{r}$) influences the numerical accuracy of the algorithms under study. The computed component-wise backward errors,  shown in Figure \ref{fig:balance-buttefly1}, 
clearly demonstrate the impact of the balancing $(A,B,C,D,E)\rightsquigarrow \varDelta_{\ell}(A,B,C,D,E)\varDelta_{r}$.
	\begin{figure}[ht]
		\centering
		\begin{minipage}{.5\textwidth}
			\includegraphics[width=0.99\textwidth]{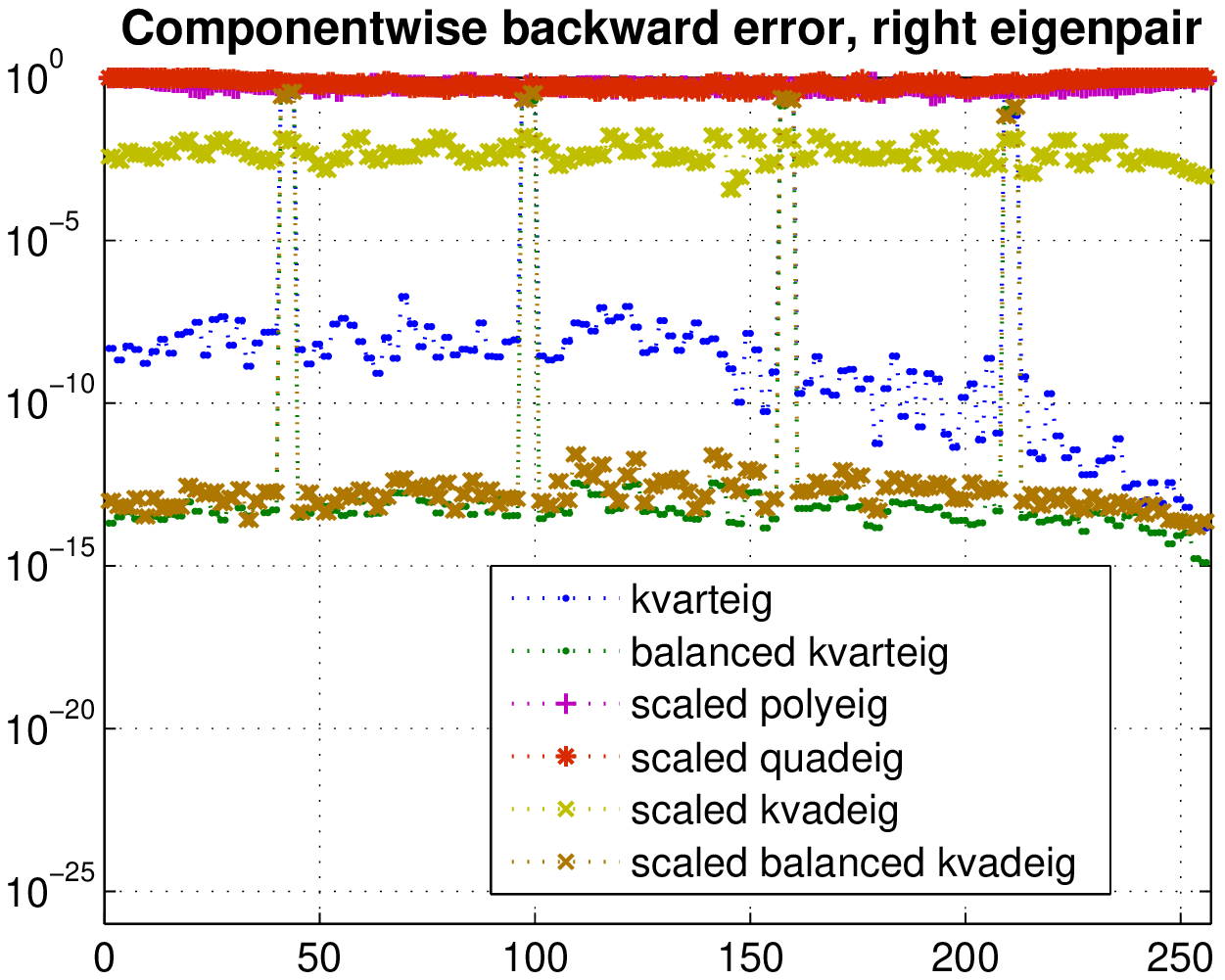}
			
		\end{minipage}%
		\begin{minipage}{.5\textwidth}
			\centering
			\includegraphics[width=0.99\textwidth]{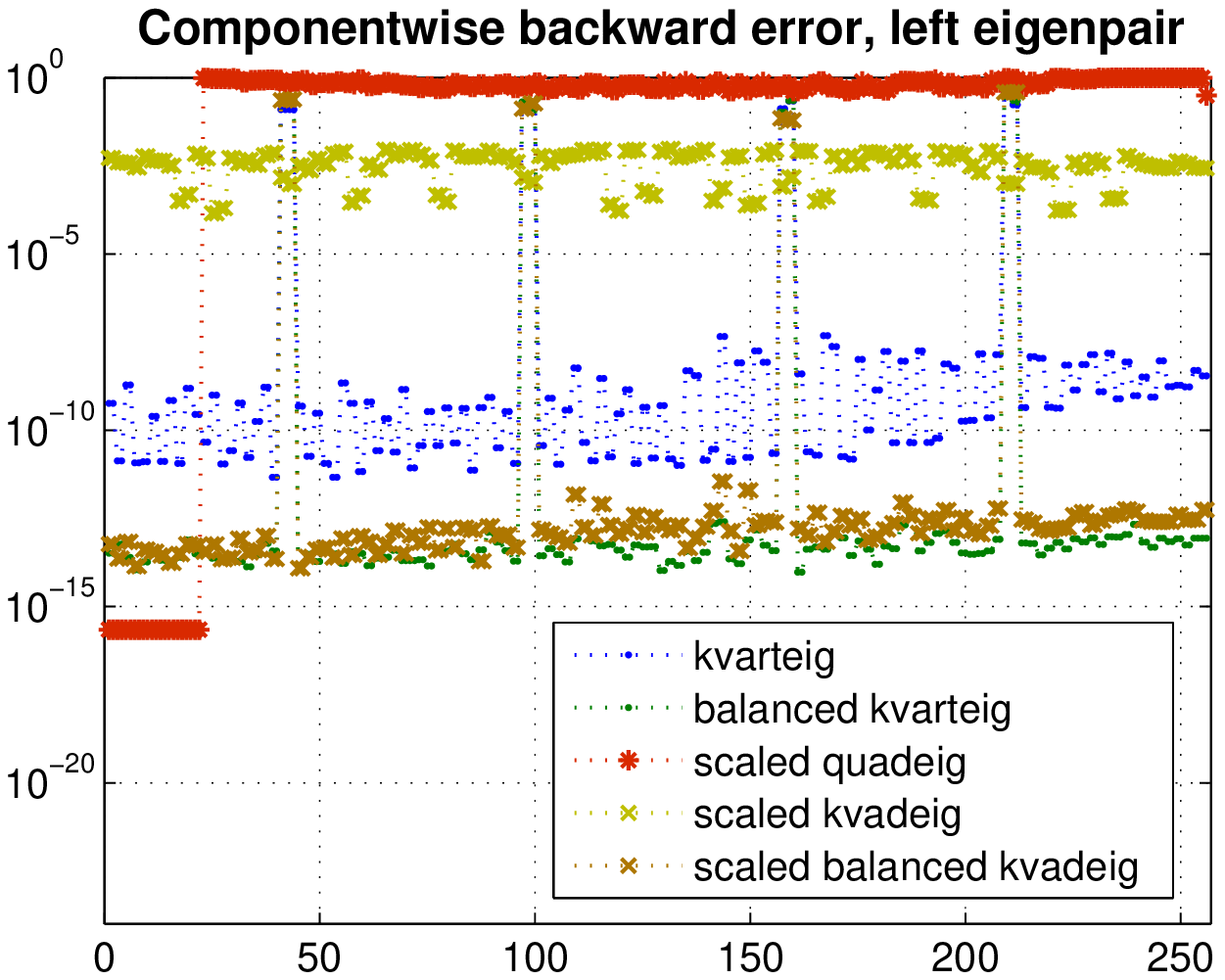}	
			
		\end{minipage}
		\caption{(Example \ref{NUMEX:balance-butterfly}, \texttt{butterfly}.) Component-wise backward errors for the modified \texttt{butterfly} example, where the coefficients are premultiplied by a diagonal matrix $\varDelta$, $(A,B,C,D,E)\rightsquigarrow \varDelta (A,B,C,D,E)$. }
		\label{fig:balance-buttefly1}
	\end{figure}
 Note also that without balancing \texttt{kvadeig} still performs well, much better than \texttt{polyeig} and \texttt{quadeig} under the same conditions.	
			
			
\end{example}

\begin{example}
	In our last example, we checked the least squares approach to recovering the eigenvectors, as described in \S \ref{SSS=LS-vectors}. The computed backward errors in all tested cases were comparable with the method of selecting the vector with smallest residual. {We believe that this least squares approach could be useful for getting good eigenvectors for selected eigenvalues that are of particular importance in some applications.}
	\end{example}

\section{Concluding remarks}
We have shown  that the proposed algorithm \texttt{kvarteig} for solving quartic eigenvalue problems is a useful contribution that fills the gap in the toolbox for the polynomial eigenvalue problems,  both for the full solution of medium size non-structured problems and for solving the projected problems in subspace based methods for large scale structured/sparse problems.
Numerical experiments with the benchmark examples from the NLEVP collection show that \texttt{kvarteig} is superior to \texttt{polyeig} from Matlab, or \texttt{quadeig} applied to a quadratification of the original quartic problem. Further, the numerical performances of \texttt{kvadeig} on the quadratificaton of the quartic problem additionally justify the modifications that underpinned the development both in \cite{KVADeig-arxiv} and in this paper. 

Given the wide spectrum of applications of the quartic eigenvalue problem, we are certain that our proposed algorithm will prove useful in many computational tasks in applied sciences and engineering. Further, the presented techniques can be adapted for other methods and suitable linearizations of polynomial eigenvalue problems. 

An early version of this work is available at \cite{2019arXiv190507013D}.

\section*{Acknowledgement}
{This research has been supported by  the Croatian Science Foundation (CSF) grants IP-2019-04-6268, and in part (the second author) UIP-2019-04-5200.}
Parts of this work originate from the second author's thesis \cite{ISG-thesis-2018}.
	The authors thank to Serkan Gugercin (Virginia Tech, Blacksburg), Luka Grubi\v{s}i\'{c} and Zvonimir Bujanovi\'{c} (University of Zagreb) for valuable comments, and in particular to the three anonymous referees for their constructive criticism and detailed reports. 




\bibliography{KVARTeig_References}

\begin{thebibliography}{10}

\bibitem{betcke2008optimal}
Timo Betcke.
\newblock Optimal scaling of generalized and polynomial eigenvalue problems.
\newblock {\em SIAM Journal on Matrix Analysis and Applications},
  30(4):1320--1338, 2009.

\bibitem{betcke2010nlevp}
Timo Betcke, Nicholas~J. Higham, Volker Mehrmann, Christian Schr{\"o}der, and
  Fran{\c{c}}oise Tisseur.
\newblock {NLEVP}: A collection of nonlinear eigenvalue problems.
\newblock {\em ACM Transactions on Mathematical Software (TOMS)}, 39(2):1--28,
  2013.

\bibitem{BBD-Gen-Hess-2013}
Nela Bosner, Zvonimir Bujanovi\'{c}, and Zlatko Drma\v{c}.
\newblock Efficient generalized {Hessenberg} form and applications.
\newblock {\em ACM Trans. Math. Softw.}, 39(3), May 2013.

\bibitem{BBD-shifted-2018}
Nela Bosner, Zvonimir Bujanovi\'{c}, and Zlatko Drma\v{c}.
\newblock Parallel solver for shifted systems in a hybrid {CPU--GPU} framework.
\newblock {\em SIAM Journal on Scientific Computing}, 40(4):C605--C633, 2018.

\bibitem{bus-gol-65}
Peter Businger and Gene~H. Golub.
\newblock Linear least squares solutions by {Householder} transformations.
\newblock {\em Numerische Mathematik}, 7(3):269--276, 1965.

\bibitem{campos2016parallel}
Carmen Campos and Jose~E. Roman.
\newblock Parallel {Krylov} solvers for the polynomial eigenvalue problem in
  {SLEPc}.
\newblock {\em SIAM Journal on Scientific Computing}, 38(5):S385--S411, 2016.

\bibitem{Chen2017}
Hongjia Chen, Akira Imakura, and Tetsuya Sakurai.
\newblock Improving backward stability of {Sakurai--Sugiura} method with
  balancing technique in polynomial eigenvalue problem.
\newblock {\em Applications of Mathematics}, 62(4):357--375, 2017.

\bibitem{cheb-orr-somm-1990}
G\"{o}khan Danabasoglu and Sedat Biringen.
\newblock A {Chebyshev} matrix method for the spatial modes of the
  {Orr–-Sommerfeld} equation.
\newblock {\em International Journal for Numerical Methods in Fluids}, 11:1033
  -- 1037, 11 1990.

\bibitem{de2014spectral}
Fernando De~Ter{\'a}n, Froil{\'a}n~M Dopico, and D~Steven Mackey.
\newblock Spectral equivalence of matrix polynomials and the index sum theorem.
\newblock {\em Linear Algebra and its Applications}, 459:264--333, 2014.

\bibitem{DETERAN2016344}
Fernando {De Ter\'{a}n}, Froil\'{a}n~M. Dopico, and Paul {Van Dooren}.
\newblock Constructing strong $\ell$-ifications from dual minimal bases.
\newblock {\em Linear Algebra and its Applications}, 495:344 -- 372, 2016.

\bibitem{DMYTRYSHYN2017213}
Andrii Dmytryshyn and Froil\'{a}n~M. Dopico.
\newblock Generic complete eigenstructures for sets of matrix polynomials with
  bounded rank and degree.
\newblock {\em Linear Algebra and its Applications}, 535:213--230, 2017.

\bibitem{geom-matrix-pol-2020}
Andrii Dmytryshyn, Stefan Johansson, Bo~K\aa{}gstr\"{o}m, and Paul van Dooren.
\newblock Geometry of matrix polynomial spaces.
\newblock {\em Found Comput Math}, 20(3):423--450, 2020.

\bibitem{dopico2018block}
Froil{\'a}n~M Dopico, Piers~W Lawrence, Javier P{\'e}rez, and Paul Van~Dooren.
\newblock Block {Kronecker} linearizations of matrix polynomials and their
  backward errors.
\newblock {\em Numerische Mathematik}, 140(2):373--426, 2018.

\bibitem{KVADeig-arxiv}
Zlatko Drma{\v{c}} and Ivana {\v{S}}ain~Glibi{\'c}.
\newblock New numerical algorithm for deflation of infinite and zero
  eigenvalues and full solution of quadratic eigenvalue problems.
\newblock {\em ACM Transactions on Mathematical Software (TOMS)}, 46(4):1--32,
  2020.

\bibitem{drmac-bujanovic-2008}
Zlatko Drma\v{c} and Zvonimir Bujanovi\'{c}.
\newblock On the failure of rank revealing {QR} factorization software -- a
  case study.
\newblock {\em ACM Transactions on Mathematical Software (TOMS)}, 35(2):1--28,
  2008.

\bibitem{2019arXiv190507013D}
Zlatko {Drma{\v{c}}} and Ivana {{\v{S}}ain Glibi{\'c}}.
\newblock {An algorithm for the complete solution of the quartic eigenvalue
  problem}.
\newblock {\em arXiv e-prints}, page arXiv:1905.07013, May 2019.

\bibitem{fan2004normwise}
Hung-Yuan Fan, Wen-Wei Lin, and Paul Van~Dooren.
\newblock Normwise scaling of second order polynomial matrices.
\newblock {\em SIAM Journal on Matrix Analysis and Applications},
  26(1):252--256, 2004.

\bibitem{Gavin_2018}
Brendan Gavin, Agnieszka Miedlar, and Eric Polizzi.
\newblock {FEAST} eigensolver for nonlinear eigenvalue problems.
\newblock {\em Journal of Computational Science}, 27:107--117, jul 2018.

\bibitem{Golub-Klema-Stewart-Numerical_rank}
Gene Golub, Virginia Klema, and Gilbert~W Stewart.
\newblock Rank degeneracy and least squares problems.
\newblock Technical report, Computer Science Department, Stanford University,
  1976.

\bibitem{Hammarling:QUADEIG}
Sven Hammarling, Christopher~J. Munro, and Fran{\c{c}}oise Tisseur.
\newblock An algorithm for the complete solution of quadratic eigenvalue
  problems.
\newblock {\em ACM Transactions on Mathematical Software (TOMS)}, 39(3):1--19,
  2013.

\bibitem{higham2007backward}
Nicholas~J. Higham, Ren-Cang Li, and Fran{\c{c}}oise Tisseur.
\newblock Backward error of polynomial eigenproblems solved by linearization.
\newblock {\em SIAM Journal on Matrix Analysis and Applications},
  29(4):1218--1241, 2008.

\bibitem{JOHANSSON20131062}
Stefan Johansson, Bo~K\aa{}gstr\"{o}m, and Paul {Van Dooren}.
\newblock Stratification of full rank polynomial matrices.
\newblock {\em Linear Algebra and its Applications}, 439(4):1062--1090, 2013.
\newblock 17th Conference of the International Linear Algebra Society,
  Braunschweig, Germany, August 2011.

\bibitem{Mastronardi-VanDooren-quad-back-err-ETNA}
Nicola Mastronardi and Paul Van~Dooren.
\newblock Revisiting the stability of computing the roots of a quadratic
  polynomial.
\newblock {\em CoRR}, abs/1409.8072, 2014.

\bibitem{descriptor-sys-Mehrmann-2006}
Volker Mehrmann and Tatjana Stykel.
\newblock Descriptor systems: A general mathematical framework for modelling,
  simulation and control ({Deskriptorsysteme}: Ein allgemeines mathematisches
  {Konzept} f\"{u}r {Modellierung}, {Simulation} und {Regelung}).
\newblock {\em Automatisierungstechnik}, 54(8):405 -- 415, 2006.

\bibitem{mehrmann-watkins2002}
Volker Mehrmann and David~S Watkins.
\newblock Polynomial eigenvalue problems with {Hamiltonian} structure.
\newblock {\em Electronic Transactions on Numerical Analysis}, 13:106--118,
  2002.

\bibitem{SAKURAI2003119}
Tetsuya Sakurai and Hiroshi Sugiura.
\newblock A projection method for generalized eigenvalue problems using
  numerical integration.
\newblock {\em Journal of Computational and Applied Mathematics}, 159(1):119 --
  128, 2003.
\newblock 6th Japan-China Joint Seminar on Numerical Mathematics; In Search for
  the Frontier of Computational and Applied Mathematics toward the 21st
  Century.

\bibitem{SONG2020109871}
Jung~Heon Song, Matthias Maier, and Mitchell Luskin.
\newblock Nonlinear eigenvalue problems for coupled {Helmholtz} equations
  modeling gradient-index graphene waveguides.
\newblock {\em Journal of Computational Physics}, 423:109871, 2020.

\bibitem{Stowell-waveguide-2009}
David Stowell and Johannes Tausch.
\newblock Variational formulation for guided and leaky modes in multilayer
  dielectric waveguides.
\newblock {\em Communications in Computational Physics}, 7(3):564, 01 2010.

\bibitem{tisseur2000backward}
Fran{\c{c}}oise Tisseur.
\newblock Backward error and condition of polynomial eigenvalue problems.
\newblock {\em Linear Algebra and its Applications}, 309(1-3):339--361, 2000.

\bibitem{tisseur2020min}
Francoise Tisseur and Marc Van~Barel.
\newblock Min-max elementwise backward error for roots of polynomials and a
  corresponding backward stable root finder.
\newblock {\em arXiv preprint arXiv:2001.05281}, 2020.

\bibitem{VANBAREL2018186}
Marc {Van Barel} and Françoise Tisseur.
\newblock Polynomial eigenvalue solver based on tropically scaled {Lagrange}
  linearization.
\newblock {\em Linear Algebra and its Applications}, 542:186--208, 2018.
\newblock Proceedings of the 20th ILAS Conference, Leuven, Belgium 2016.

\bibitem{van1979computation}
Paul Van~Dooren.
\newblock The computation of {Kronecker's} canonical form of a singular pencil.
\newblock {\em Linear Algebra and its Applications}, 27:103--140, 1979.

\bibitem{van1983eigenstructure}
Paul Van~Dooren and Patrick Dewilde.
\newblock The eigenstructure of an arbitrary polynomial matrix: computational
  aspects.
\newblock {\em Linear Algebra and its Applications}, 50:545--579, 1983.

\bibitem{doi:10.1063/1.4819209}
Istvan~A. Veres, Thomas Berer, and Osamu Matsuda.
\newblock Complex band structures of two dimensional phononic crystals:
  {Analysis} by the finite element method.
\newblock {\em Journal of Applied Physics}, 114(8):083519, 2013.

\bibitem{VIEIRA2014575}
R.F. Vieira, F.B. Virtuoso, and E.B.R. Pereira.
\newblock A higher order model for thin--walled structures with deformable
  cross--sections.
\newblock {\em International Journal of Solids and Structures}, 51(3--4):575 --
  598, 2014.

\bibitem{ISG-thesis-2018}
I.~\v{S}ain Glibi\'{c}.
\newblock {\em Robust numerical methods for nonlinear eigenvalue problems}.
\newblock PhD thesis, University of Zagreb, Faculty of Science, Department of
  Mathematics, 12 2018.

\bibitem{watkins2000performance}
David~S. Watkins.
\newblock Performance of the {QZ} algorithm in the presence of infinite
  eigenvalues.
\newblock {\em SIAM Journal on Matrix Analysis and Applications},
  22(2):364--375, 2000.

\bibitem{yong-phononic-plate-2014}
Yong Xiao, Jihong Wen, Lingzhi Huang, and Xisen Wen.
\newblock Analysis and experimental realization of locally resonant phononic
  plates carrying a periodic array of beam--like resonators.
\newblock {\em Journal of Physics D: Applied Physics}, 47:045307, 01 2013.

\bibitem{Shinnosuke-Yokota2013}
Shinnosuke Yokota and Tetsuya Sakurai.
\newblock A projection method for nonlinear eigenvalue problems using contour
  integrals.
\newblock {\em JSIAM Letters}, 5:41--44, 2013.

\bibitem{zeng-su-quadeig-bs}
Linghui Zeng and Yangfeng Su.
\newblock A backward stable algorithm for quadratic eigenvalue problems.
\newblock {\em {SIAM Journal on Matrix Analysis and Applications}},
  35(2):499--516, 2014.

\bibitem{Zha-triplets-RSVD}
H.~Zha.
\newblock The restricted singular value decomposition of matrix triplets.
\newblock {\em SIAM {Journal on Matrix Analysis and Applications}},
  12(1):172--194, 1991.

\bibitem{Zhang-vision-4608140}
B.~{Zhang} and Y.~F. {Li}.
\newblock A method for calibrating the central catadioptric camera via
  homographic matrix.
\newblock In {\em 2008 International Conference on Information and Automation},
  pages 972--977, June 2008.

\end{thebibliography}
\bibliographystyle{plain} 


\end{document}